\documentclass[a4paper,11pt]{amsart}

\usepackage{amsmath, amsfonts, amssymb, amsthm, amscd}
\usepackage{graphicx}
\usepackage{perpage}
\usepackage{url}
\usepackage{color}
\usepackage[T1]{fontenc}
\usepackage{hyperref}
\usepackage[a4paper,scale={0.72,0.74},marginratio={1:1},footskip=7mm,headsep=10mm]{geometry}

\numberwithin{equation}{section}

\newtheorem{theorem}{Theorem}[section]
\newtheorem{lemma}[theorem]{Lemma}
\newtheorem{proposition}[theorem]{Proposition}

\newtheorem{remark}[theorem]{Remark}
\newtheorem{definition}[theorem]{Definition}

\newcommand{\bbE}{{\ensuremath{\mathbb E}} }

\newcommand{\bbN}{{\ensuremath{\mathbb N}} }

\newcommand{\bbP}{{\ensuremath{\mathbb P}} }

\newcommand{\bbR}{{\ensuremath{\mathbb R}} }

\newcommand{\bbZ}{{\ensuremath{\mathbb Z}} }

\newcommand{\bbVar}{{\ensuremath{\mathbb{V}\mathrm{ar}}} }

\newcommand{\cA}{{\ensuremath{\mathcal A}} }
\newcommand{\cB}{{\ensuremath{\mathcal B}} }
\newcommand{\cC}{{\ensuremath{\mathcal C}} }

\newcommand{\cF}{{\ensuremath{\mathcal F}} }
\newcommand{\cG}{{\ensuremath{\mathcal G}} }
\newcommand{\cH}{{\ensuremath{\mathcal H}} }

\newcommand{\cK}{{\ensuremath{\mathcal K}} }
\newcommand{\cL}{{\ensuremath{\mathcal L}} }
\newcommand{\cM}{{\ensuremath{\mathcal M}} }
\newcommand{\cN}{{\ensuremath{\mathcal N}} }

\newcommand{\cP}{{\ensuremath{\mathcal P}} }
\newcommand{\cQ}{{\ensuremath{\mathcal Q}} }

\newcommand{\cS}{{\ensuremath{\mathcal S}} }
\newcommand{\cT}{{\ensuremath{\mathcal T}} }

\newcommand{\cV}{{\ensuremath{\mathcal V}} }

\newcommand{\cZ}{{\ensuremath{\mathcal Z}} }

\newcommand{\ga}{\alpha}
\newcommand{\gb}{\beta}

\newcommand{\gd}{\delta}
\newcommand{\gD}{\Delta}
\newcommand{\gep}{\varepsilon}

\newcommand{\gl}{\lambda}
\newcommand{\gL}{\Lambda}

\newcommand{\gs}{\sigma}

\newcommand{\go}{\omega}
\newcommand{\gO}{\Omega}

\renewcommand{\tilde}{\widetilde}          
\DeclareMathSymbol{\leqslant}{\mathalpha}{AMSa}{"36} 
\DeclareMathSymbol{\geqslant}{\mathalpha}{AMSa}{"3E} 
\DeclareMathSymbol{\eset}{\mathalpha}{AMSb}{"3F}     
\newcommand{\dd}{\text{\rm d}}             
\newcommand{\suptwo}[2]{\sup_{\substack{#1 \\ #2}}} 
\newcommand{\sumtwo}[2]{\sum_{\substack{#1 \\ #2}}} 

\newcommand{\R}{\mathbb{R}}

\newcommand{\Z}{\mathbb{Z}}
\newcommand{\N}{\mathbb{N}}

\newcommand\bP{\ensuremath{\mathrm{P}}}
\newcommand\bE{\ensuremath{\mathrm{E}}}

\newcommand{\intr}{{\rm int}}

\newcommand{\ind}{{\sf 1}}

\newcommand{\spr}{{\rm sp}}

\newcommand{\epi}[1]{\stackrel{\mathrm{(epi)}}{#1}}

\renewcommand{\epsilon}{\varepsilon}
\renewcommand{\phi}{\varphi}

\newcommand{\tW}{\mathrm{\widetilde{W}}}

\newenvironment{myenumerate}{
\renewcommand{\theenumi}{\arabic{enumi}}
\renewcommand{\labelenumi}{{\rm(\theenumi)}}
\begin{list}{\labelenumi}
{
\setlength{\itemsep}{0.4em}
\setlength{\topsep}{0.5em}
\setlength\leftmargin{2.45em}
\setlength\labelwidth{2.05em}
\setlength{\labelsep}{0.4em}
\usecounter{enumi}
}
}
{\end{list}
}

\renewenvironment{enumerate}{
\begin{myenumerate}}
{\end{myenumerate}}

\newcommand{\Geo}{\mathrm{Geo}}

\begin{document}

\title{Annealed scaling for a charged polymer}

\author{F.\ Caravenna}
\address{Dipartimento di Matematica e Applicazioni,
Universit\`a degli Studi di Milano-Bicocca,
Via Cozzi 53, 20125 Milano,  Italy.}
\email{francesco.caravenna@unimib.it}

\author{F.\ den Hollander}
\address{Mathematical Institute, Leiden University, P.O.\ Box 9512,
2300 RA Leiden, The Netherlands.}
\email{denholla@math.leidenuniv.nl}

\author{N.\ P\'etr\'elis}
\address{Laboratoire de Math\'ematiques Jean Leray UMR 6629,
Universit\'e de Nantes, 2 Rue de la Houssini\`ere,
BP 92208, F-44322 Nantes Cedex 03, France.}
\email{nicolas.petrelis@univ-nantes.fr}

\author{J.\ Poisat}
\address{CEREMADE, Universit\'e Paris-Dauphine, PSL Research University, UMR 7534, 
Place du Mar\'echal de Lattre de Tassigny,
75775 Paris Cedex 16 - France.}
\email{poisat@ceremade.dauphine.fr}

\begin{abstract}
This paper studies an undirected polymer chain living on the one-dimensional integer lattice 
and carrying i.i.d.\ random charges. Each self-intersection of the polymer chain contributes 
to the interaction Hamiltonian an energy that is equal to the product of the charges of the 
two monomers that meet. The joint probability distribution for the polymer chain and the 
charges is given by the Gibbs distribution associated with the interaction Hamiltonian. 
The focus is on the \emph{annealed free energy} per monomer in the limit as the length 
of the polymer chain tends to infinity. 

We derive a \emph{spectral representation} for the free energy and use this to prove that there 
is a critical curve in the parameter plane of charge bias versus inverse temperature separating 
a \emph{ballistic phase} from a \emph{subballistic phase}. We show that the phase transition 
is \emph{first order}. We prove large deviation principles for the laws of the empirical speed 
and the empirical charge, and derive a spectral representation for the associated rate 
functions. Interestingly, in both phases both rate functions exhibit \emph{flat pieces}, which 
correspond to an inhomogeneous strategy for the polymer to realise a large deviation. The 
large deviation principles in turn lead to laws of large numbers and central limit theorems. 
We identify the scaling behaviour of the critical curve for small and for large charge bias. In 
addition, we identify the scaling behaviour of the free energy for small charge bias and small 
inverse temperature. Both are linked to an associated \emph{Sturm-Liouville eigenvalue 
problem}. 

A key tool in our analysis is the Ray-Knight formula for the local times of the one-dimensional 
simple random walk. This formula is exploited to derive a \emph{closed form expression} 
for the generating function of the annealed partition function, and for several related quantities. 
This expression in turn serves as the starting point for the derivation of the spectral 
representation for the free energy, and for the scaling theorems.

What happens for the \emph{quenched free energy} per monomer remains open. We state 
two modest results and raise a few questions.
\end{abstract}

\thanks{The research in this paper was supported by ERC Advanced Grant 267356-VARIS.
JP held a postdoc-position at the Mathematical Institute of Leiden University from September
2012 until August 2014, FC and NP made extended visits in the same period. FC acknowledges 
the support of GNAMPA-INdAM. The authors also thank the University of Nantes and the University 
of Milano-Bicocca for hospitality.}

\date{\today}

\maketitle

\newpage

\tableofcontents


\section{Introduction}
\label{s:intro}


\subsection{Motivation}
\label{ss:motivation}

DNA and proteins are polyelectrolytes whose monomers are in a charged state 
that depends on the pH of the solution in which they are immersed. The charges 
may fluctuate in space (`quenched') and in time (`annealed').   

In this paper we consider the charged polymer chain introduced in Kantor 
and Kardar~\cite{KaKa91}. The polymer chain is modelled by the path of a 
simple random walk on $\Z^d$, $d \geq 1$. Each monomer in the polymer 
chain carries a \emph{random electric charge}, drawn in an i.i.d.\ fashion 
from $\R$. Each self-intersection of the polymer chain contributes an energy 
that is equal to the product of the charges of the two monomers that meet (i.e., 
a negative energy when the charges have opposite sign and a positive energy 
when the charges have the same sign). The polymer chain has a probability 
distribution on path space that is given by the \emph{Gibbs measure} associated 
with the energy. Our goal is to study the scaling properties of the polymer as its 
length tends to infinity.

Very little is known mathematically about the \emph{quenched} version of the 
model, where the charges are frozen. The two main questions of interest are:
\begin{itemize} 
\item[(1)]
Is the free energy self-averaging in the disorder?
\item[(2)] 
Is there a phase transition from a `collapsed phase' to an `extended phase' at 
some critical value of the temperature? 
\end{itemize}
We expect that the answer to (1) is yes and the answer to (2) is no. All we are able 
to show is the following (see Appendix~\ref{AppB}):
\begin{itemize}
\item[(3)] 
If the average charge is non-zero, then the number of different sites visited by the 
polymer is proportional to its length.
\item[(4)] 
In $d=1$, if the average charge is sufficiently positive or negative and the temperature 
is sufficiently low, then the polymer behaves ballistically. 
\end{itemize}
We expect that in any $d \geq 1$ the scaling of the polymer is similar to that of 
the \emph{self-avoiding walk} when the average charge is non-zero. We further 
expect that the polymer is subdiffusive when the average charge is zero. All these 
problems remains open.

In the present paper we focus on the \emph{annealed} version of the model, where 
the charges are averaged out. This version, which we study in $d=1$ only, is easier 
to deal with, yet turns out to exhibit a very rich scaling behavior. The answer to (2) is 
yes for the annealed model. We will obtain a detailed description of the phase transition 
curve separating a \emph{subballistic phase} from a \emph{ballistic phase}. Moreover, 
we show that the phase transition is first order, and show that the empirical speed and 
the empirical charge satisfy a law of large numbers, a central limit theorem, as well 
as a large deviation principle with a rate function that exhibits flat pieces. The latter 
corresponds to an inhomogeneous strategy for the polymer to realise a large deviation. 
We identify the scaling of the free energy in the limit of small average charge and small 
inverse temperature, which exhibits anomalous behaviour. 

A key tool in our analysis is the Ray-Knight formula for the local times of the one-dimensional 
simple random walk. This tool, which has been used extensively in the literature, is 
exploited in full throughout the paper in order to obtain the fine details of the phase 
diagram of the charged polymer. The Ray-Knight formula is no longer available in 
$d \geq 2$. In Berger, den Hollander and Poisat~\cite{BdHPpr} it is shown that the 
phase diagram is qualitatively similar, but no detailed description of the scaling behaviour 
in the two phases is obtained.

The outline of the paper is as follows. In Section~\ref{ss:model} we define the model.
In Section~\ref{ss:theoremsgeneral} we state six theorems with general properties 
and in Section~\ref{ss:theoremsasymptotics} three theorems with asymptotic properties. 
In Section~\ref{ss:discussion} we discuss these theorems. Proofs are given in 
Sections~\ref{freeenergy}--\ref{s:asymptotics}. Appendices~\ref{AppA}--\ref{AppC} 
contain a few technical computations, while Appendix~\ref{AppD} states two modest 
results for the quenched version of the model. 


\subsection{Model and assumptions}
\label{ss:model}

Throughout the paper we use the notation $\N=\{1,2,\dots\}$ and $\N_0=\N\cup\{0\}$.
  
Let $S=(S_i)_{i\in\N_0}$ be a simple random walk on $\Z^d$, $d \geq 1$, i.e.,  $S_0=0$ and
$S_i = \sum_{j=1}^i X_j$, $i\in\N$, with $X=(X_j)_{j\in\N}$ i.i.d.\ random variables such that 
$\bP(X_1=x)=\tfrac{1}{2d}$ for $x\in \Z^d$ with $\|x\|=1$ and zero otherwise ($\|\cdot\|$ 
denotes the lattice norm). The path $S$ models the configuration of the polymer chain, i.e., 
$S_i$ is the location of monomer $i$. We use the letters $\bP$ and $\bE$ for probability and 
expectation with respect to $S$. 

Let $\omega=(\omega_i)_{i\in\N}$ be i.i.d.\ random variables taking values in $\R$. The 
sequence $\omega$ models the electric charges along the polymer chain, i.e., $\omega_i$ 
is the charge of monomer $i$ (see Fig.~\ref{fig-charpol}). We use the letters $\bbP$ and 
$\bbE$ for probability and expectation with respect to $\omega$. Throughout the paper 
we assume that 
\begin{equation}
\label{Mdeltacond}
M(t) = \bbE(e^{t\omega_1}) < \infty \qquad \forall\,t \in \R.
\end{equation} 
Without loss of generality we may take (see \eqref{eq:def.polmeasure}--\eqref{eq:def.hamiltonian} 
below) 
\begin{equation}
\label{eq:omegacond}
\bbE(\omega_1)=0, \qquad \bbVar(\omega_1)=1.
\end{equation}
To allow for biased charges, we use a tilting parameter $\delta \in \R$ and write $\bbP^\delta$ 
for the i.i.d.\ law of $\omega$ with marginal
\begin{equation}
\label{eq:Pdeltadef}
\bbP^\delta(\dd \omega_1) = \frac{e^{\delta \omega_1}\,\bbP(\dd \omega_1)}{M(\delta)}.
\end{equation} 
Note that $\bbE^\delta(\omega_1)=M'(\delta)/M(\delta)$. In what follows we may, without loss 
of generality, take $\delta \in [0,\infty)$.

\medskip\noindent
{\bf Example:} The special case where the charges are $+1$ with probability $p$ and $-1$ 
with probability $1-p$ for some $p\in (0,1)$ corresponds to $\bbP=[\tfrac12(\delta_{-1}
+\delta_{+1})]^{\otimes\N}$ and $\delta=\tfrac12\log(\frac{p}{1-p})$.
  
\medskip 
Let $\Pi$ denote the set of nearest-neighbor paths starting at $0$. Given $n\in\N$, we associate with 
each $(\omega,S) \in \R^\N \times \Pi$ an energy given by the Hamiltonian (see Fig.~\ref{fig-charpol})
\begin{equation}
\label{eq:def.ham}
H_n^\omega(S) = \sum_{1 \leq i < j \leq n} \omega_i \omega_j\, \ind_{\{S_i=S_j\}}.
\end{equation} 
Let $\beta$ denote the inverse temperature. Throughout the sequel the relevant space for the 
pair of parameters $(\delta,\beta)$ is the quadrant
\begin{equation}
\cQ = [0,\infty) \times (0,\infty).
\end{equation}
Given $(\delta,\beta) \in \cQ$, the \emph{annealed polymer measure of length} $n$ is the
Gibbs measure $\bbP_n^{\delta,\beta}$ defined as
\begin{equation}
\label{eq:def.polmeasure}
\frac{\dd\bbP_n^{\delta,\beta}}{\dd(\bbP^\delta \times \bP)}(\omega,S) 
= \frac{1}{\bbZ_n^{\delta,\beta}}\,e^{-\beta H_n^\omega(S)},
\qquad (\omega,S) \in \R^\N \times \Pi,
\end{equation} 
where
\begin{equation}
\label{eq:def.annpartfunc}
\Z_n^{\delta,\beta} = (\bbE^\delta \times \bE)\left[e^{-\beta H_n^\omega(S)}\right]
\end{equation}
is the \emph{annealed partition function of length} $n$. The measure $\bbP_n^{\delta,\beta}$ 
is the joint probability distribution for the polymer chain and the charges at charge bias $\delta$
and inverse temperature $\beta$ when the polymer chain has length $n$.

\begin{figure}[htbp]
\vspace{0.5cm}
\begin{center}
\includegraphics[scale = 0.4]{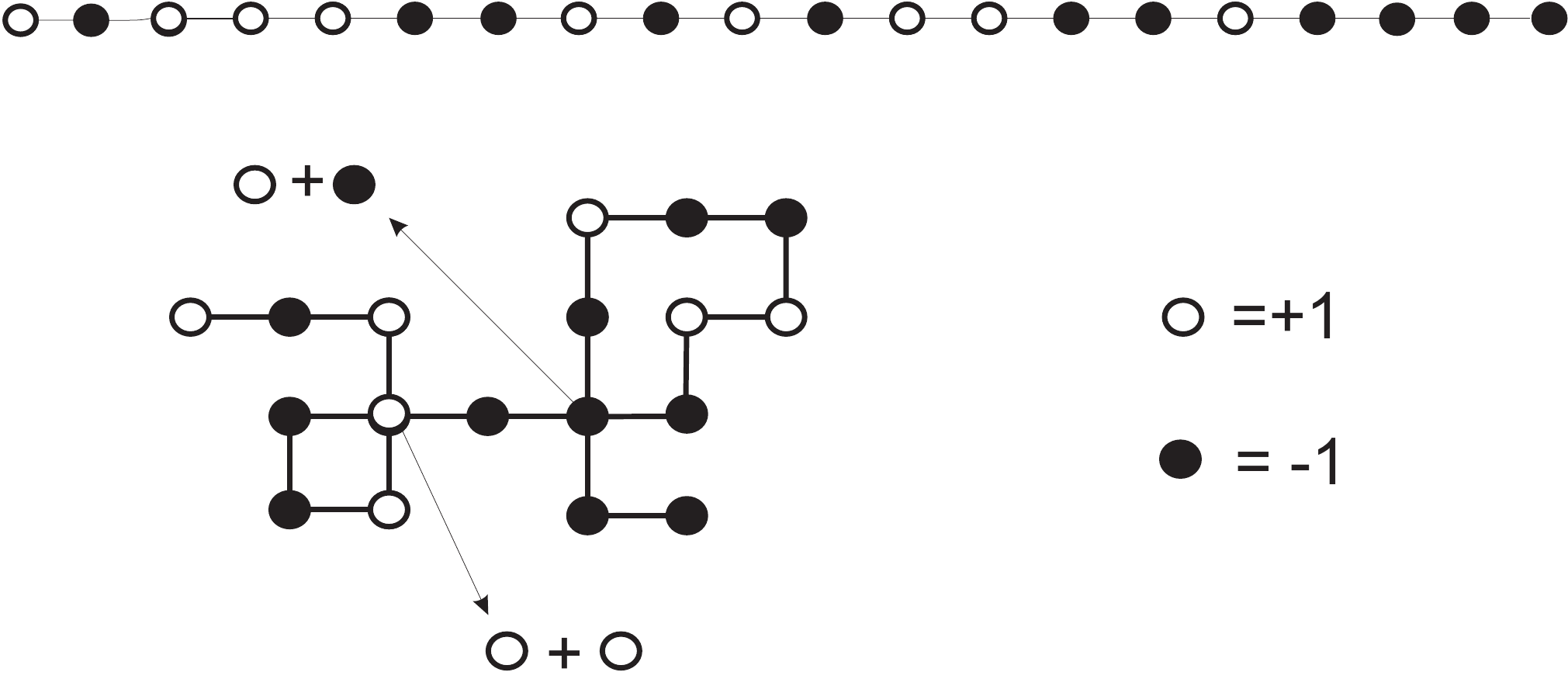}
\end{center}
\caption{\emph{Top:} A polymer chain carrying $(\pm 1)$-valued random charges. \emph{Bottom:} 
The path may or may not be self-avoiding. The charges only interact at self-intersections.}
\label{fig-charpol}
\end{figure}

In what follows, instead of \eqref{eq:def.ham} we will work with the Hamiltonian 
\begin{equation}
\label{eq:def.hamiltonian}
H_n^\omega(S) 
= \sum_{1 \leq i,j \leq n} \omega_i\omega_j\,\ind_{\{S_i=S_j\}}
= \sum_{x\in \Z^d} \left(\sum_{i=1}^n \omega_i\, \ind_{\{S_i=x\}}\right)^2.
\end{equation} 
The sum under the square is the local time of $S$ at site $x$ weighted by the charges that are 
encountered in $\omega$. The change from \eqref{eq:def.ham} to \eqref{eq:def.hamiltonian} 
amounts to replacing $\beta$ by $2\beta$ and adding a charge bias (see Section~\ref{ss:loctimerepr} 
for more details).


\subsection{Theorems: general properties}
\label{ss:theoremsgeneral}

Let $Q(i,j)$ be the probability matrix defined by
\begin{equation}
\label{eq:def.Qij}
Q(i,j) = \begin{cases}
\ind_{\{j = 0\}}, & \text{if } i=0, \ j \in \N_0, \\
\rule{0pt}{2em}
\displaystyle \binom{i+j-1}{i-1}
\left( \frac12 \right)^{i+j},
& \text{if }
i \in \N, \ j \in \N_0,
\end{cases}
\end{equation}
which is the transition kernel of a critical Galton-Watson branching process with a geometric 
offspring distribution (of parameter $\frac{1}{2}$). For $(\delta,\beta) \in \cQ$, let 
$G^*_{\delta,\beta}$ be the function defined by
\begin{equation}
\label{eq:rel.G.star}
G^*_{\delta,\beta}(\ell) 
= \log \bbE\left[e^{\delta \Omega_\ell-\beta \Omega_\ell^2}\right]
\quad \text{ with} \quad \Omega_\ell = \sum_{k=1}^\ell \omega_k, \qquad \ell \in \N_0.
\end{equation}
($\Omega_0 = 0$.) For $(\mu,\delta,\beta) \in [0,\infty)\times \cQ$, define the $\N_0 \times \N_0$ 
matrices ${A}_{\mu,\delta,\beta}$ and $\tilde{A}_{\mu,\delta,\beta}$ by
\begin{align}
\label{eq:def.Abeta.p.r}
A_{\mu,\delta,\beta}(i,j) 
&= e^{- \mu(i+j+1) + G^*_{\delta,\beta}(i+j+1)}\, Q(i+1,j),
\qquad i,j\in\N_0, \\
\label{eq:def.Abeta.p.r2}
\tilde{A}_{\mu,\delta,\beta}(i,j) 
&=
\begin{cases}
0, & \text{ if } i = 0, \, j \in \N_0, \\
A_{\mu,\delta,\beta}(i-1,j), & \text{ if } i\in \N, \, j\in\N_0.
\end{cases}
\end{align} 
Note that ${A}_{\mu,\delta,\beta}$ is symmetric while $\tilde{A}_{\mu,\delta,\beta}$ is not.

Let $\lambda_{\delta,\beta}(\mu)$ and $\tilde{\lambda}_{\delta,\beta}(\mu)$ be the spectral 
radius of $A_{\mu,\delta,\beta}$, respectively, $\tilde{A}_{\mu,\delta,\beta}$ in $\ell^2(\N_0)$.
We will see in Section~\ref{subsec:eigenvalue} that, for every $(\delta,\beta) \in \cQ$, both 
$\mu \mapsto \lambda_{\delta,\beta}(\mu)$ and $\mu \mapsto \tilde{\lambda}_{\delta,\beta}
(\mu)$ are continuous, strictly decreasing and log-convex on $[0,\infty)$, tend to zero at infinity, 
and satisfy $\tilde{\lambda}_{\delta,\beta}(\mu) < \lambda_{\delta,\beta}(\mu)$ for all $\mu 
\in [0,\infty)$. Let
\begin{equation}
\label{mudeltabetadef}
\begin{array}{ll}
&\bullet\,\, \mu(\delta,\beta) \text{ be the unique solution of the equation } 
\lambda_{\delta,\beta}(\mu)=1\\ 
&\text{when it exists and } \mu(\delta,\beta) = 0 \text{ otherwise},\\
&\bullet\,\, \tilde{\mu}(\delta,\beta) \text{ be the unique solution of the equation } 
\tilde{\lambda}_{\delta,\beta}(\mu)=1\\ 
&\text{when it exists and } \tilde{\mu}(\delta,\beta) = 0 \text{ otherwise},
\end{array}
\end{equation}
which satisfy $\tilde{\mu}(\delta,\beta) \leq \mu(\delta,\beta)$, with strict inequality as soon as 
$\mu(\delta,\beta)>0$. We will also see that, for every $(\delta,\beta) \in \cQ$, $\mu \mapsto 
\lambda_{\delta,\beta}(\mu)$ is analytic and strictly log-convex on $(0,\infty)$, and has a finite 
strictly negative right-slope at $0$ (see Fig.~\ref{fig-spectral}).

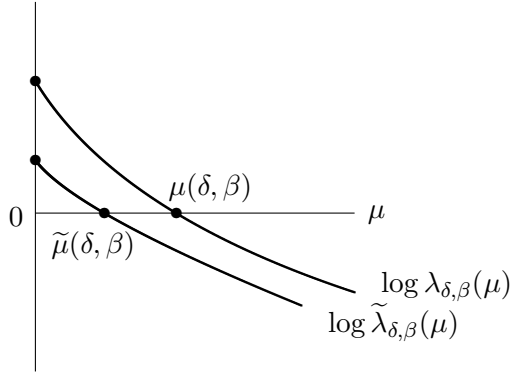
\begin{figure}[htbp]
\vspace{-.5cm}
\begin{center}
\setlength{\unitlength}{0.35cm}
\begin{picture}(12,12)(0,-3)
\put(0,0){\line(1,0){12}}
\put(0,-6){\line(0,1){14}}
{\thicklines
\qbezier(0,5)(3,0)(12,-3)
\qbezier(0,2)(1.5,0)(10,-3.5)
}
\put(-1,-.5){$0$}
\put(12.5,-0.2){$\mu$}
\put(5,.7){$\mu(\delta,\beta)$}
\put(.6,-1.5){$\tilde{\mu}(\delta,\beta)$}
\put(13,-3){$\log\lambda_{\delta,\beta}(\mu)$}
\put(11,-4.5){$\log\tilde{\lambda}_{\delta,\beta}(\mu)$}
\put(0,5){\circle*{.4}}
\put(0,2){\circle*{.4}}
\put(5.3,0){\circle*{.4}}
\put(2.6,0){\circle*{.4}}
\end{picture}
\end{center}
\vspace{1cm}
\caption{\small Qualitative plot of $\mu \mapsto \log \lambda_{\delta,\beta}(\mu)$ (top curve) 
and $\mu \mapsto \log \tilde{\lambda}_{\delta,\beta}(\mu)$ (bottom curve) for fixed $(\delta,
\beta) \in \cQ$. Only the case $\lambda_{\delta,\beta}(0) > \tilde{\lambda}_{\delta,\beta}(0)>1$ 
is shown. The interior of the ballistic phase $\intr(\cB)$ corresponds to $\lambda_{\delta,\beta}
(0) > 1$, the subballistic phase $\cS$ corresponds to $\lambda_{\delta,\beta}(0) < 1$, the critical 
curve corresponds to $\lambda_{\delta,\beta}(0)=1$ (see \eqref{BSdef}).}
\label{fig-spectral}
\end{figure}

We begin with a \emph{spectral representation} for the annealed free energy.
Abbreviate
\begin{equation}
\label{eq:fdeltadef}
f(\delta) = -\log M(\delta) \in (-\infty,0].
\end{equation}

\begin{theorem}
\label{thm:free.energy}
For all $(\delta,\beta) \in \cQ$, the annealed free energy per monomer
\begin{equation}
\label{eq:Fdef}
F(\delta,\beta) = \lim_{n\to\infty} \frac{1}{n}\,\log \bbZ_n^{\delta,\beta}
\end{equation}
exists, takes values in $(-\infty,0]$, and satisfies the inequality 
\begin{equation}
\label{eq:Fineq}
F(\delta,\beta) \geq f(\delta).
\end{equation} 
Moreover, the excess free energy 
\begin{equation}
F^*(\delta,\beta) = F(\delta,\beta) - f(\delta)
\end{equation}
is convex in $(\delta,\beta)$ and has the spectral representation 
\begin{equation}\label{varc}
F^*(\delta,\beta) = \mu(\delta,\beta).
\end{equation}
\end{theorem}

The inequality in \eqref{eq:Fineq} leads us to define two phases:
\begin{equation}
\label{phasesdef}
\begin{aligned}
\cQ^{>} &= \big\{(\delta,\beta) \in \cQ \colon\,F^*(\delta,\beta) > 0\big\},\\
\cQ^{=} &= \{(\delta,\beta) \in \cQ \colon\, F^*(\delta,\beta) = 0\big\}.
\end{aligned}
\end{equation}
We next show that these phases are separated by a single \emph{critical curve} (see 
Fig.~{\rm \ref{fig-critcurve}}) and that there are no further subphases.

\begin{figure}[htbp]
\begin{center}
\setlength{\unitlength}{0.35cm}
\begin{picture}(12,12)(0,-2)
\put(0,0){\line(12,0){12}}
\put(0,0){\line(0,8){8}}
{\thicklines
\qbezier(0,0)(5,0.5)(9,6.5)
}
\put(-1,-.5){$0$}
\put(12.5,-0.2){$\delta$}
\put(-0.3,8.5){$\beta$}
\put(8.5,2){$\cQ^{>}$}
\put(3,4){$\cQ^{=}$}
\put(8.5,7.3){$\beta_c(\delta)$}
\end{picture}
\end{center}
\vspace{0cm}
\caption{\small Qualitative plot of the critical curve $\delta \mapsto \beta_c(\delta)$ where 
the \emph{excess free energy} $F^*(\delta,\beta)$ changes from being zero to being strictly 
positive (see \eqref{phasesdef}). The critical curve is part of $\cQ^{=}$.}
\label{fig-critcurve}
\end{figure}
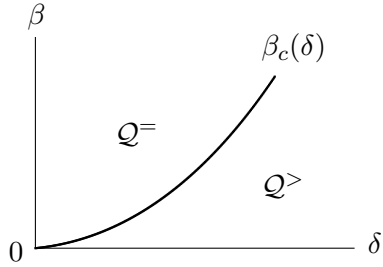
 
\begin{theorem}
\label{thm:critcurve}
There exists a critical curve $\delta \mapsto \beta_c(\delta)$  such that
\begin{equation}
\begin{aligned}
\cQ^{>} &= \big\{(\delta,\beta) \in \cQ \colon\,0 < \beta < \beta_c(\delta)\big\},\\
\cQ^{=} &= \big\{(\delta,\beta) \in \cQ \colon\, \beta \geq \beta_c(\delta)\big\}. 
\end{aligned}
\end{equation}
For every $\delta\in [0,\infty)$, $\beta_c(\delta)$ is the unique solution of the equation
$\lambda_{\delta,\beta}(0) = 1$. Moreover, $\delta \mapsto \beta_c(\delta)$ is continuous, 
strictly increasing and convex on $[0,\infty)$, analytic on $(0,\infty)$, and satisfies 
$\beta_c(0)=0$. In addition, $(\delta,\beta) \mapsto F^*(\delta,\beta)$ is analytic on
$\cQ^{>}$.
\end{theorem} 

Let
\begin{equation}
\label{BSdef}
\cB = \big\{(\delta,\beta) \in \cQ\colon\,0<\beta\leq\beta_c(\delta)\big\},
\qquad \cS = \cQ \backslash \cB.
\end{equation}
The set $\cB$ will be referred to as the \emph{ballistic phase}, the set $\cS$ as the 
\emph{subballistic phase}, for reasons we explain next. Namely, we proceed by stating 
a \emph{law of large numbers} for the empirical speed $n^{-1}S_n$ and the empirical 
charge $n^{-1} \Omega_n$, respectively, with 
\begin{equation}
S_n = \sum_{i=1}^n X_i, \qquad \Omega_n = \sum_{i=1}^n \omega_i.
\end{equation} 
In the statement below the condition $S_n > 0$ is put in to choose a direction for the 
endpoint of the polymer chain. 

\begin{theorem}
\label{thm:speed}
For every $(\delta,\beta) \in \cQ$ there exists a $v(\delta,\beta) \in [0,1]$ such that
\begin{equation}
\label{LLNspeed}
\lim_{n\to\infty} \bbP_n^{\delta,\beta}\Big(\big|n^{-1}S_n-v(\delta,\beta)\big|>\gep
\,\Big|\, S_n > 0 \Big) = 0 \qquad \forall\,\gep>0,
\end{equation}
where
\begin{equation}
\label{vident}
v(\delta,\beta) \left\{\begin{array}{ll}
> 0, &(\delta,\beta) \in \cB,\\
= 0, &(\delta,\beta) \in \cS.
\end{array}
\right.
\end{equation} 
For every $(\delta,\beta) \in \cB$,
\begin{equation}
\label{eq:speedspec}
\frac{1}{v(\delta,\beta)} 
= \left[-\frac{\partial}{\partial\mu} \log \lambda_{\delta,\beta}(\mu)\right]_{\mu=\mu(\delta,\beta)}
= \left[-\frac{\partial}{\partial\mu}\lambda_{\delta,\beta}(\mu)\right]_{\mu=\mu(\delta,\beta)}.
\end{equation}
(Take the right-derivative when $\mu(\delta,\beta)=0$; see Fig.~{\rm \ref{fig-spectral}}.)
Moreover, $(\delta,\beta) \mapsto v(\delta,\beta)$ is analytic on $\intr(\cB)$. 
\end{theorem}

\begin{theorem}
\label{thm:charge}
For every $(\delta,\beta) \in \cQ$, there exists a $\rho(\delta,\beta) \in [0,\infty)$ such that
\begin{equation}
\lim_{n\to\infty} \bbP_n^{\delta,\beta}\left(\left| n^{-1} \Omega_n 
- \rho(\delta,\beta) \right| > \epsilon\right) = 0
\qquad \forall\,\epsilon>0,
\end{equation}
where
\begin{equation}
\label{rhoident}
\rho(\delta,\beta)
\left\{\begin{array}{ll}
>0, &(\delta,\beta) \in \cB,\\
=0, &(\delta,\beta) \in \cS. 
\end{array}
\right.
\end{equation}
For every $(\delta,\beta) \in \cB$,
\begin{equation}
\label{eq:chargespec}
\rho(\delta,\beta)
= \left[\frac{\frac{\partial}{\partial\delta} \log \lambda_{\delta,\beta}(\mu)}
{-\frac{\partial}{\partial\mu} \log \lambda_{\delta,\beta}(\mu)}\right]_{\mu=\mu(\delta,\beta)}
= \frac{\partial}{\partial\delta}\,\mu\big(\delta,\beta\big).
\end{equation}
Moreover, $(\delta,\beta) \mapsto \rho(\delta,\beta)$ is analytic on $\intr(\cB)$.
\end{theorem}

\begin{remark}\rm
Since $\gd \mapsto (\gd, \gb_c(\gd))$ lies in the ballistic phase
$\cB$ (recall \eqref{BSdef}),
Theorems~\ref{thm:speed}--\ref{thm:charge} imply that $(\gd,\gb) \mapsto v(\gd,\gb)$ 
and $(\gd,\gb) \mapsto\rho(\gd,\gb)$ are \emph{discontinuous at criticality}. This
means that the \emph{phase transition is first order}.
See Fig.~\ref{fig-speedcharge} 
below for numerical plots of $(\delta,\beta) \mapsto v(\delta,\beta)$ and $(\delta,\beta) 
\mapsto \rho(\delta,\beta)$. 
\end{remark}

In fact, \emph{large deviation principles} holds for the laws of the empirical speed and the 
empirical charge. Let
\begin{equation}
\label{mugammaid}
\begin{array}{ll} 
&\mu(\delta,\beta,\gamma) \text{ be the solution of the equation } \lambda_{\delta,\beta}(\mu)
= e^{-\gamma}\\
&\text{when it exists and } \mu(\delta,\beta,\gamma)=0 \text{ otherwise}
\end{array}
\end{equation}
and note that $\mu(\delta,\beta,0)=\mu(\delta,\beta)$. 

\begin{theorem}
\label{thm:jointLDP}
For every $(\delta,\beta) \in \cQ$:\\ 
{\rm (1)} The sequence $(n^{-1}S_n)_{n\in\N}$ conditionally on $\{S_n>0\}_{n\in\N}$ satisfies 
the large deviation principle on $[0,\infty)$ with rate function $I^v_{\delta,\beta}$ given by
\begin{equation}
\label{Ivdef}
I^v_{\delta,\beta}(\theta) = \mu(\delta,\beta) + \sup_{\gamma \in \R} 
\big[ \theta\gamma - \{\mu(\delta,\beta,\gamma) \vee \tilde{\mu}(\delta,\beta)\}\big],
\qquad \theta \in [0,\infty).
\end{equation}
{\rm (2)} The sequence $(n^{-1}\Omega_n)_{n\in\N}$ satisfies the large deviation principle on 
$[0,\infty)$ with rate function $I^\rho_{\delta,\beta}$ given by
\begin{equation}
\label{Irhodef}
I^\rho_{\delta,\beta}(\theta') = \mu(\delta,\beta) + \sup_{\gamma' \in \R} 
\big[ \theta'\gamma' - \mu(\delta+\gamma',\beta)\big], \qquad \theta' \in [0,\infty).
\end{equation}
(The large deviation principle on $(-\infty,0)$ is obtained from that on $(0,\infty)$ after reflection 
of the charge distribution.)
\end{theorem}

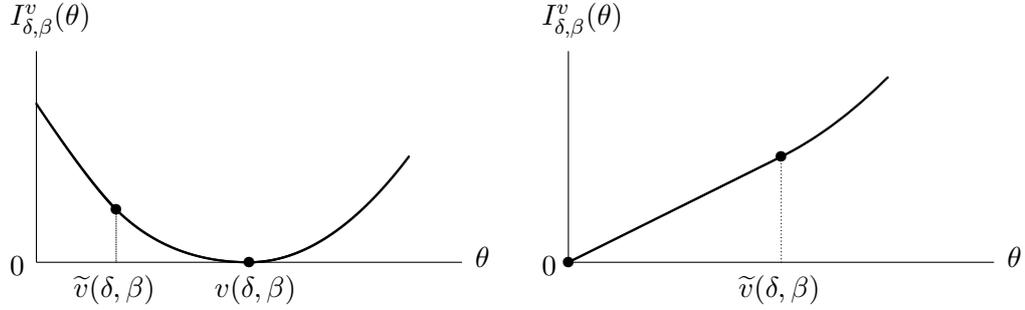
\begin{figure}[htbp]
\vspace{1.5cm}
\begin{center}
\setlength{\unitlength}{0.35cm}
\begin{picture}(12,8)(10,-2)
\put(0,0){\line(16,0){16}}
\put(0,0){\line(0,8){8}}
{\thicklines
\qbezier(0,6)(2,3)(3,2)
\qbezier(3,2)(5,0)(8,0)
\qbezier(8,0)(11,0)(14,4)
}
\qbezier[30](3,2)(3,1)(3,0)
\put(-1,-.5){$0$}
\put(16.5,-0.2){$\theta$}
\put(8,0){\circle*{.4}}
\put(3,2){\circle*{.4}}
\put(-1,9){$I^v_{\delta,\beta}(\theta)$}
\put(6.7,-1.3){$v(\delta,\beta)$}
\put(1.4,-1.3){$\tilde{v}(\delta,\beta)$}
\put(20,0){\line(16,0){16}}
\put(20,0){\line(0,8){8}}
{\thicklines
\qbezier(20,0)(24,2)(28,4)
\qbezier(28,4)(30,5)(32,7)
}
\qbezier[30](28,0)(28,2)(28,4)
\put(19,-.5){$0$}
\put(36.5,-0.2){$\theta$}
\put(20,0){\circle*{.4}}
\put(28,4){\circle*{.4}}
\put(19,9){$I^v_{\delta,\beta}(\theta)$}
\put(26.4,-1.3){$\tilde{v}(\delta,\beta)$}
\end{picture}
\end{center}
\vspace{0cm}
\caption{\small Qualitative plot of $\theta \mapsto I^v_{\delta,\beta}(\theta)$ 
for $(\delta,\beta) \in \intr(\cB)$ (left) and $(\delta,\beta) \in \cS$ (right). The 
slope of the flat piece on the left and on the right equals $-\log \lambda_{\delta,\beta}
(\tilde{\mu}(\delta,\beta))$. For $(\delta,\beta)$ on the critical curve, the two pictures 
merge and the flat piece becomes horizontal because $\lambda_{\delta,\beta_c(\delta)}
(0)=1$, $\mu(\delta,\beta_c(\delta))=\tilde{\mu}(\delta,\beta_c(\delta))=0$ and $v(\delta,
\beta_c(\delta))=\tilde{v}(\delta,\beta_c(\delta))$. The boundary value is $I^v_{\delta,\beta}
(0) = \mu(\delta,\beta)-\tilde{\mu}(\delta,\beta)$, while $I^v_{\delta,\beta}(\theta)=\infty$ 
for $\theta \in (1,\infty)$.}
\label{fig-rfspeed}
\end{figure}

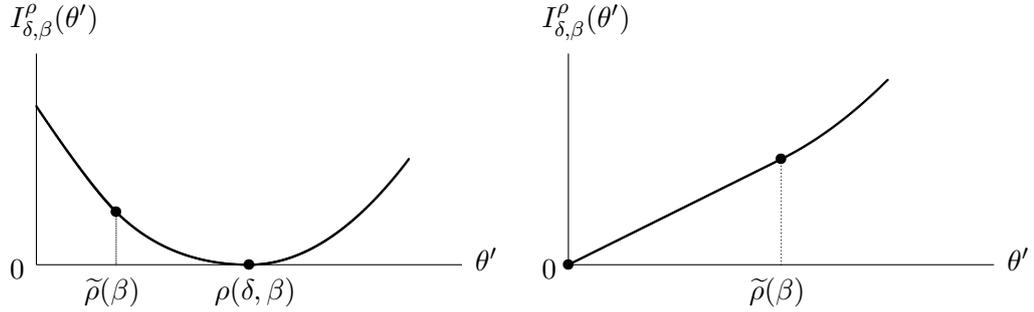
\begin{figure}[htbp]
\vspace{1.5cm}
\begin{center}
\setlength{\unitlength}{0.35cm}
\begin{picture}(12,8)(10,-2)
\put(0,0){\line(16,0){16}}
\put(0,0){\line(0,8){8}}
{\thicklines
\qbezier(0,6)(2,3)(3,2)
\qbezier(3,2)(5,0)(8,0)
\qbezier(8,0)(11,0)(14,4)
}
\qbezier[30](3,2)(3,1)(3,0)
\put(-1,-.5){$0$}
\put(16.5,-0.2){$\theta'$}
\put(8,0){\circle*{.4}}
\put(3,2){\circle*{.4}}
\put(-1,9){$I^\rho_{\delta,\beta}(\theta')$}
\put(6.7,-1.3){$\rho(\delta,\beta)$}
\put(1.8,-1.3){$\tilde{\rho}(\beta)$}
\put(20,0){\line(16,0){16}}
\put(20,0){\line(0,8){8}}
{\thicklines
\qbezier(20,0)(24,2)(28,4)
\qbezier(28,4)(30,5)(32,7)
}
\qbezier[30](28,0)(28,2)(28,4)
\put(19,-.5){$0$}
\put(36.5,-0.2){$\theta'$}
\put(20,0){\circle*{.4}}
\put(28,4){\circle*{.4}}
\put(19,9){$I^\rho_{\delta,\beta}(\theta')$}
\put(26.8,-1.3){$\tilde{\rho}(\beta)$}
\end{picture}
\end{center}
\vspace{0cm}
\caption{\small Qualitative plot of $\theta' \mapsto I^\rho_{\delta,\beta}(\theta')$ for
$(\delta,\beta) \in \intr(\cB)$ (left) and $(\delta,\beta) \in \cS$ (right). The slope of the
flat piece on the left and on the right equals $\delta_c(\beta)-\delta$. For $(\delta,\beta)$
on the critical curve, the two pictures merge and the flat piece becomes horizontal because
$\tilde{\rho}(\beta)=\rho(\delta,\beta_c(\delta))$. The boundary value is $I^\rho_{\delta,\beta}(0) 
= \mu(\delta,\beta)$, while $I^\rho_{\delta,\beta}(\theta')=\infty$ for $\theta' \in (m,\infty)$
with $m \in (0,\infty]$ the essential supremum of the law of $\omega_1$.}
\label{fig-rfcharge}
\end{figure}

The two rate functions are depicted in Figs.~\ref{fig-rfspeed}--\ref{fig-rfcharge}. They are strictly 
convex, except for linear pieces on $[0,\tilde{v}(\delta,\beta)]$ and $[0,\tilde{\rho}(\beta)]$ with
\begin{equation}
\label{eq:speedspec*}
\frac{1}{\tilde{v}(\delta,\beta)} = \left[-\frac{\partial}{\partial\mu} 
\log \lambda_{\delta,\beta}(\mu)\right]_{\mu=\tilde\mu(\delta,\beta)},
\qquad \tilde{\rho}(\beta) = \rho(\delta_c(\beta),\beta),
\end{equation}
where $\beta \mapsto \delta_c(\beta)$ is the inverse of $\delta \mapsto \beta_c(\delta)$
(recall Fig.~\ref{fig-critcurve}). Note that, whereas $v(\delta,\beta)$ in \eqref{vident} and 
$\rho(\delta,\beta)$ in \eqref{rhoident} jump from a strictly positive value to zero when 
$(\delta,\beta)$ moves from $\cB$ to $\cS$ inside $\intr(\cQ)$, $\tilde{v}(\delta,\beta)$ 
and $\tilde\rho(\beta)$ in \eqref{eq:speedspec*} are strictly positive throughout $\intr(\cQ)$.

\medskip
The large deviation principles in turn yield \emph{central limit theorems}:
\begin{theorem}
\label{thm:CLTspeedcharge}
For every $(\delta,\beta) \in \intr(\cB)$,
\begin{equation}
\frac{S_n - nv(\delta,\beta)}{\sigma_v(\delta,\beta)\sqrt{n}},
\qquad \frac{\Omega_n - n\rho(\delta,\beta)}{\sigma_\rho(\delta,\beta)\sqrt{n}},
\end{equation}
converge in distribution to the standard normal law, with $\sigma_v(\delta,\beta),
\sigma_\rho(\delta,\beta) \in (0,\infty)$ given by
\begin{equation}
\label{eq:formula.variance}
\begin{aligned}
\sigma_v(\delta,\beta)^2 
&= \left[\frac{\partial^2}{\partial\theta^2}\, I^v_{\delta,\beta}(\theta)\right]^{-1}_{\theta=v(\delta,\beta)}
= \left[\frac{\partial^2}{\partial\gamma^2}\,\mu(\gd,\gb,\gamma)\right]_{\gamma = 0}\\
&= v(\delta,\beta)^3 \left[\frac{\partial^2}{\partial\mu^2}
\log\lambda_{\delta,\beta}(\mu)\right]_{\mu=\mu(\delta,\beta)},\\
\sigma_\rho(\delta,\beta)^2 
&= \left[\frac{\partial^2}{\partial\theta'^2}\, I^\rho_{\delta,\beta}(\theta')\right]^{-1}_{\theta'=\rho(\delta,\beta)}
= \frac{\partial^2}{\partial\delta^2}\,\mu(\gd,\gb)  = \frac{\partial}{\partial\delta}\,\rho(\gd,\gb).
\end{aligned}
\end{equation}
\end{theorem}

\noindent
The proof of Theorem~\ref{thm:CLTspeedcharge}, that we give
in Section~\ref{CLTproof}, is inspired by K\"onig~\cite{K96}.

The expression in the second line of \eqref{eq:formula.variance} can be written out in terms 
of $v(\gd,\gb)$, $\rho(\gd,\gb)$ and second order derivatives with respect to $\delta$ and 
$\mu$ of $\log\lambda_{\delta,\beta}(\mu)$ at $\mu=\mu(\gd,\gb)$, but the resulting expression 
is not particularly illuminating.


\subsection{Theorems: asymptotic properties}
\label{ss:theoremsasymptotics}

In Theorem \ref{thm:asymp.cc}(1) and Theorem \ref{thm:Fscalbeta} below we need to make an additional assumption on the charge distribution, namely, 
we require that one of the following properties holds: 
\begin{equation}
\label{intcond}
\begin{array}{ll}
& (a) \text{ $\omega_1$ is discrete with a distribution that is lattice}.\\
& (b) \text{ $\omega_1$ is continuous with a density that is in $L^p$ for some $p>1$}. 
\end{array}
\end{equation}

For $a \in \R$ and $b \in (0,\infty)$, let $\cL^{a,b}$ be the \emph{Sturm-Liouville operator} 
defined by
\begin{equation}
\label{SL}
(\cL^{a,b}g)(x) = (2ax-4bx^2)g(x) + g'(x) + xg''(x), \qquad g \in C^2((0,\infty)).
\end{equation}
This is a two-parameter version of a one-parameter family of operators considered in van 
der Hofstad and den Hollander~\cite{vdHdH95}. Let
\begin{equation}
\label{Csetdef}
\cC = \left\{g \in L^2\big((0,\infty)) \cap C^\infty((0,\infty)\big)\colon\,
\|g\|_2 = 1,\,g>0,\,\int_0^\infty \Big[x^{\tfrac92}g(x)^2 + xg'(x)^2\Big]\, \dd x < \infty\right\}.
\end{equation}
The largest eigenvalue problem
\begin{equation}
\label{SLeig}
\cL^{a,b}g = \chi g, \qquad \chi \in \R,\,g \in \cC, 
\end{equation}
has a unique solution $(g^{a,b},\chi(a,b))$ with the following properties: For every
$b \in (0,\infty)$,
\begin{equation}
\label{SLprop}
\begin{aligned}
&a \mapsto \chi(a,b) \text{ is analytic, strictly increasing and strictly convex on } \R,\\
&\chi(0,b) < 0,\, \lim_{a\to\infty} \chi(a,b) = \infty,\,\lim_{a\to-\infty} \chi(a,b) = -\infty,\\
&a \mapsto g^{a,b} \text{ is analytic as a map from } \R \text{ to } L^2((0,\infty)).  
\end{aligned}
\end{equation}
(See Coddington and Levinson~\cite{CoLe55} for general background on Sturm-Liouville 
theory.) 

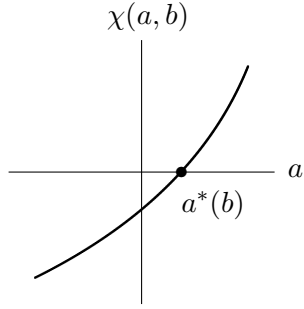
\begin{figure}[htbp]
\vspace{0.5cm}
\begin{center}
\setlength{\unitlength}{0.35cm}
\begin{picture}(8,8)(0,-3)
\put(-5,0){\line(10,0){10}}
\put(0,-5){\line(0,10){10}}
{\thicklines
\qbezier(-4,-4)(2,-1)(4,4)
}
\put(5.5,-.2){$a$}
\put(-1.3,5.6){$\chi(a,b)$}
\put(1.5,-1.5){$a^*(b)$}
\put(1.5,0){\circle*{.45}}
\end{picture}
\end{center}
\vspace{0.5cm}
\caption{\small Qualitative plot of $a \mapsto \chi(a,b)$ for fixed $b \in (0,\infty)$.}
\label{fig-chiab}
\end{figure}

Let $a^*=a^*(b)$ denote the unique solution of the equation $\chi(a,b)=0$ (see Fig.~\ref{fig-chiab}). 
The critical curve has the following scaling behaviour for small and for large charge bias.

\begin{theorem} 
\label{thm:asymp.cc} 
{\rm (1)} As $\delta \downarrow 0$,
\begin{equation}
\label{eq:betacasympzero}
\beta_c(\delta) - \tfrac12 \delta^2 \sim - a^*(1)(\tfrac12 \delta^2)^{\tfrac43}.
\end{equation}
{\rm (2)} As $\delta \to \infty$,
\begin{equation} 
\label{eq:betacasympinf}
\beta_c(\delta) \sim \frac{\delta}{T}
\end{equation}
with 
\begin{equation} 
\label{eq:Tlat}
T = \sup\big\{t > 0\colon\,\bbP(\omega_1 \in t\Z) = 1\big\}
\end{equation}
(with the convention $\sup\emptyset = 0$). Either $T>0$ (`lattice case') or $T=0$ (`non-lattice case'). 
If $T=0$ and $\omega_1$ has a bounded density (with respect to the Lebesgue measure), then
\begin{equation}
\beta_c(\delta) \sim \frac{1}{4} \frac{\delta^2}{\log \delta}.
\end{equation}
\end{theorem}

The proof of \eqref{eq:betacasympzero}, given in Section~\ref{ss:WIL},
follows
van der Hofstad and den Hollander~\cite{vdHdH95}, but we have to address additional
difficulties, due to our more complicated Hamiltonian.

The scaling behaviour of the excess free energy near the critical curve shows that the phase
transition is first order.

\begin{theorem}
\label{thm:orderphtr}
For every $\delta \in (0,\infty)$,
\begin{equation}
\label{Cidscal}
F^*(\delta, \beta) \sim K_\delta [\beta_c(\delta)-\beta],
\quad \mbox{ as } \beta \uparrow \beta_c(\delta),
\end{equation}
where $K_\delta \in (0,\infty)$ is given by
\begin{equation}
\label{Cid}
K_\delta = \left[ \frac{\frac{\partial}{\partial\beta}\log\lambda_{\delta,\beta}(\mu)}
{\frac{\partial}{\partial\mu}\log\lambda_{\delta,\beta}(\mu)}\right]_{\beta=\beta_c(\delta),\mu=0}.
\end{equation}
\end{theorem}

We close by identifying the scaling behaviour of the free energy for small charge bias and small 
inverse temperature. The proof also follows der Hofstad and den Hollander~\cite{vdHdH95}.

\begin{theorem}
\label{thm:Fscalbeta}
{\rm (1)} For every $\delta \in (0, \infty)$,
\begin{equation}
\label{Fvrhoscal}
F(\delta,\beta) \sim - A_\delta \beta^{\tfrac23},
\qquad v(\delta,\beta) \sim B_\delta \beta^{\tfrac13},
\qquad \rho(\delta,\beta) - \rho_\delta \sim C_\delta\beta^{\tfrac23},
\qquad \mbox{ as } \beta \downarrow 0,
\end{equation}
where $\rho_\delta = \bbE^\delta(\omega_1)=-f'(\delta)$, and $A_\delta, B_\delta,C_\delta 
\in (0,\infty)$ are given by
\begin{equation}
\label{eq:scalconst}
A_\delta = a^*(\rho_\delta), \qquad 
\frac{1}{B_\delta} = \left[\frac{\partial}{\partial a}\, \chi(a,b)\right]_{a = a^*(\rho_\delta),
\,b=\rho_\delta}, \qquad C_\delta = -\frac{\dd}{\dd\delta}\,a^*(\rho_\delta).
\end{equation}
The third statement in \eqref{Fvrhoscal} holds under the assumption that
\begin{equation}
\label{assvars}
\limsup_{\gb \downarrow 0} \gb^{-\tfrac23}
|\sigma^2_\rho(\gd,\gb) - \sigma^2_\rho(\gd,0)| < \infty
\quad \text{ uniformly on a neighbourhood of } \gd, 
\end{equation} 
with $\sigma_\rho(\gd,\gb)$ defined in \eqref{eq:formula.variance} and $\sigma^2_\rho
(\gd,0) = -f''(\gd)$. Without this assumption only the weaker result $\lim_{\beta 
\downarrow 0} \rho(\delta,\beta) = \rho_\delta$ holds.\\
{\rm (2)} For every $\epsilon>0$,
\begin{equation}
\label{eq:three}
F^*(\delta,\beta) \sim \beta_c(\delta)-\beta, \qquad \mbox{ as } \delta,\beta \downarrow 0,
\quad \text{ provided } \beta_c(\gd) - \beta \asymp \delta^{\tfrac83}.
\end{equation}
\end{theorem}

\noindent
(The notation $f \asymp g$ means that the ratio $f/g$ stays bounded from above 
and below by finite and positive constants.)


\subsection{Discussion}
\label{ss:discussion}

We discuss the theorems stated in Sections~\ref{ss:theoremsgeneral}--\ref{ss:theoremsasymptotics} 
and place them in their proper context. 

\medskip\noindent
{\bf 1.} 
The quenched charged polymer model with $\bbP = [\tfrac12(\delta_{-1}+\delta_{+1})]^{\otimes\N}$ 
interpolates between the simple random walk ($\beta = 0$), the self-avoiding walk ($\beta=\delta=\infty$) 
and the weakly self-avoiding walk ($\beta \in (0,\infty)$, $\delta=\infty$), for which an abundant literature 
is available (see den Hollander~\cite[Chapter 2]{dH09} for references). The latter corresponds to the 
situation where all the charges are $+1$, in which case the Hamiltonian in \eqref{eq:def.hamiltonian} 
equals $H_n(S) = \sum_{x\in\Z} L_n(S,x)^2$ with 
\begin{equation}
\label{eq:deflocaltimes}
L_n(S,x) = \sum_{i=1}^n \ind_{\{S_i = x\}}
\end{equation} 
the local time of $S$ at site $x$ up to time $n$. Theorem~\ref{thm:free.energy} shows that the 
annealed excess free energy exists and has a spectral representation. The latter generalizes 
the spectral representation derived in Greven and den Hollander~\cite{GrdH93} for weakly 
self-avoiding walk (see den Hollander~\cite[Chapter IX]{dH00}). Theorem~\ref{thm:critcurve}
shows that there is a phase transition at a non-trivial critical curve and that there are no further
subphases.

\begin{figure}[htbp]
\vspace{0.5cm}
\begin{center}
\includegraphics[scale = 0.4]{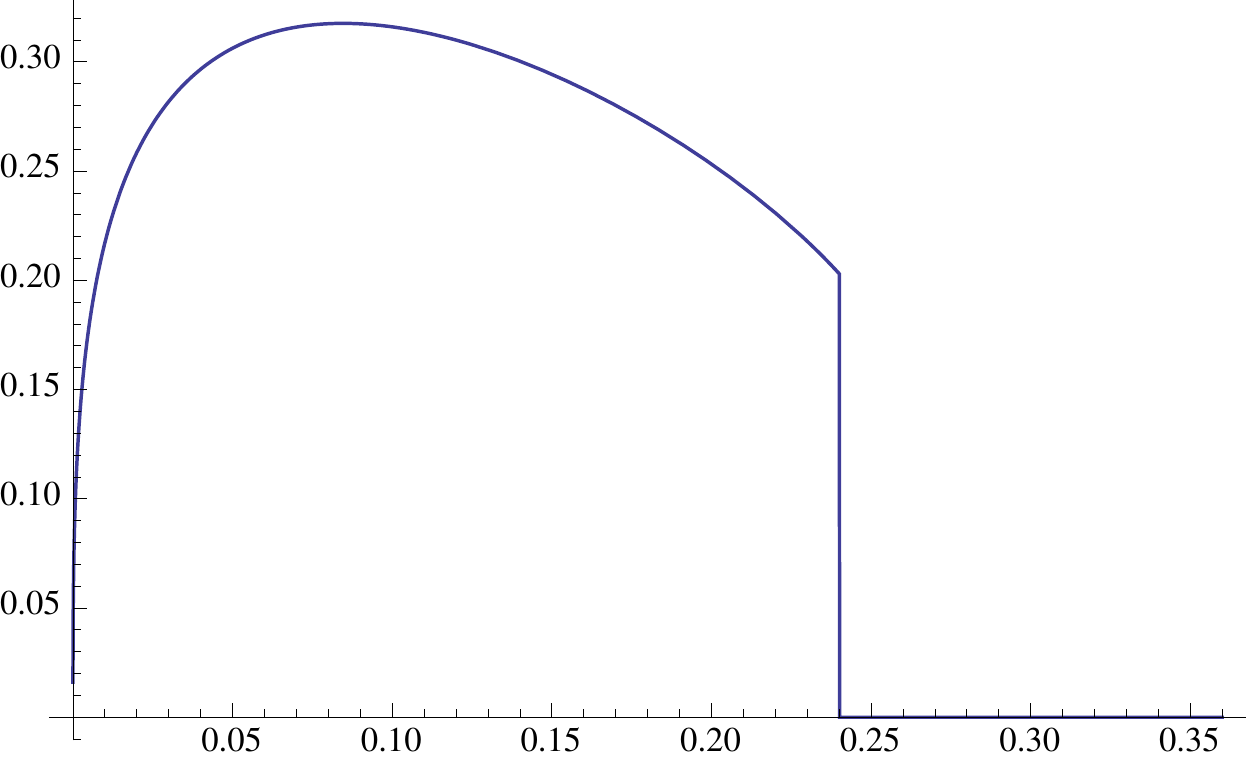}
\hspace{1cm}
\includegraphics[scale = 0.4]{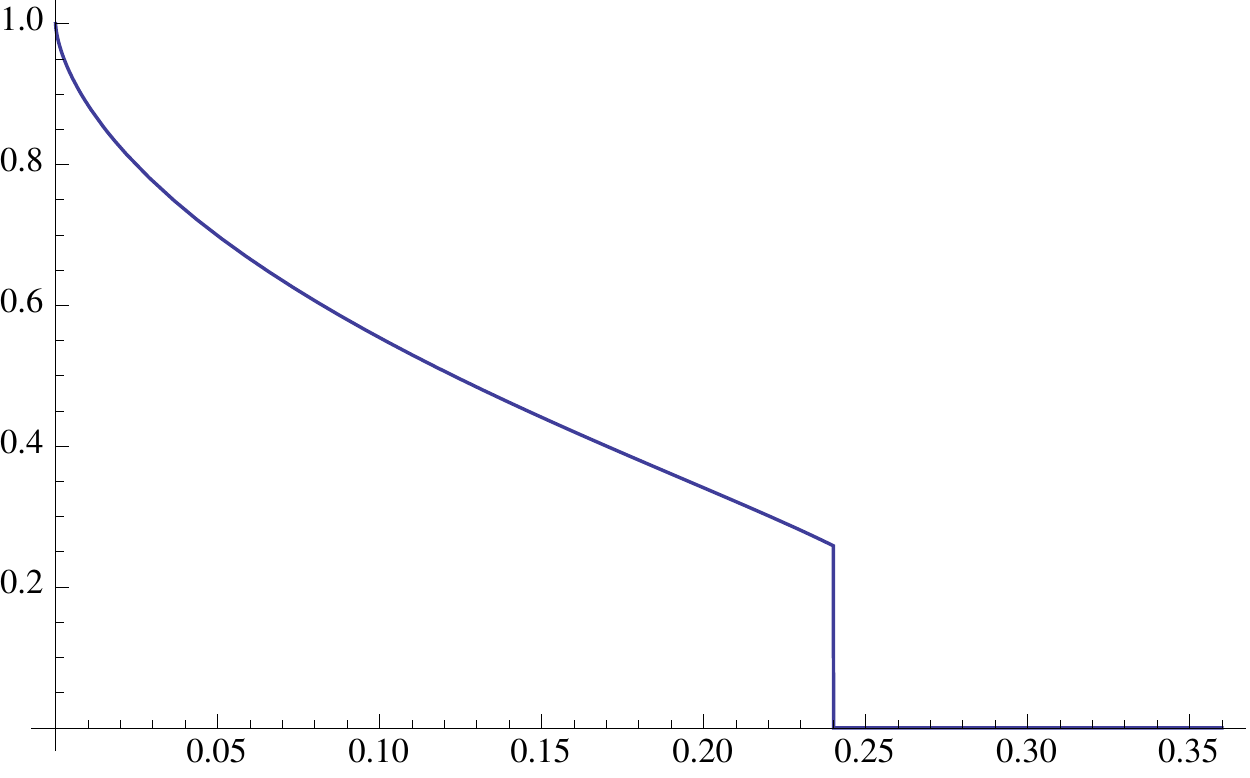}

$\mbox{}$

\includegraphics[scale = 0.25]{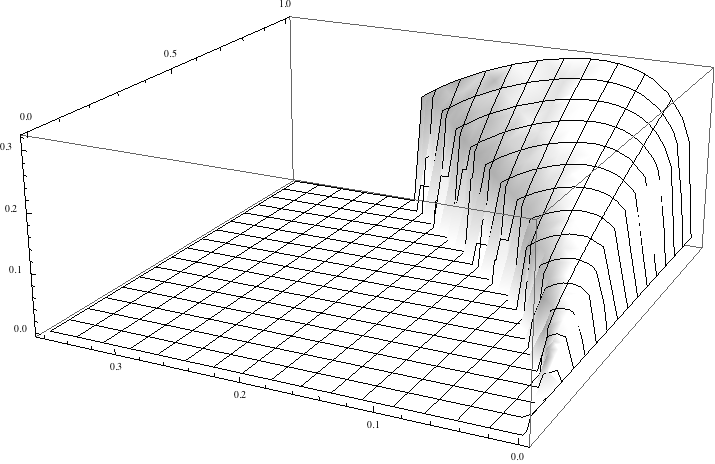}
\hspace{.15cm}
\includegraphics[scale = 0.25]{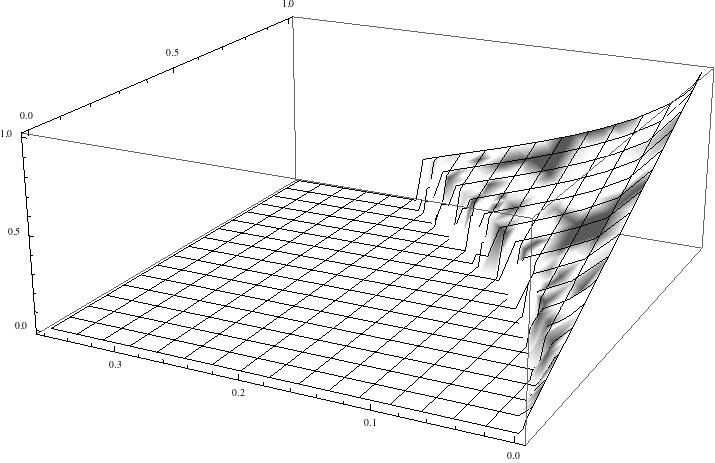}
\end{center}
\vspace{0cm}
\caption{\small Numerical plots of the typical speed $v(\delta, \beta)$
and the typical charge $\rho(\delta,\beta)$ in Theorems~\ref{thm:speed}
and~\ref{thm:charge}, based on a $100 \times 100$ truncation of the 
matrix in \eqref{eq:def.Abeta.p.r}, for the case where $\omega_1$ is 
standard normal. \emph{Above:} plot of $\beta \mapsto v(\delta,\beta)$ 
and $\beta \mapsto \rho(\delta,\beta)$ for $\delta=1$ and $\beta \in 
(0,0.36)$. \emph{Below:} same  for $\delta \in (0,1)$ and $\beta \in (0,0.36)$ 
(for graphical clarity the axes have been rotated: the $\delta$-axis runs 
from front to back, the $\beta$-axis runs from right to left).}
\label{fig-speedcharge}
\end{figure}

\medskip\noindent
{\bf 2.}
Theorems~\ref{thm:speed}--\ref{thm:charge} and \ref{thm:orderphtr} show that the 
annealed charged polymer exhibits a phase transition of first order. The speed 
$v(\delta,\beta)$ of the polymer chain is strictly positive in the ballistic phase and 
zero in the subballistic phase (which explains the names associated with these two 
phases). In the ballistic phase the speed is given by the spectral formula in 
\eqref{eq:speedspec}. The latter generalizes the spectral formula derived in Greven 
and den Hollander~\cite{GrdH93} for the speed $v(\beta)=v(\infty,\beta)$ of the
weakly self-avoiding walk. The charge $\rho(\delta,\beta)$ of the polymer chain is strictly 
positive in the ballistic phase and zero in the subballistic phase. In the ballistic phase the 
charge is given by the spectral formula in \eqref{eq:chargespec}. Fig.~\ref{fig-speedcharge} 
shows a numerical plot of $\beta \mapsto v(1,\beta)$ and $\beta \mapsto \rho(1,\beta)$ when $\go_1$ is standard normal.
Interestingly, the speed is not monotone on $(0,\beta_c(1)]$. This is in contrast with
the monotonicity that was found (but was not proven) in \cite{GrdH93} for the weakly 
self-avoiding walk (for which $\beta_c(\infty)=\infty$). Equally interesting, the charge
is monotone on $(0,\beta_c(1)]$. A rough heuristics behind the shape of $v(\delta,\beta)$
and $\rho(\delta,\beta)$ is the following. Approximating the distributions of $S_n$ and 
$\Omega_n$ by standard normal laws, we get
\begin{equation}
\begin{aligned}
F^*(\delta,\beta) 
&= \lim_{n\to\infty} \frac{1}{n} \log (\bbE^\delta \times \bE)\,
\left(\exp\left[-\beta \sum_{x \in \Z} \left(\sum_{i=1}^n \omega_i\,1_{\{S_i=x\}}\right)^2 
+ \delta \sum_{i=1}^n \omega_i \right]\right)\\ 
&\approx \sup_{ {v \in (0,\infty)} \atop {\rho \in \R} } 
\left[-\beta v \left(\frac{\rho}{v}\right)^2 + \delta \rho - \tfrac12(v^2+\rho^2)\right].
\end{aligned}
\end{equation}
Here, the supremum runs over the possible values of the empirical speed and the empirical 
charge, the first term arises from the Hamiltonian in \eqref{eq:def.hamiltonian}, the second 
term comes from the tilting of the charges in \eqref{eq:Pdeltadef}, together with the approximation
$\bbE^\delta(\omega_1) \approx \delta$, while the third term embodies the normal approximation. 
For fixed $\rho$ the supremum over $v$ is taken at $v=\beta^{1/3}\rho^{2/3}$. Substitution 
of this relation shows that the supremum over $\rho$ is taken at the solution of the equation
$\rho = \delta - 2\beta^{2/3}\rho^{1/3}$. Hence
\begin{equation}
v(\delta,\beta) \approx \beta^{\tfrac13} \rho(\delta,\beta)^{\tfrac23},
\qquad \rho(\delta,\beta) \approx \delta - 2\beta^{\tfrac23}\rho(\delta,\beta)^{\tfrac13}.
\end{equation} 
These approximations are compatible with the numerical plots in Fig.~\ref{fig-speedcharge}.  

\medskip\noindent
{\bf 3.} 
Theorem~\ref{thm:jointLDP} identifies the rate functions in the large deviation principles 
for the speed and the charge. Both rate functions exhibit flat pieces in both phases, as 
indicated in Figs.~\ref{fig-rfspeed}--\ref{fig-rfcharge}. These flat pieces correspond to 
an inhomogeneous strategy for the polymer to realise a large deviation. For instance, 
in the flat piece on the left of Fig.~\ref{fig-rfspeed}, if the speed is $\theta<\tilde{v}(\delta,
\beta)$, then the charge makes a large deviation on a stretch of the polymer of length 
$\theta/\tilde{v}(\delta,\beta)$ times the total length, so as to allow it to move at speed 
$\tilde{v}(\delta,\beta)$ along that stretch at zero cost, and then makes a large deviation 
on the remaining stretch, so as to allow it to be subballistic along that remaining stretch 
at zero cost. For the weakly self-avoiding walk the presence of a flat piece in the rate 
function for the speed was noted in den Hollander~\cite[Chapter 8]{dH00}. It is possible 
to extend Theorem~\ref{thm:jointLDP} to a joint LDP, but we refrain from doing so.

\medskip\noindent
{\bf 4.}
Theorem~\ref{thm:CLTspeedcharge} provides the central limit theorem for the speed 
and the charge in the interior of the ballistic regime. The variance is the inverse of the 
curvature of the rate function at its unique zero, as is to be expected. Numerical plots 
are given in Fig.~\ref{fig-speedchargevar}. It is hard to obtain accurate simulations for 
$\beta$ small, but the plots appear to be compatible with the assumption made in 
\eqref{assvars}. For weakly self-avoiding walk it was shown in van der Hofstad, den 
Hollander and K\"onig~\cite{vdHdHK97a,vdHdHK97b} that $\beta \mapsto \sigma_v
(\beta)^2 = \sigma_v(\infty,\beta)^2$ is discontinuous at $\beta=0$, namely, $\lim_{\beta
\downarrow 0} \sigma_v(\beta) = C_v < 1 = \sigma_v(0)$. Fig.~\ref{fig-speedcharge} 
suggests that this behaviour persists for $\delta<\infty$. The heuristics is that the 
variance of the endpoint of the polymer gets squeezed because the polymer moves 
ballistically. Apparently this squeezing does not vanish as the speeds tends to zero.   

\begin{figure}[htbp]
\vspace{0.5cm}
\begin{center}
\includegraphics[scale = 0.4]{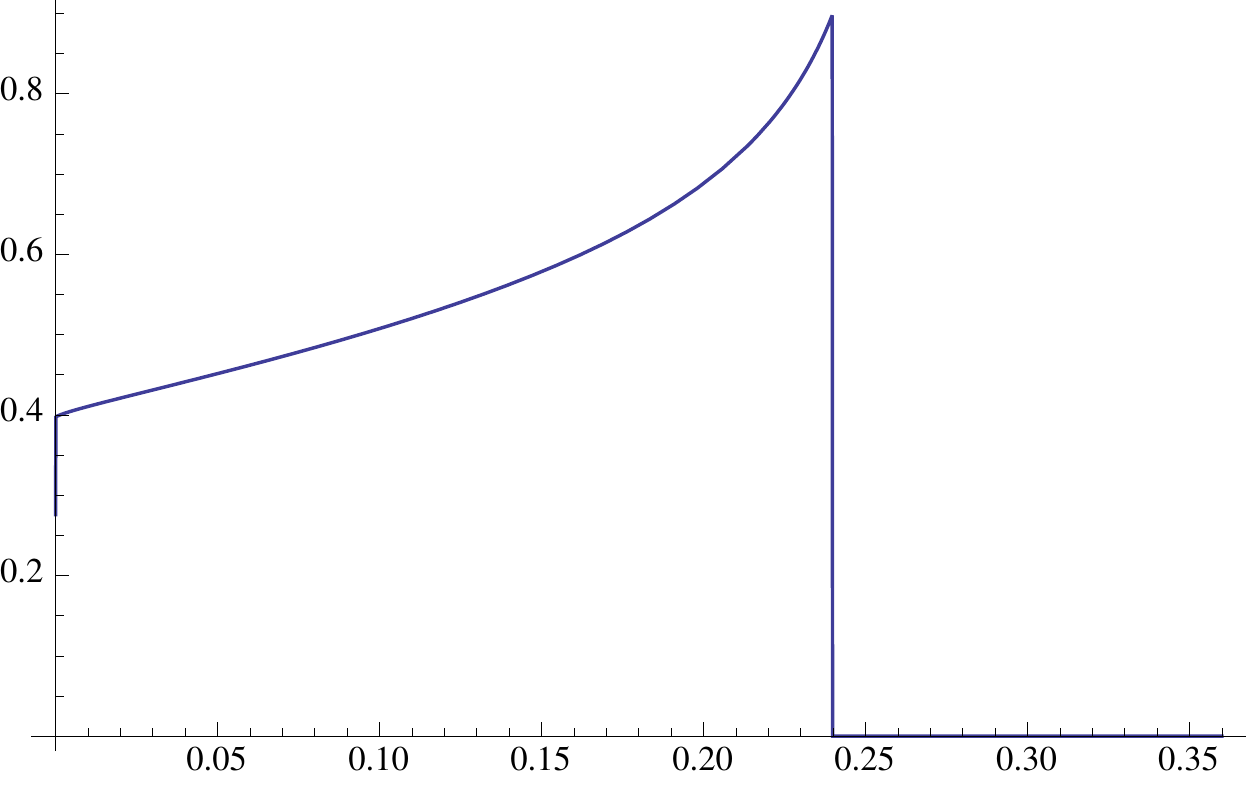}
\hspace{1cm}
\includegraphics[scale = 0.4]{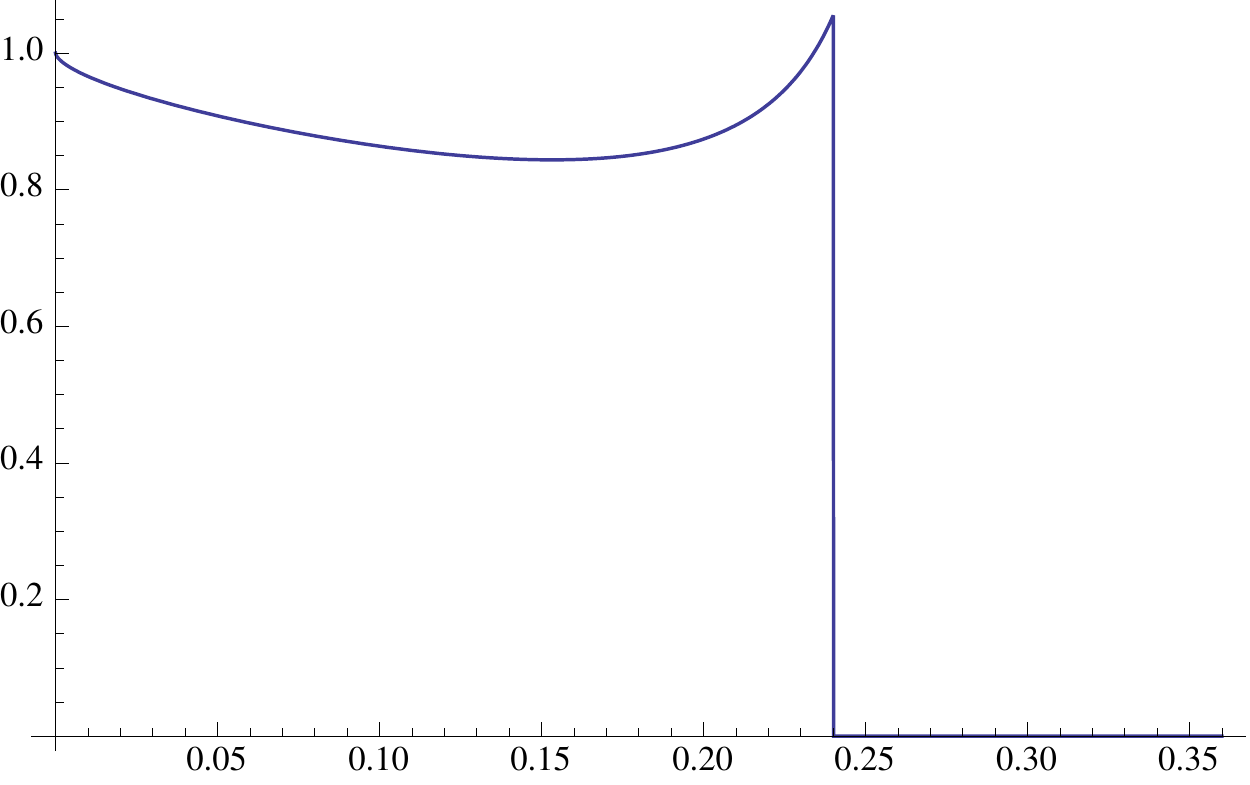}

$\mbox{}$

\includegraphics[scale = 0.25]{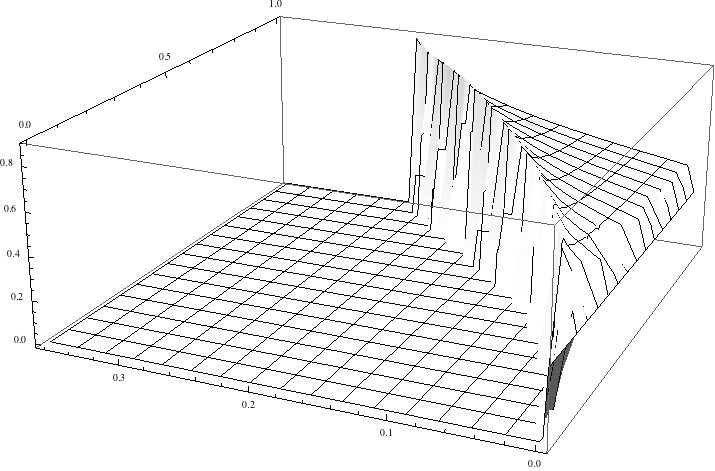}
\hspace{.15cm}
\includegraphics[scale = 0.25]{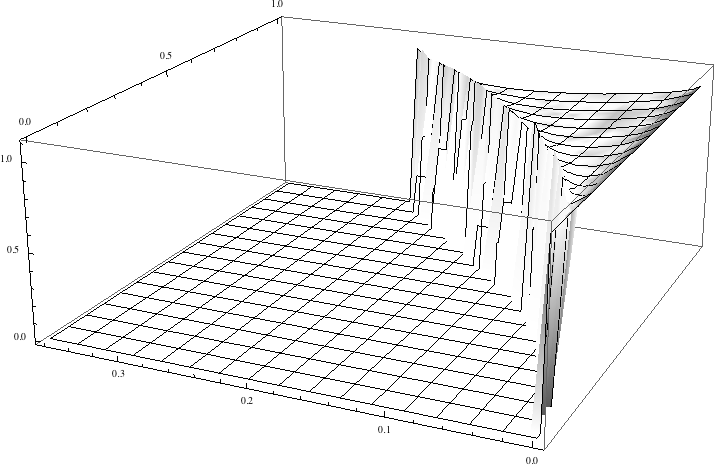}
\end{center}
\vspace{0cm}
\caption{\small Numerical plots of the variance of the speed $\sigma_v(\delta,\beta)^2$ 
and the variance of the charge $\sigma_\rho(\delta,\beta)^2$ in 
Theorem~\ref{thm:CLTspeedcharge} for the same range of $\beta$ and $\delta$ as in 
Fig.~\ref{fig-speedcharge}.}
\label{fig-speedchargevar}
\end{figure}

\noindent
We do not deduce the central limit theorem from the large deviation principle, 
but rather exploit finer properties of the spectral representation for the excess free energy. 
We have no result about the fluctuations at criticality. We expect these fluctuations to be 
of order $\sqrt{n}$ in the upward direction and of order $n^{2/3}$ in the downward direction.

\medskip\noindent
{\bf 5.}
Theorem~\ref{thm:asymp.cc} identifies the scaling behavior of the critical curve for small 
and for large charge bias. Part (1) shows that the scaling is anomalous for small charge bias,
and implies that the critical curve is not analytic at the origin. Part (2) shows that the scaling 
is also delicate for large charge bias. Heuristically, it is easier to build small absolute values 
of $\Omega_\ell=\sum_{k=1}^\ell \go_k$ for small values of $\ell$ when the charge distribution 
is non-lattice rather than lattice. Since the local times are of order one in the ballistic phase, 
we expect that the ballistic phase for the lattice case is contained in the ballistic phase for 
the non-lattice case (because smaller values of $\beta$ are needed to compensate for the 
larger absolute values of $\Omega_\ell$). 

\medskip\noindent
{\bf 6.}
Theorem~\ref{thm:Fscalbeta} deals with weak interaction limits. Part (1) shows that near 
the horizontal axis in Fig.~\ref{fig-critcurve} the free energy, the speed and the charge exhibit 
an anomalous scaling. This is a generalization of the scaling found in van der Hofstad and 
den Hollander~\cite{vdHdH95} for weakly self-avoiding walk. Part (2) shows that near the 
origin of Fig.~\ref{fig-critcurve} the free energy scales like the distance to the critical curve,
provided the latter is approached properly. The constants $A_\gd, B_\gd, C_\gd$ are expected 
to represent the free energy, speed and charge of a Brownian version of the charged polymer 
with Hamiltonian
\begin{equation}
\label{HTcont}
H_T^\tW(W[0,T]) = \int_\R L_T^\tW(x)^2\,\dd x, \quad L_T^\tW(x) = \int_0^T \dd\tW_s\,\gd(W_s - x),
\end{equation}
where $W[0,T]$ is the path of the polymer, $\dd \tW_s$ is the charge of the interval $\dd s$, 
$\tW[0,T]$ is an independent Brownian motion with drift $\gd$, and the polymer measure has 
$\gb=1$ with the Wiener measure as reference measure. The version without charges is known as 
the Edwards model (see van der Hofstad, den Hollander and K\"onig~\cite{vdHdHK97b, vdHdHK03}). 
The limit below \eqref{eq:scalconst} is expected to represent the standard deviation in the 
central limit theorem as $T\to\infty$ for the charge in the continuum model defined via \eqref{HTcont}. 

\medskip\noindent
{\bf 7.}
Theorem~\ref{thm:critcurve} corrects a mistake in den Hollander~\cite[Chapter 8]{dH09},
where it was argued that $F^* \equiv 0$ (i.e., $\cS$ covers the full quadrant, or $\beta_c 
\equiv 0$). The mistake can be traced back to a failure of convexity of the function $\ell
\mapsto G^*_{\delta,\beta}(\ell)$. Using the technique outlined in den Hollander~\cite[Chapter 8]{dH09}, 
it can be shown that for every $d \geq 1$ and every $(\delta,\beta) \in \cS$,
\begin{equation}
\label{fannlimit}
\lim_{n\to\infty} \frac{(\alpha_n)^2}{n} \, \log \bbZ_n^{*,\delta,\beta}  = -\chi,
\quad \mbox{ where } \quad \bbZ_n^{*,\delta,\beta} = e^{-f(\delta)n}\,\bbZ_n^{\delta,\beta}, 
\end{equation} 
with $\alpha_n = (n/\log n)^{1/(d+2)}$ and with $\chi \in (0,\infty)$ a constant that is explicitly 
computable. The idea behind \eqref{fannlimit} is the following. For $(\delta,\beta) \in \cS$ the 
empirical charge makes a large deviation under the disorder measure $\bbP^\delta$ so that it 
becomes zero. The price for this large deviation is 
\begin{equation}
e^{-nH(\bbP^0 \,|\, \bbP^\delta) + o(n)},
\end{equation}
where $H(\bbP^0 \,|\, \bbP^\delta)$ denotes the specific relative entropy of $\bbP^0=\bbP$ with 
respect to $\bbP^\delta$. Since the latter equals $\log M(\delta)=-f(\delta)$ (recall 
\eqref{eq:omegacond}--\eqref{eq:Pdeltadef}), this accounts for the leading term in the 
free energy. Conditional on the empirical charge being zero, the attraction between charged 
monomers with the same sign wins from the repulsion between charged monomers with 
opposite sign, making the polymer chain contract to a \emph{subdiffusive} scale $\alpha_n$. 
This accounts for the correction term in the free energy. It is shown in \cite{dH09} that under 
the annealed polymer measure,
\begin{equation}
\left(\frac{1}{\alpha_n}\,S_{\lfloor nt\rfloor}\right)_{0 \leq t \leq 1} 
\Longrightarrow (U_t)_{0 \leq t \leq 1}, \qquad n\to\infty,
\end{equation}
where $\Longrightarrow$ denotes convergence in distribution and $(U_t)_{t \geq 0}$ is a 
Brownian motion on $\R^d$ conditioned not to leave a ball with a certain radius and a certain 
randomly shifted center.

\medskip\noindent
{\bf 8.} 
Previous results on the charged polymer model include limit theorems for the Hamiltonian 
in \eqref{eq:def.ham}. Chen~\cite{Ch08} proves an annealed central limit theorem and an
annealed law of the iterated logarithm, and identifies the annealed moderate deviations (see 
also Chen and Khoshnevisan~\cite{ChK09}). Asselah~\cite{As10}, \cite{As11} derives upper 
and lower bounds for annealed large deviations. Hu and Khoshnevisan~\cite{HK10} give a 
law of the iterated logarithm and a strong approximation theorem: on an enlarged probability 
space the properly normalised Hamiltonian converges almost-surely to a reparametrised 
Brownian motion. Guillotin-Plantard and dos Santos~\cite{GPS14} prove a quenched central 
limit theorem in dimensions $d=1,2$. Hu, Khoshnevisan and Wouts~\cite{HKW11} consider 
the quenched weak interaction regime (where the Hamiltonian is multiplied by $\gb/n$ rather 
than $-\beta$) and prove a phase transition from Brownian scaling to four-point localization: 
for small $\beta$ the polymer behaves like a simple random walk, while for large $\gb$ a large 
fraction of the monomers are located on four sites.

\begin{figure}[htbp]
\begin{center}
\includegraphics[scale = 0.4]{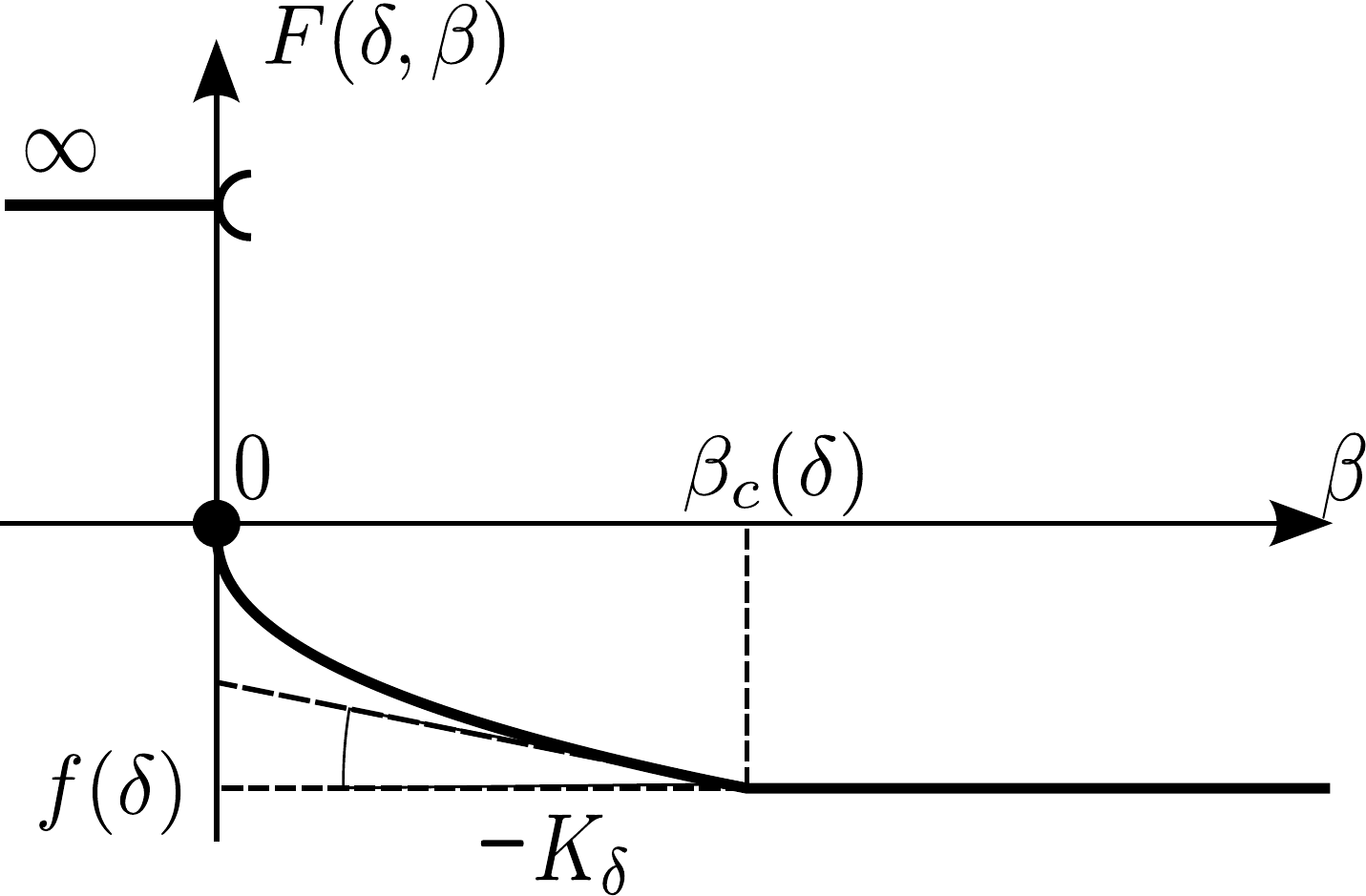}
\hspace{1cm}
\includegraphics[scale = 0.4]{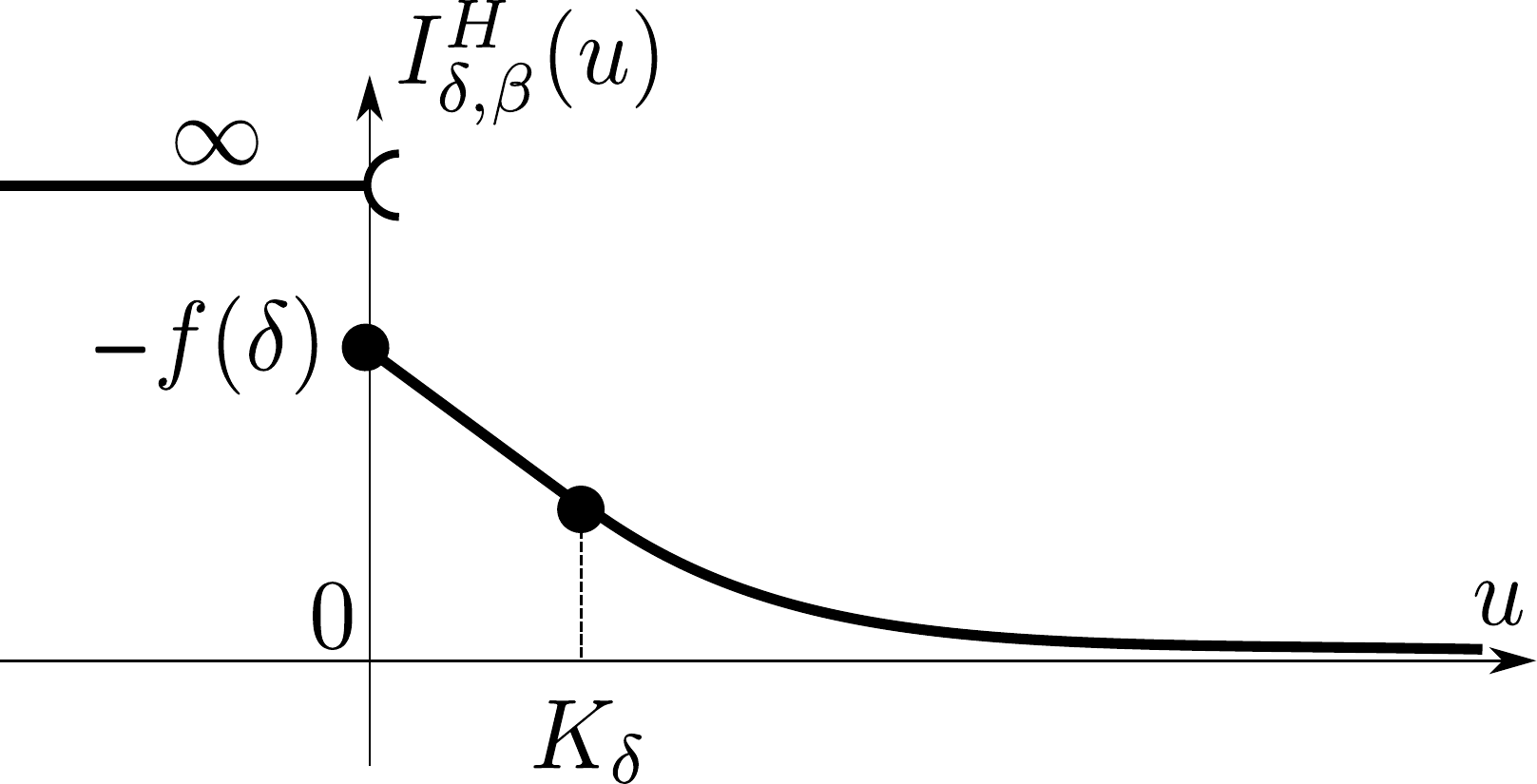}

$\mbox{}$

\end{center}
\vspace{0cm}
\caption{\small Qualitative plots of the maps $\beta \mapsto F(\gd,\gb)$ and $u \mapsto 
I_\gd^H(u)$. The latter is linear on $[0,K_\gd]$ and strictly convex on $(K_\gd,\infty)$, 
where $K_\delta$ is the constant in \eqref{Cid}, and tends to zero at infinity.}
\label{fig-ratefunctionHam}
\end{figure}

\medskip\noindent
{\bf 9.}
The large deviation bounds derived by Asselah~\cite{As10}, \cite{As11} can be completed as 
follows. Under the annealed polymer measure, the sequence $(n^{-1}H_n^\omega)_{n\in\N}$ 
(recall \eqref{eq:def.hamiltonian}) satisfies the (weak) large deviation principle on $\R$ 
with (weak) rate function $I^H_{\delta}$ given by (see Fig.~\ref{fig-ratefunctionHam}) 
\begin{equation}
\label{eq:rateH}
I^H_{\delta}(u) = \sup_{\beta \in (0,\infty)} [-u\beta - F(\delta,\beta)],
\qquad u \in \R.
\end{equation}
Here we use that $F(\delta,\beta) = \infty$ for $\beta \in (-\infty,0)$ and $F(\delta,0)=0$, to 
restrict the supremum to $\beta \in (0,\infty)$. (Indeed, the strategy where the charges are 
bounded from below by a positive constant and the walk zigzags between two consecutive 
sites has an entropic cost that is linear in the length of the polymer, whereas the positive 
energetic contribution is quadratic.) Since $F(\delta,\beta) = \mu(\delta,\beta)+f(\delta)$ by 
Theorem~\ref{thm:free.energy}, \eqref{eq:rateH} provides us with an explicit variational 
formula similar to \eqref{Ivdef}--\eqref{Irhodef}. 

\medskip\noindent
{\bf 10.} 
Here are some open problems for the \emph{quenched} version of the model (see 
Appendix~\ref{AppB}): 
\begin{itemize}
\item[(1)]
Does the quenched free energy exist for $\bbP^\delta$-a.e.\ $\omega$, and is it constant? How 
does it depend on $\delta$ and $\beta$? Trivially, it is convex in $\beta$ for all $\delta$, but what 
more can be said?
\item[(2)] 
Is the quenched charged polymer ballistic for all $\delta \in (0,\infty)$? How does the speed 
depend on $\beta$ and $\delta$?
\item[(3)]  
In the quenched model with $\delta=0$, is the polymer chain subdiffusive (like in the annealed 
model; see item 3 above)? The fluctuations of the charges are expected to push the polymer 
farther apart than in the annealed model. Is there a scaling limit for $\bbP$-a.e.\ $\omega$, or 
does the polymer chain fluctuate so much that there is a scaling limit only along $\omega$-dependent 
subsequences (``sample dependence'')? 
\end{itemize}

\medskip\noindent
{\bf 11.} 
Still looking at a quenched model, Derrida, Griffiths and Higgs~\cite{DeGrHi92} and Derrida and Higgs~\cite{DeHi94}
consider the case where the steps of the random walk are drawn from $\{0,1\}$ rather than 
$\{-1,+1\}$, which makes the model a  bit more tractable, both theoretically and numerically. 
In \cite{DeGrHi92} the charge disorder is \emph{binary}, and numerical evidence is found for the 
free energy to be self-averaging and to exhibit a \emph{freezing transition} at a critical threshold 
$\beta_c \in (0,\infty)$, i.e., the quenched charged polymer is ballistic when $0\leq\beta<\beta_c$ 
and subballistic when $\beta>\beta_c$. In the latter phase numerical simulation shows that 
the end-to-end distance scales like $n^\nu$, with $\nu=\nu(\beta)$ an exponent that depends 
on $\beta$. In this phase, long and rare stretches of the polymer that are globally neutral find 
it energetically favorable to collapse onto single sites. Numerical simulation indicates that 
$\beta_c \geq 0.48$. In \cite{DeHi94} the charge disorder is \emph{standard normal} and the 
total charge $\sum_{i=1}^n \omega_i$ is \emph{conditioned} to grow like $n^\xi$, $\xi \in [-\frac12,1]$. 
It is found numerically that the end-to-end distance scales like $n^\nu$, with $\nu=\nu(\xi)$ 
an exponent that depends on $\xi$ and grows roughly linearly from $\nu(-\frac12)=0$ to 
$\nu(1)=1$, with $\nu(\frac12) \approx 0.574$. The latter is the exponent for the quenched 
charged polymer when the charges are typical.

\medskip\noindent
{\bf 12.}
It would be interesting to deal with charges whose interaction extends beyond the `on-site' 
interaction in (\ref{eq:def.ham}), like a Coulomb potential (polynomial decay) or a Yukawa 
potential (exponential decay). A Yukawa potential arises from a Coulomb potential via 
screening of the charges when the polymer chain is immersed in an ionic fluid. 
 
\medskip\noindent
{\bf 13.}
Biskup and K\"onig~\cite{BiKo01}, Ioffe and Velenik~\cite{IoVe12}, Kosygina and 
Mountford~\cite{KoMo13} deal with annealed versions of various models of simple
random walk in a random potential. In all these models the interaction is either 
\emph{attractive} or \emph{repulsive}, meaning that the annealed partition function 
is the expectation of the exponential of a functional of the local times of simple random 
walk that is either \emph{subadditive} or \emph{superadditive}. As we will see in 
Section~\ref{freeenergy}, our annealed charged polymer model is neither attractive
nor repulsive. However, our spectral representation is flexible so as to include such
models.


\section{Spectral representation for the free energy}
\label{freeenergy}

Our goal in this section is to prove Theorem~\ref{thm:free.energy}, i.e., the existence 
of the annealed free energy and its characterization in terms of an eigenvalue problem. 
In Section~\ref{sec:markov_edge_local_times} we show that the edge-crossing numbers 
of the simple random walk have a Markovian structure. In Section~\ref{ss:loctimerepr} we 
rewrite the annealed partition function as the expectation of a functional of the local times 
of the simple random walk, which are a two-block functional of the edge-crossing numbers. 
In Section~\ref{ss:grcanrepr} we introduce the generating function of the annealed excess 
partition function, and show that this can be expressed in terms of the matrices defined in
\eqref{eq:def.Abeta.p.r}--\eqref{eq:def.Abeta.p.r2}. The annealed excess free energy is the 
radius of convergence of this generating function. In Section~\ref{subsec:eigenvalue} we 
analyze the spectral radii of the matrices. In Section~\ref{ss:spec} we identify the annealed 
excess free energy in terms of these spectral radii. In Section~\ref{ss:concl} we put everything 
together to prove Theorem~\ref{thm:free.energy}.

This section is the cornerstone of the following sections, since the representation of the partition function developed here will be used throughout the paper.


\subsection{Markov property of the edge-crossing numbers}
\label{sec:markov_edge_local_times}

The observation that the edge-crossing numbers of the simple random walk have a Markovian 
property goes back at least to Knight~\cite{Kn63}. This property can be formulated in various 
ways. In this section we present a version that holds for a fixed time horizon, which is based 
on the well-known link between random walk excursions and rooted planar trees (see 
Remark~\ref{rem:tree} below).

We work conditionally on the event $\{S_n = x\}$ for fixed $n \in \N_0$ and $x \in \Z$, and 
w.l.o.g.\ we assume that $x \in \N_0$. Then all edges are crossed the same number of times 
upwards and downwards, except for the edges in the stretch $\{0, \ldots, x\}$, which have 
one extra upward crossing. We define the \emph{edge-crossing number} $M^+_y$, $y\in\N_0$, 
as the number or upward crossing of the edge $(y,y+1)$ that are eventually followed by a 
downward crossing (i.e., we disregard the last upward crossing for $0 \leq y < x$). To keep 
the notation symmetric, we define $M^-_y$, $y\in\N_0$, as the number of downward crossings 
of the edge $(-y-1,-y)$, each of which necessarily is eventually followed by an upward crossing. 
In formulas,
\begin{equation}
\label{eq:edge}
M^+_y = \Bigg\lfloor \frac{1}{2} \sum_{k=1}^n 
\ind_{\big\{\{S_{k-1},S_k\} = \{y, y + 1\}\big\}} \Bigg\rfloor,
\quad 
M^-_y = \Bigg\lfloor \frac{1}{2} \sum_{k=1}^n 
\ind_{\big\{\{S_{k-1},S_k\} = \{-y,-1-y\}\big\}}
\Bigg\rfloor, \quad y \in \N_0.
\end{equation}
For ease of notation we suppress the dependence on $n$.

\begin{remark}
\rm In what follows we will work with the \emph{local times} of the random walk,
i.e., the \emph{site visit numbers} defined by
\begin{equation}
\label{eq:localtime}
L_n(x) = \sum_{i=1}^n \ind_{\{S_i = x\}} , \qquad n \in \N, \, x \in \Z.
\end{equation}
These can be expressed in terms of the edge-crossing numbers as follows:
\begin{equation} 
\label{eq:LM}
\text{On the event } \{S_n = x\} \text{ with } x \in\N_0\colon \quad
L_n(y) = \begin{cases}
M^+_{y-1} + M^+_y, & \text{if } y > x, \\
M^+_{y-1} + M^+_y + 1, & \text{if } 1 \le y \le x, \\
M^+_0 + M^-_0, & \text{if } y = 0, \\
M^-_{-y-1} + M^-_{-y}, & \text{if } y < 0. 
\end{cases}
\end{equation}
\end{remark}

\medskip
We next define a specific \emph{branching process}, which will be shown to be closely
linked to the edge-crossing numbers $M^\pm_y$, $y\in\N_0$.

\begin{definition}
\rm
\label{def:proc}
Fix $\ell, x \in \N_0$. Define a two-species \emph{branching process}
\begin{equation}
(M^+, M^-) = (M^+_y, M^-_y)_{y\in\N_0}
\end{equation} 
with law $\cP_{\ell, x}$ as follows:
\begin{itemize}
\item 
At generation 0 there are $\ell$ individuals, which are divided by fair 
coin tossing into two subpopulations, labelled $+$ and $-$.
\item 
Each subpopulation evolves independently as a \emph{critical Galton-Watson 
branching process with a geometric offspring distribution}, denoted by $\Geo_0
(\frac{1}{2})$ and given by $\Geo_0(\frac{1}{2})(i)=2^{-(i+1)}$, $i\in\N_0$.
\item 
If $x\in\N$, then there is additional immigration of a $\Geo_0(\frac{1}{2})$-distributed 
number of individuals in the $+$ subpopulation, at each generation $1,\dots,x$ 
(equivalently, the generations $0,\dots,x-1$ have an additional ``hidden'' individual, 
which is not counted but produces offspring).
\item 
Define $M^\pm_y$ as the size of the $\pm$ subpopulation in the $y$-th generation.
\end{itemize}
Define the total population size
\begin{equation} 
\label{eq:Xi}
\Xi = \sum_{y\in\N_0} (M^+_y + M^-_y)
\end{equation}
and note that $\Xi < \infty$ a.s.\ because a critical Galton-Watson process eventually 
dies out.
\end{definition}

We can now state the main result of this section. Abbreviate $L_0=L_n(0)$. 

\begin{theorem} 
\label{th:bra}
Fix $\ell, n , x \in \N_0$ such that $0 \le \ell \le \frac{1}{2} n$, $0 \le x \le n$ and $x-n$ is even.
The edge-crossing numbers $(M^+, M^-)$ of the simple random walk defined in \eqref{eq:edge} 
conditionally on $\{L_0 = \ell, S_n = x\}$ have the same joint distribution as the branching process 
with law $\cP_{\ell,x}$ defined in Definition~{\rm \ref{def:proc}} conditionally on $\{\Xi = \frac{1}{2}(n-x)\}$. 
In formulas,
\begin{equation} 
\label{eq:topr}
\begin{aligned}
&\bP \Big( (M^+, M^-) = (m^+, m^-),\, L_0 = \ell, \, S_n = x \Big) \\
&\qquad = \cP_{\ell, x} \Big( (M^+, M^-)=(m^+, m^-),\, \Xi = \tfrac{1}{2}(n-x) \Big) ,
\end{aligned}
\end{equation}
for all sequences $(m^+, m^-) = (m^+_y, m^-_y)_{y\in\N_0} \in (\N_0 \times \N_0)^{\N_0}$.
\end{theorem}

\begin{remark}\rm
Taking the scaling limit of \eqref{eq:topr} we obtain the famous Ray-Knight relation between Brownian 
motion local time and squared Bessel processes (see Revuz and Yor~ \cite{RY91}). We refer to 
T\'oth \cite{T95,T96} for analogous relations involving more general processes, arising in the context 
of self-interacting random walks.
\end{remark}

Before proving Theorem~\ref{th:bra}, 
we note that the transition kernel of a critical Galton-Watson 
branching process with geometric offspring distibution is given by the matrix $Q(i,j)$, 
$i,j \in \N_0$, defined in \eqref{eq:def.Qij}. In fact, if $(\xi_n)_{n\in\N}$ are i.i.d.\ $\Geo_0(\frac{1}{2})$ 
random variables, then 
\begin{equation}
\label{eq:Q}
Q(i,j) = \begin{cases}
\ind_{\{j = 0\}}, & \text{if } i=0, \ j \in \N_0, \\
\rule{0pt}{2em}
\displaystyle \binom{i+j-1}{i-1}
\left( \frac12 \right)^{i+j}= P(\xi_1 + \ldots + \xi_i = j),
& \text{if }
i \in \N, \ j \in \N_0,
\end{cases}
\end{equation}
In the presence of immigration, the transition kernel becomes $Q(i+1, j)$.

By Definition~\ref{def:proc}, $(M^+, M^-)$ is a Markov chain on $\N_0 \times \N_0$
that is not time-homogeneous whenever $x \ne 0$ (due to the immigration). The initial 
distribution of this Markov chain is
\begin{equation}
\label{eq:start}
\cP_{\ell,x} \big( (M^+_0, M^-_0) = (a,b) \big) = \rho_\ell(a,b)
\quad \text{ with } \quad \rho_\ell(a,b) = \binom{\ell}{a} \, \frac{1}{2^\ell} 
\, \ind_{\{a+b = \ell\}}, \quad a,b \in \N_0,
\end{equation}
while the transition kernel factorizes, i.e., it is the product of its marginals, because conditionally 
on $(M^+_0, M^-_0)$ the two components $(M^+_y)_{y\in\N}$ and $(M^-_y)_{y\in\N}$ evolve 
independently, with marginal transition kernels
\begin{gather}
\label{eq:trans}
\begin{split}
\cP_{\ell,x} \big( M^+_{y+1}  = j \mid M^+_{y} = i \big) 
&= Q(i+1,j) \, \ind_{\{y < x\}} + Q(i,j) \, \ind_{\{y \ge x\}},  \\
\cP_{\ell,x} \big( M^-_{y+1} = j \mid  M^-_{y} = i \big) 
&= Q(i,j), \qquad i,j\in\N_0.
\end{split}
\end{gather}

We are now ready to give the proof of Theorem~\ref{th:bra}.

\begin{proof}
Note that both sides of \eqref{eq:topr} vanish, unless the sequences $m^\pm$ satisfy the conditions
\begin{equation} 
\label{eq:condmpm}
m_0^+ + m_0^- = \ell \,, \qquad
\sum_{y\in\N_0} (m_y^+ + m_y^-) = \frac{1}{2}(n-x).
\end{equation}
The first condition holds because $L_0 = M_0^+ + M_0^-$ for the random walk (each visit to zero 
is preceded by a crossing of either $(0,1)$ or $(-1,0)$), while $\cP_{\ell, x}(M_0^+ + M_0^- = \ell) = 1$ 
for the branching process by construction. Analogously, the second condition in \eqref{eq:condmpm} 
holds for the branching process by the definition of $\Xi$ in \eqref{eq:Xi}, while it holds for the random 
walk, because the total number of steps $n$ equals the total number of upward or downward crossings, 
which is given by $2 \sum_{y\in\N_0} (M^+_y + M^-_y) + x$ (recall that the last upward crossing of a 
bond in the stretch $\{0, \ldots, x\}$ is not counted in $M^+_y$).

Henceforth we fix two sequences $(m^+, m^-) = (m^+_y, m^-_y)_{y\ge 0} \in (\N_0 \times \N_0)^{\N_0}$ 
that satisfy \eqref{eq:condmpm}. Below we will show that the number of simple random walk paths $(S_1, 
\ldots, S_n)$ contributing to the event $\{(M^+, M^-) = (m^+, m^-), \, S_n = x, \, L_0 = \ell\}$ equals
\begin{equation} 
\label{eq:combin}
\binom{\ell}{m_0^+} \, \prod_{y=0}^{x-1} C(m_y^+ + 1, m_{y+1}^+)
\, \prod_{y \ge x} C(m_y^+, m_{y+1}^+) \, \prod_{y \ge 0} C(m_y^-, m_{y+1}^-) \,,
\end{equation}
where
\begin{equation} 
\label{eq:C}
C(0,j) = \ind_{\{j=0\}}, \quad j \in \N_0, \qquad
C(i,j) = \binom{i+j-1}{i-1} \quad i \in\N,\,j\in \N_0.
\end{equation}
The first product in \eqref{eq:combin} is $1$ when $x=0$, by convention. Note that $m^\pm$ have 
a finite sum by \eqref{eq:condmpm}, and hence are eventually zero: 
$m^\pm_y = 0$ for large enough $y$. Since $C(0,0)=1$, this means that the products in 
\eqref{eq:combin} are finite.

We can now prove \eqref{eq:topr}. The probability in the left-hand side of \eqref{eq:topr} is 
obtained after dividing \eqref{eq:combin} by $2^n$, which is the total number of random walk 
paths. Recalling \eqref{eq:def.Qij}, \eqref{eq:start} and \eqref{eq:condmpm}, we obtain
\begin{equation} 
\label{eq:coin}
\begin{split}
&\bP \Big( (M^+, M^-) = (m^+, m^-),\, L_0 = \ell, \, S_n = x \Big) \\
& \qquad = \rho_\ell(m_0^+, m_0^-) 
\, \prod_{y=0}^{x-1} Q (m_y^++1, m_{y+1}^+)
\, \prod_{y \ge x} Q(m_y^+, m_{y+1}^+) \, \prod_{y \ge 0} Q(m_y^-, m_{y+1}^-),
\end{split}
\end{equation}
which is precisely the probability in the right-hand side of \eqref{eq:topr}, by the Markov property 
of the process $(M^+, M^-)$ under the law $\cP_{\ell,x}$ (recall \eqref{eq:start}-\eqref{eq:trans}).

It remains to prove \eqref{eq:combin}. Observe that $C(i,j)$ in \eqref{eq:C} equals the number 
of ways in which $j$ objects can be allocated to $i$ boxes, i.e., the number of sequences
$(a_1, \ldots, a_i) \in (\N_0)^i$ satisfying $a_1 + \ldots + a_i = j$. As to the random walk, the 
crossing number $m_0^-$ counts the number of excursions below $0$, while the crossing number 
$m_1^-$ counts the number of excursions below $-1$. The key observation is that each excursion 
below $-1$ is included in precisely one excursion below $0$. Therefore the number of ways in 
which the $m_1^-$ excursions below $-1$ can be ``allocated'' to the $m_0^-$ excursions below 
$0$ equals $C(m_0^-, m_1^-)$. Iterating this argument, we see that the last product in \eqref{eq:combin} 
counts the number of random walk paths in the negative half-plane that are compatible with the 
given bond crossing-numbers $(m^-_y)_{y\in\N_0}$.

For the positive part of the random walk path there is one difference. When  $x\in\N$, each of the 
$m^+_1$ excursions above $+1$ can be allocated not only to the $m^+_0$ excursions above $0$, 
but also to the last incomplete excursion leading to $x$. This explains the presence of ``$+1$'' in 
the combinatorial factor $C(m_0^+ + 1, m_1^+)$ in \eqref{eq:combin}. This holds until level $x$, 
while above level $x$ the combinatorial factor $C(m^+_y, m^+_{y+1})$ applies. The first two products 
in \eqref{eq:combin} therefore count the number of random walk paths in the positive half-plane that 
lead to $x$ and are compatible with the given bond crossing-numbers $(m^+_y)_{y\in\N_0}$.

It remains to combine the positive and the negative parts of the random walk path we have just built.
This can be done by alternating the $m_0^+$ positive excursions and the $m^-_0$ negative excursions 
in an arbitrary way, while preserving their relative order. Since $m^+_0 + m^-_0 = \ell$, this can be 
done in $\binom{\ell}{m_0^+}$ ways, which leads to the first factor in \eqref{eq:combin}.
\end{proof}

\begin{remark}
\rm
\label{rem:tree}
One way to visualize \eqref{eq:combin} is to identify a random walk excursion with a planar rooted 
tree (the random walk path traces out the ``external boundary'' of the tree). With this identification, 
the bond-crossing numbers $(m_y)_{y\in\N_0}$ represent the number of branches of the tree at 
level $y\in\N_0$, and the number of such trees is given by $\prod_{i\in\N_0} C(m_i,m_{i+1})$.
\end{remark}


\subsection{From the annealed partition function to functionals of local times}
\label{ss:loctimerepr}

The first step in our analysis  of the free energy is to rewrite the annealed partition function 
as the partition function of the simple random walk weighted by a functional of its local times 
$L_n(y)$ defined in \eqref{eq:localtime}. To that end we define, for $\ell\in\N_0$ and 
$(\delta,\beta) \in \cQ$,
\begin{equation}
\label{eq:def.Gbeta.p}
G_{\delta,\beta}(\ell) 
= \log \bbE^{\delta}\left[e^{-\beta \Omega_\ell^2}\right].
\end{equation}
Note that $\Omega_0 = 0$, so that $G_{\gd,\gb}(0)=0$.

\begin{lemma}
\label{lem:id_part_fct} 
For $n\in\N$ and $(\delta,\beta) \in \cQ$,
\begin{equation}
\label{eq:identity}
\bbZ_n^{\delta,\beta} = \bE\left[\exp\left\{\sum_{x\in\bbZ} G_{\delta,\beta}(L_n(x))\right\}\right].
\end{equation}
\end{lemma}

\begin{proof}
Rewrite \eqref{eq:def.annpartfunc} as
\begin{equation}
\bbZ_{n}^{\delta,\beta} 
= (\bbE^{\delta} \times \bE) \left[ \prod_{x\in \bbZ} 
\exp \left\{ -\beta \left( \sum_{i=1}^n \omega_i \ind_{\{S_i = x\}} \right)^2 \right\}  \right] 
= \bE \left[ \,\prod_{x\in \bbZ} \bbE^{\delta} \left[ e^{ -\beta \Omega_{L_n(x)}^2} \right] \right],
\end{equation}
and use \eqref{eq:def.Gbeta.p}. 
\end{proof}

Depending on which phase we are working in, it will be convenient to also use the function 
$G^*_{\delta,\beta}(\ell)$ defined by
\begin{equation}
\label{eq:rel.G.staralt}
G^*_{\delta,\beta}(\ell) = G_{\delta,\beta}(\ell) - f(\delta)\ell, \qquad \ell \in \N_0,
\end{equation}
which equals \eqref{eq:rel.G.star} by \eqref{eq:def.Gbeta.p} (recall \eqref{eq:Pdeltadef} 
and \eqref{eq:fdeltadef}), and to rewrite Lemma~\ref{lem:id_part_fct} as 
\begin{equation}
\label{eq:identity*}
\bbZ_n^{*,\delta,\beta}  
= (\bbE \times \bE) \left[ e^{\sum_{x\in\Z} \big(\delta \Omega_{L_n(x)} 
- \beta \Omega_{L_n(x)}^2\big)} \right]
= \bE\left[\exp\left\{\sum_{x\in\bbZ} G^*_{\delta,\beta}(L_n(x))\right\}\right]
\end{equation}
with (recall \eqref{fannlimit})
\begin{equation}
\label{eq:ZnZn*rel}
\bbZ_n^{*,\delta,\beta} = e^{-f(\delta)n}\,\bbZ_n^{\delta,\beta}.
\end{equation}
In Appendix~\ref{AppA} we collect some properties of $G^*_{\delta,\beta}$ that will be needed
along the way.

\medskip\noindent
{\bf Example:} If the marginal of $\bbP$ is standard normal, then by direct computation
\begin{equation}
\label{eq:Gstar.gaussian}
G^*_{\delta,\beta}(\ell) = -\tfrac12 \log(1+2\beta\ell) 
+ \tfrac12\,\delta^2\,\frac{\ell}{1+2\beta\ell}.
\end{equation}

\medskip\noindent
{\bf Remark:}
We close this section with the following observation. In Section~{\rm \ref{ss:model}} we argued 
that working with \eqref{eq:def.hamiltonian} rather than \eqref{eq:def.ham} as the interaction 
Hamiltonian amounts to replacing $\beta$ by $2\beta$ and adding a charge bias. Indeed, this 
is immediate from the relation
\begin{equation}
\label{eq:Hamrel}
2  \sum_{1\leq i <  j\leq n} \omega_i \omega_j\, \ind_{\{S_i=S_j\}}
=  \sum_{x\in \Z^d}  \left(\sum_{i=1}^n \omega_i\, \ind_{\{S_i=x\}}\, \right)^2 
- \sum_{i=1}^n \omega_i^2.
\end{equation}
For the annealed model, the last sum is not constant (unless $\omega_i = \pm 1$). To handle this, 
define $\bar{\bbP}^\beta$ as the product law with marginal given by
\begin{equation}
\bar{\bbP}^\beta(d\omega_1) 
= \frac{e^{\beta\omega_1^2}\,\bbP(\dd\omega_1)}{\bar{M}(\beta)},
\qquad \bar{M}(\beta)=\bbE(e^{\beta\omega_1^2}),
\end{equation}
where we need to assume that $\bar{M}(\beta)<\infty$ for all $\beta \in (0,\infty)$. We put
\begin{equation}
\bar{G}_{\delta,\beta}(\ell) = \log \bar{\bbE}^\beta 
\left[e^{\delta \Omega_\ell -\beta\Omega_\ell^2}\right],
\end{equation}
which is the same as \eqref{eq:rel.G.star} but with $\bar{\bbE}^\beta$ instead of 
$\bbE$, and we define the partition function
\begin{equation}
\bar{\bbZ}_n^{\delta,\beta} = \bE\left[ e^{\sum_{x\in \Z^d} \bar{G}_{\delta,\beta} (L_n(x))}\right].
\end{equation}
Including the last sum in \eqref{eq:Hamrel} amounts to switching from $\bbZ_n^{*,\delta,\beta}$ 
to $\bar{\bbZ}_n^{\delta,\beta}$. As mentioned in Section~{\rm \ref{ss:model}}, in this paper 
we work with the Hamiltonian without the last sum. The reader may check that $\bar{G}_{\delta,
\beta}$ has the same qualitative properties as $G^*_{\delta,\beta}$, so that all the computations 
carried out below can be easily transferred.


\subsection{Grand-canonical representation}
\label{ss:grcanrepr}

To compute the annealed free energy we use the generating function associated with 
the sequence of excess annealed partition functions, i.e.,
\begin{equation}
\label{eq:grand_canonical_trick}
\cZ(\mu,\delta,\beta) = \sum_{n\in\N_0} e^{-\mu n}\,\bbZ_{n}^{*,\delta,\beta},
\qquad \mu \in [0,\infty),
\end{equation}
where $\bbZ_{0}^{*,\delta,\beta} = 1$. The main result of this section is the following matrix 
representation of $\cZ(\mu,\delta,\beta)$. Recall the matrices $A_{\mu,\delta,\beta}(i,j)$ and 
$\tilde{A}_{\mu,\delta,\beta}(i,j)$ defined in \eqref{eq:def.Abeta.p.r}--\eqref{eq:def.Abeta.p.r2}, 
and introduce an extra matrix
\begin{equation}
\label{eq:hatA}
\widehat{A}_{\mu,\delta,\beta}(i,j) 
= e^{-\mu(i+j)+G^*_{\delta,\beta}(i+j)}\, Q(i+1,j), \qquad i, j\in\N_0.
\end{equation}

\begin{proposition}
\label{prop:grcanspec}
For $\mu \in [0,\infty)$ and $(\delta,\beta) \in \cQ$,
\footnote{Given a non-negative matrix $B = B(i,j)$, $i,j \in \N_0$, we define
$\frac{1}{1-B}$ by putting $(\frac{1}{1-B})(i,j) = \sum_{k\in\N_0} B^k(i,j)$, $i,j \in \N_0$, 
with $B^0(i,j) = \ind_{\{i=j\}}$. This is well-defined as a matrix with entries in $[0,\infty]$, 
and if all the entries are finite, then $\frac{1}{1-B}$ is the inverse of $1-B$ (where $1$ 
denotes the identity matrix) and commutes with $B$. Hence we can write $\frac{1+B}{1-B} 
= (1+B)\frac{1}{1-B} = \frac{1}{1-B} (1+B)$ without ambiguity.}
\begin{equation}
\label{Z+form}
\cZ(\mu, \delta, \beta) = \left[\frac{1}{1 - \tilde{A}^{\intercal}_{\mu,\delta,\beta}}\,
\widehat{A}_{\mu,\delta,\beta}\, 
\frac{1 + A_{\mu,\delta,\beta}}{1 - A_{\mu,\delta,\beta}}\,
\frac{1}{1 - \tilde{A}_{\mu,\delta,\beta}}\right](0,0).
\end{equation}
\end{proposition}

\begin{proof}
To lighten the notation, we \emph{suppress the dependence on} $\mu,\delta,\beta$.
Recalling \eqref{eq:identity*}, we can write
\begin{equation}
\cZ = \sum_{n\in\N_0}
e^{-\mu n} \, \sum_{(\ell, x) \in (\N_0)^2} \,
\big(1+\ind_{\{x>0\}}\big) \bE\Bigg[\exp\Bigg\{\sum_{y\in\bbZ} G^*(L_n(y))\Bigg\}
\ind_{\{L_0 = \ell, S_n = x\}} \Bigg] \,,
\end{equation}
where $\ind_{\{x>0\}}$ accounts for the contribution of $\{S_n = -x\}$, which is the same 
as the contribution of $\{S_n = x\}$. Recalling \eqref{eq:LM}, we can write
\begin{equation}
\label{eq:long}
\begin{split}
\cZ &= \sum_{(\ell, x) \in (\N_0)^2} \, \big(1+\ind_{\{x>0\}}\big) \,
\sum_{n\in\N_0} \, \sum_{(m^+,m^-) \in (\N_0 \times \N_0)^{\N_0}} \\
&\bP \big( (M^+, M^-) = (m^+, m^-), \, L_0 = \ell, S_n = x \big) \\
& \qquad \times e^{-\mu n} \, e^{G^*(m_0^+ + m_0^-)} 
\prod_{y=0}^{x-1} e^{G^*(m^+_y + m^+_{y+1} + 1)}
\prod_{y\ge x} e^{G^*(m^+_y + m^+_{y+1})}
\prod_{y\ge 0} e^{G^*(m^-_y + m^-_{y+1})} \,.
\end{split}
\end{equation}
Next we apply Theorem~\ref{th:bra} to rewrite the probability in the right-hand side of \eqref{eq:long} 
as a product of $Q$ matrices, as in \eqref{eq:coin}. In order to do this, we must restrict to sequences 
$(m^+,m^-)$ satisfying the two conditions in \eqref{eq:condmpm}. Thus, we define
\begin{equation}
\begin{split}
\cS_{\ell,x,n} = \bigg\{ (m^+_y, m^-_y)_{y\ge 0} \in (\N_0 \times \N_0)^{\N_0}\colon\, 
m^+_0 + m^-_0 = \ell, \, \sum_{y\in\N_0} (m^+_y + m^-_y) = \tfrac{1}{2}(n-x) \bigg\},
\end{split}
\end{equation}
and restrict the sum over $(m^+,m^-)$ in \eqref{eq:long} to the set $\cS_{\ell,x,n}$. At this point we 
combine the definitions of the matrix $Q$, the function $G^*$ and the factor $\mu$ into the matrix 
$A$ defined in \eqref{eq:def.Abeta.p.r}:
\begin{equation}
A(i,j) = e^{G^*(i+j+1) - \mu(i+j+1)} Q(i+1,j) , \qquad i,j \in \N_0.
\end{equation}
We also introduce an extra matrix $\tilde A'$ by
\begin{equation}
\tilde A'(i,j) = e^{G^*(i+j) - \mu(i+j)} Q(i,j) , \qquad i,j \in \N_0,
\end{equation}
which almost coincides with the matrix $\tilde A$ introduced in \eqref{eq:def.Abeta.p.r2}, the 
difference being that $\tilde A'(0,0)= 1$ while $\tilde A(0,0)= 0$. Altogether, we can rewrite 
\eqref{eq:long} as (recall \eqref{eq:start})
\begin{equation} 
\label{eq:long2}
\begin{split}
\cZ = & \sum_{(\ell, x) \in (\N_0)^2} \, \big(1+\ind_{\{x>0\}}\big)
\sum_{n\in\N_0} \, \sum_{(m^+,m^-) \in \cS_{\ell,x,n}} \\
& \qquad \ e^{G^*(\ell) - \mu \ell} \, \rho_\ell(m_0^+, m_0^-) 
\prod_{y=0}^{x-1} A(m^+_y, m^+_{y+1})
\prod_{y\ge x} \tilde A'(m^+_y, m^+_{y+1})
\prod_{y\ge 0} \tilde A'(m^-_y, m^-_{y+1}).
\end{split}
\end{equation}

Next we introduce the set of sequences $m^\pm$ that are \emph{eventually zero} and satisfy 
$m^+_0 + m^-_0 = \ell$:
\begin{equation}
\cT_{\ell} = \bigg\{ (m^+_y, m^-_y)_{y\in\N_0} \in (\N_0 \times \N_0)^{\N_0}\colon\,
m^+_0 + m^-_0 = \ell,\, m^\pm_y = 0 \text{ for large enough $y$} \bigg\},
\end{equation}
and observe that, for every fixed $x \in \N_0$,
\begin{equation}
\bigcup_{n\in\N_0} \cS_{\ell,x,n} = \cT_\ell.
\end{equation}
The inclusion $\cS_{\ell,x,n} \subseteq \cT_\ell$ is obvious because sequences in $\cS_{\ell, x, n}$ 
are integer-valued with a finite sum and hence are eventually zero. Conversely, if $m^\pm$ are 
eventually zero, then $\sum_{y\in\N_0} (m^+_y + m^-_y) = \tfrac{1}{2}(n-x)$ for some $n\in\N_0$.
This means that
\begin{equation} 
\label{eq:sums}
\begin{split}
\sum_{n\in\N_0} \,
& \sum_{(m^+,m^-) \in \cS_{\ell,x,n}} \big(\ldots\big)
= \sum_{(m^+,m^-) \in \cT_{\ell}} \big(\ldots\big) \\
& = \sum_{(m^+_0, \ldots, m^+_{x-1}) \in (\N_0)^{x-1}} \;
\sumtwo{(t, s) \in (\N_0)^2}{t \ge x} \;
\sumtwo{(m^+_{x}, \ldots, m^+_{t-1}) \in \N^{t-x}}{(m^-_0, \ldots, m^-_{s-1}) \in \N^{s}}
\ind_{\{m_0^+ + m_0^- = \ell\}} \big(\ldots\big),
\end{split}
\end{equation}
where the last line provides a convenient parametrization of the set $\cT_\ell$:
\begin{itemize}
\item 
Sum over all possible values of $(m^+_0, \ldots, m^+_{x-1})$ (when $x \in\N$).
\item 
Denote by $t$ the smallest value of $y \geq x$ for which $m^+_y = 0$. Likewise denote 
by $s$ the smallest value of $y \geq 0$ for which $m^-_y=0$, so that by construction 
$(m^+_{x}, \ldots, m^+_{t-1})$ and $(m^-_0, \ldots, m^-_{s-1})$ can only take values 
in $\N$.
\item 
It is implicit that $m^+_y = 0$ for $y \ge t$ and $m^-_y = 0$ for $y \ge s$.
\end{itemize}

We now apply \eqref{eq:sums} to \eqref{eq:long2}. We further split $\cZ = \cZ_0 + \cZ_+$, 
where $\cZ_0$ is the contribution of the single term $x=0$, while $\cZ_+$ is the sum 
$\sum_{x\in\N} (\cdots)$. Observing that
\begin{equation}
\tfrac12 \sum_{\ell\in\N_0} e^{G^*(\ell) - \mu \ell} \, \rho_\ell(m_0^+, m_0^-)
= \widehat A(m_0^-, m_0^+)
\end{equation}
by \eqref{eq:start} and \eqref{eq:hatA}, we have (with $m^+_t = m^-_s = 0$ by convention)
\begin{equation} 
\label{eq:long3a}
\begin{split}
\cZ_0 =
& \sum_{(t, s) \in (\N_0)^2} \; 
\sumtwo{(m^+_{0}, \ldots, m^+_{t-1}) \in \N^{t}}{(m^-_0, \ldots, m^-_{s-1}) \in \N^{s}}
\widehat A(m_0^-, m_0^+) \,
\prod_{y= 0}^{t-1} \tilde A'(m^+_y, m^+_{y+1})
\prod_{y = 0}^{s-1} \tilde A'(m^-_y, m^-_{y+1})
\end{split}
\end{equation}
and
\begin{equation} 
\label{eq:long3b}
\begin{split}
\cZ_+ = \, 2 \sum_{x \in \N} \,
& \sumtwo{(t, s) \in (\N_0)^2}{t \ge x} \;
\sum_{(m^+_0, \ldots, m^+_{x-1}) \in (\N_0)^{x-1}} \;
\sumtwo{(m^+_{x}, \ldots, m^+_{t-1}) \in \N^{t-x}}{(m^-_0, \ldots, m^-_{s-1}) \in \N^{s}}\\
& \widehat A(m_0^-, m_0^+) \,
\prod_{y=0}^{x-1} A(m^+_y, m^+_{y+1})
\prod_{y= x}^{t-1} \tilde A'(m^+_y, m^+_{y+1})
\prod_{y = 0}^{s-1} \tilde A'(m^-_y, m^-_{y+1}).
\end{split}
\end{equation}

Finally, we observe that we can replace $\tilde A'(i,j)$ by $\tilde A(i,j)$ in 
\eqref{eq:long3a}--\eqref{eq:long3b}, because for $i\in\N$ these two matrices coincide. After 
this replacement, the ranges $\N^{t-x}$ and $\N^{s}$ in the sums can be replaced by 
$(\N_0)^{t-x}$ and $(\N_0)^{s}$, respectively, because $\tilde A(0,j) = 0$. This leads us 
to the desired matrix product representations:
\begin{equation}
\label{eq:long4a}
\begin{split}
\cZ_0 = \sum_{(t, s) \in (\N_0)^2} \; 
\big[ (\tilde A^\intercal)^s \, \widehat A \, \tilde A^t \big](0,0)
= \bigg[ \frac{1}{1- \tilde A^\intercal} \, \widehat A \,
\frac{1}{1 - \tilde A} \bigg] (0,0)
\end{split}
\end{equation}
and
\begin{equation} 
\label{eq:long4b}
\begin{split}
\cZ_+ &= 2 \sum_{x \in \N} \,
\sumtwo{(t, s) \in (\N_0)^2}{t \ge x} \; 
\big[ (\tilde A^\intercal)^s \, \widehat A \, A^x \, \tilde A^{t-x} \big](0,0)\\
&= \bigg[ \frac{1}{1- \tilde A^\intercal} \, \frac{A}{1 - A}\, \widehat A \, \frac{1}{1 - \tilde A} 
+  \frac{1}{1- \tilde A^\intercal} \, \widehat A \,  \frac{A}{1 - A}\, \frac{1}{1 - \tilde A} 
\bigg] (0,0)\\ 
&= \bigg[ \frac{1}{1- \tilde A^\intercal} \, \widehat A \, \frac{2 A}{1 - A} \, \frac{1}{1 - \tilde A} \bigg] (0,0).
\end{split}
\end{equation}
Summing these two formulas, we obtain the formula \eqref{Z+form}.
\end{proof}

\noindent
Note that $\widehat A_{\mu,\delta,\beta}$ plays only a minor role in \eqref{Z+form}: the
divergence of $\cZ(\mu,\delta,\beta)$ is controlled by $A_{\mu,\delta,\beta}$ and 
$\tilde A_{\mu,\delta,\beta}$. 

\medskip
For later use we state a version of Proposition~\ref{prop:grcanspec} for bridges. Namely, let
\begin{equation}
\cZ_\mathrm{bridge}(\mu,\delta,\beta) = \sum_{n\in\N_0} e^{-\mu n}\,
\bbZ_{n,\mathrm{bridge}}^{*,\delta,\beta}, \qquad \mu \in [0,\infty),
\end{equation} 
with
\begin{equation}
\label{partfuncbridge}
\bbZ_{n,\mathrm{bridge}}^{*,\delta,\beta}
= \bE\left[\exp\left\{\sum_{y\in \Z} G^*_{\delta,\beta}(L_n(y))\right\} 
\ind_{\big\{0< S_i \leq  S_n\,\,\forall\,0<i<n\big\}}\right].
\end{equation}

\begin{lemma}
\label{lem:grcanspecloop}
For $\mu \in [0,\infty)$ and $(\delta,\beta) \in \cQ$,
\begin{equation}
\label{Zbridgeform}
\cZ_\mathrm{bridge}(\mu,\delta, \beta) 
= \left[\frac{A_{\mu,\delta,\beta}}{1 - A_{\mu,\delta,\beta}}\right](0,0).
\end{equation}
\end{lemma}

\begin{proof}
Because of the bridge condition, the term with $\cZ_0$ is absent, while in $\cZ_+$  
only the part with the $m_y^+$'s survives (without $\hat{A}$).
\end{proof}


\subsection{Spectral analysis of the relevant matrices} 
\label{subsec:eigenvalue}

In this section we prove some properties of the matrices $A_{\mu,\delta,\beta}$ in 
\eqref{eq:def.Abeta.p.r} and $\tilde A_{\mu,\delta,\beta}$ in \eqref{eq:def.Abeta.p.r2},
viewed as operators from $\ell_2(\N_0)$ into itself. These will be needed to exploit 
Propositions~\ref{prop:grcanspec}--\ref{lem:grcanspecloop}.

\begin{proposition} 
\label{pr:prop.A.lambda.mu}
For all $(\mu,\delta,\beta)\in [0,\infty)\times \cQ$, the matrix $A_{\mu,\delta,\beta}\colon\,
\ell_2(\N_0)\mapsto \ell_2(\N_0)$ has strictly positive entries, is symmetric, and is 
Hilbert-Schmidt.
\end{proposition}

\begin{proof}
It is obvious that $A_{\mu,\delta,\beta}$ has strictly positive entries and is symmetric.
The Hilbert-Schmidt norm of the operator $A_{\mu,\delta,\beta}$ is defined by
\begin{equation}
\label{eq:Amubetadelta.HS}
\|A_{\mu,\delta,\beta}\|_{\rm HS} =  \left( \sum_{i,j\in\N_0} [A_{\mu,\delta,\beta}(i,j)]^2 \right)^{1/2}.
\end{equation}
By \eqref{eq:def.Qij} and \eqref{eq:app.estG2} in Appendix~\ref{AppA}, we have 
($C$ denotes a generic constant that may change from line to line) 
\begin{equation}
A_{\mu,\delta,\beta}(i,j) \leq \frac{C}{\sqrt{i+j+1}}\, Q(i+1,j), \qquad i,j\in\N_0.
\end{equation}
Observe that, by \eqref{eq:Q}, 
\begin{equation}
\label{eq:QandP}
Q(i+1,j) = \tfrac12 \binom{i+j}{i} \frac{1}{2^{i+j}}
= \tfrac12\,\bP(S_{i+j} = i-j), \qquad i,j\in\N_0,
\end{equation}
where $S=(S_i)_{i\in\N_0}$ is a simple random walk. We may therefore write
\begin{equation}
[A_{\mu,\delta,\beta}(i,j)]^2 \leq \frac{C}{i+j+1}\, \bP(S_{i+j} = i-j)^2 
= \frac{C}{i+j+1}\, \bP^{\otimes 2}(S_{i+j} = S'_{i+j} = i-j),
\end{equation}
where $S'$ is an independent copy of $S$. Using the change of variables $u=i+j$, $v=i-j$, we obtain
\begin{equation}
\begin{aligned}
\sum_{i,j\in\N_0} [A_{\mu,\delta,\beta}(i,j)]^2 
&\leq C \sum_{u\in\N_0,\,v\in \Z} \frac{1}{1+u}\, \bP^{\otimes 2}(S_u = S'_u = v)\\
& = C \sum_{u \in \N_0} \frac{1}{1+u}\, \bP^{\otimes 2}(S_u = S'_u) 
\leq C \sum_{u \in \N} u^{-3/2} < \infty,
\end{aligned}
\end{equation}
which proves that $\|A_{\mu,\delta,\beta}\|_{\rm HS}<\infty$. Thus, $A_{\mu,\delta,\beta}$ 
maps $\ell_2(\N_0)$ into itself. 
\end{proof}
Adapting the proof of Proposition~\ref{pr:prop.A.lambda.mu}, we see that 
$\hat{A}_{\mu,\delta,\beta}$ and $\tilde{A}_{\mu,\delta,\beta}$ also map $\ell_2(\N_0)$ 
into itself.

Let $\lambda_{\delta,\beta}(\mu)$ be the spectral radius of $A_{\mu,\delta,\beta}$. By 
Proposition~\ref{pr:prop.A.lambda.mu}, $A_{\mu,\delta,\beta}$ is Hilbert-Schmidt, and 
hence is compact (see Dunford and Schwartz~\cite[XI.6, Theorem 6]{DS64}). Therefore 
$\lambda_{\delta,\beta}(\mu)$ is also the largest eigenvalue of $A_{\mu,\delta,\beta}$ 
(see Kato~\cite[V.2.3.]{KA95}), and admits the Rayleigh characterization (see Dunford
and Schwarz~\cite[X.4]{DS64})
\begin{equation}
\label{eq:lambda.var.rep}
\gl_{\delta,\beta}(\mu) = \suptwo{u\in \ell_2(\bbN_0)}{\| u \|_2 = 1} 
\langle u,A_{\mu,\delta,\beta}\, u \rangle
= \suptwo{u\in \ell_2(\bbN_0)}{\| u \|_2 = 1} 
\sum_{i,j\in \N_0}  A_{\mu,\delta,\beta}(i,j)\, u_i\, u_j.
\end{equation}
Since the entries of $A_{\mu,\delta,\beta}$ are strictly positive, we have
\begin{equation}
\label{eq:lambda.var.rep2}
\gl_{\delta,\beta}(\mu)= \suptwo{u\in \ell_2(\bbN_0)\colon u\geq 0}{\| u \|_2 = 1} 
\langle u, A_{\mu,\delta,\beta} u\rangle,
\end{equation}
where $u\geq 0$ means that all coordinates of $u$ are non-negative. Moreover, by the 
Perron-Frobenius Theorem, there exists an eigenvector $\nu_{\mu,\delta,\beta}\in
\ell_2(\N_0)$ with $\|\nu_{\mu,\delta,\beta}\|_2=1$, associated with $\lambda_{\delta,\beta}(\mu)$ 
with multiplicity one, whose coordinates are all strictly positive (see Baillon, Cl\'ement, Greven 
and den Hollander~\cite[Lemma 9]{BCGdH94}).

\medskip
We next state the regularity and the monotonicity of $(\mu,\delta,\beta) \mapsto \lambda_{\delta,\beta}
(\mu)$ on $[0,\infty) \times \cQ$.

\begin{proposition}
\label{prop:regl} 
The following properties hold:\\
{\rm (i)} $(\mu,\delta,\beta) \mapsto \lambda_{\delta,\beta}(\mu)$ is finite, jointly continuous and 
log-convex on $[0,\infty) \times \cQ$, and analytic on $(0,\infty)^3$.\\
{\rm (ii)} For $(\delta,\beta) \in \cQ$, $\mu \mapsto \lambda_{\delta,\beta}(\mu)$ is strictly 
decreasing and strictly log-convex on $[0,\infty)$, with $\lim_{\mu\to\infty} \lambda_{\delta,\beta}
(\mu) = 0$.\\
{\rm (iii)} For $(\mu,\delta) \in [0,\infty)^2$, $\beta \mapsto \lambda_{\delta,\beta}(\mu)$ is 
strictly decreasing on $(0,\infty)$.
\end{proposition}

\begin{proof}
(i) By Proposition~\ref{pr:prop.A.lambda.mu}, $A_{\mu,\delta,\beta}$ has strictly positive entries, 
and as an operator on $\ell_2(\N_0)$ it is positive, irreducible and Hilbert-Schmidt, and hence compact. 
Consequently, we can apply the Perron-Frobenius Theorem to obtain that $\lambda_{\delta,\beta}(\mu)$ 
is a simple eigenvalue of  $A_{\mu,\delta,\beta}$. Since each entry of $A_{\mu,\delta,\beta}$ is 
continuous and analytic in $(\mu,\delta,\beta)$ on $[0,\infty) \times \cQ$, we can apply Crandall and 
Rabinowitz~\cite[Lemma 1.3]{CR73} to get the claimed continuity and analyticity of $\lambda_{\delta,
\beta}(\mu)$. To get log-convexity, use the variational characterization of $\lambda_{\delta,\beta}(\mu)$ 
in \eqref{eq:lambda.var.rep2} and note that: (a) $\mu \mapsto A_{\mu,\delta,\beta}(i,j)$ is log-convex for 
each $i,j\in \N_0$, (b) the sum and the product of two log-convex functions is log-convex, (c) the 
supremum of convex functions is convex. (Use that a log-convex function can be written as a 
supremum over log-linear functions.)\\
(ii) Pick  $0 \leq \mu_1<\mu_2$ and set $\nu_i=\nu_{\mu_i,\delta,\beta}$, $i=1,2$. Since $\nu_2$ 
has strictly positive entries and
\begin{equation}
A_{\mu_1,\delta,\beta}(i,j) > A_{\mu_2,\delta,\beta}(i,j), \qquad i,j\in\N_0, 
\end{equation}
we have 
\begin{equation}
\lambda_{\delta,\beta}(\mu_1) \geq  \langle \nu_2, A_{\mu_1,\delta,\beta}\, \nu_2\rangle 
>  \langle \nu_2, A_{\mu_2,\delta,\beta}\, \nu_2\rangle= \lambda_{\delta,\beta}(\mu_2),
\end{equation}
which yields the first claim. The second claim follows from the fact that $\mu \mapsto \lambda_{\delta,\beta}(\mu)$ is analytic on $(0,\infty)$ and log-convex but not log-linear 
on $[0,\infty)$. The third claim follows from the estimate $\gl_{\gd,\gb}(\mu) = \sup_{u\in \ell_2
(\N_0)\colon\,\|u\|_2 = 1} \langle u, A_{\mu,\gd,\gb} u \rangle \leq e^{-\mu} \gl_{\gd,\gb}(0)$.\\
(iii) Pick  $0 \leq \beta_1<\beta_2$ and set $\nu_i=\nu_{\mu,\delta,\beta_i}$, $i=1,2$.
By Proposition~\ref{pr:Gbetadelta.monotonicity} in Appendix~\ref{AppA}, we have
\begin{equation}
A_{\mu,\delta,\beta_1}(i,j) > A_{\mu,\delta,\beta_2}(i,j), \qquad i,j\in\N_0, 
\end{equation}
and so we can repeat the argument in (ii).
\end{proof}

Two situations emerge, which correspond to the ballistic phase $\cB$, respectively, the subballistic 
phase $\cS$ (see Figs.~\ref{fig-spectral}--\ref{fig-critcurve} and recall \eqref{mudeltabetadef}):
\begin{itemize}
\item 
If $\lambda_{\delta,\beta}(0) \geq 1$, then there is a unique $\mu(\delta,\beta) \in [0,\infty)$ 
for which $\lambda_{\delta,\beta}(\mu(\delta,\beta)) = 1$.
\item 
If $\lambda_{\delta,\beta}(0) < 1$, then there is no $\mu \in [0,\infty)$ for which $\lambda_{\beta,
\delta}(\mu) = 1$, and we set $\mu(\delta,\beta) = 0$.
\end{itemize}

For later use we state the following gap.

\begin{proposition}
\label{prop:gap}
$\tilde\lambda_{\delta,\beta}(\mu)<\lambda_{\delta,\beta}(\mu)$ for all $(\mu,\gd,\gb) 
\in [0,\infty)\times \cQ$.
\end{proposition}

\begin{proof}
Let
\begin{equation}
\| A_{\mu,\delta,\beta} \|_{\rm op} 
= \sup_{u\in\ell_2(\N_0)\setminus \{0\}} \frac{\|A_{\mu,\delta,\beta}u\|_2}{\|u\|_2}.
\end{equation}
We prove the strict inequality by showing that
\begin{equation}
\label{eq:rel.lambda2}
\tilde{\gl}_{\gd,\gb}(\mu) \leq \| \tilde{A}^2_{\mu,\gd,\gb}\|_{\rm op}^{1/2} 
< \| A_{\mu,\gd,\gb}\|_{\rm op} = \gl_{\gd,\gb}(\mu).
\end{equation}
The first inequality in \eqref{eq:rel.lambda2} is a consequence of the following relation:
\begin{equation} 
\label{eq:spr.Hadamard}
\spr(M) = \lim_{n\to \infty} \|M^n\|_{\rm op}^{1/n} = \inf_{n\in\N} \|M^n\|_{\rm op}^{1/n},
\end{equation}
where $M$ is any bounded linear operator and $\spr(M)$ is the spectral radius of $M$ 
(Dunford and Schwartz~\cite[VII.3]{DS64}). 

We now prove the second inequality in \eqref{eq:rel.lambda2}. To that end, define
\begin{equation}
\| A_{\mu,\delta,\beta} \|_{\rm op, *} 
= \suptwo{u\in\ell_2(\N_0)\setminus\{0\}}{u(0) = 0} \frac{\|A_{\mu,\delta,\beta}u\|_2}{\|u\|_2}.
\end{equation}
For $u\in \ell_2(\N_0)$, we have
\begin{equation}
(\tilde{A}_{\mu,\delta,\beta}^2u)(i) 
= \sum_{j\in \N_0} A_{\mu,\delta,\beta}(i-1,j) (\tilde{A}_{\mu,\delta,\beta}u)(j), \quad i \in\N, 
\qquad (\tilde{A}_{\mu,\delta,\beta}^2u)(0) = 0.
\end{equation}
Since $(\tilde{A}_{\mu,\delta,\beta}u)(0)=0$, this yields
\begin{equation}
\|\tilde A_{\mu,\delta,\beta}^2\|_{\rm op} 
\leq \|A_{\mu,\delta,\beta}\|_{\rm op,*} \|\tilde A_{\mu,\delta,\beta}\|_{\rm op} 
\leq \|A_{\mu,\delta,\beta}\|_{\rm op,*} \|A_{\mu,\delta,\beta}\|_{\rm op}.
\end{equation}
It therefore remains to show that $\|A_{\mu,\delta,\beta}\|_{\rm op,*} < \|A_{\mu,\delta,\beta}\|_{\rm op}$. 
To that end, note that, since $A_{\mu,\delta,\beta}$ is compact and symmetric, there exist eigenvalues 
$(\gl_k)_{k\in\N_0}$ in $\R$ and associated eigenvectors $(\nu_k)_{k\in\N_0}$ in $\ell_2(\N_0)$ such 
that
\begin{equation} \label{eq:decospect}
\gl_0 = \gl_{\gd,\gb}(\mu), \qquad |\gl_k| \leq |\gl_1| < \gl_0, \quad k \in\N,
\qquad \langle \nu_k, \nu_l \rangle = \delta_{kl}, \quad k,l\in\N_0,
\qquad \nu_0(0) > 0
\end{equation}
(see Kato~\cite[Theorem 6.38, Section 6.9]{KA95} and Zerner~\cite[Theorem 1]{Z87}). 
Let $u\in\ell_2(\N_0)$ be such that $\|u\|_2 = 1$. Write $u = \sum_{k\in\N_0} c_k \nu_k$, 
with $c_k = \langle u, \nu_k \rangle$ and $\sum_{k\in\N_0} c_k^2 = 1$. Then
\begin{equation}
\label{rel1}
\|A_{\mu,\delta,\beta}u\|_2^2 
= \sum_{k\in\N_0} \gl_k^2 |c_k|^2 \leq \gl_0^2 c_0^2 + |\gl_1|^2 (1-c_0^2).
\end{equation}
If $u(0) = 0$, then by the Cauchy-Schwarz inequality
\begin{equation}
\label{rel2}
c_0^2 = \langle u, \nu_0 \rangle^2 = \Big( \sum_{i\in \N} u(i) \nu_0(i) \Big)^2 
\leq \sum_{i\in \N} [\nu_0(i)]^2 = 1-[\nu_0(0)]^2.
\end{equation}
Since $\nu_0(0) > 0$ and $|\gl_1| < \gl_0$, we get from \eqref{rel1}--\eqref{rel2}
\begin{equation}
\|A_{\mu,\delta,\beta}\|_{\rm op,*} \leq \sqrt{\gl_0^2(1-[\nu_0(0)]^2) 
+ |\gl_1|^2 [\nu_0(0)]^2} < \gl_0 = \|A_{\mu,\delta,\beta}\|_{\rm op}.
\end{equation}

Finally, the last equality in \eqref{eq:rel.lambda2} follows
immediately by the decomposition \eqref{eq:decospect}.
\end{proof}


\subsection{Spectral representation of the generating function}
\label{ss:spec}


\subsubsection{$\bullet$ Case $\mu > \mu(\delta,\beta)$}

We show that $\cZ(\mu,\delta,\beta)$ in \eqref{Z+form} is finite. Indeed, by the definition 
of $\mu(\gd,\gb)$ and Proposition~\ref{prop:regl}(ii), we know that $\gl_{\gd,\gb}(\mu) 
< 1$, and therefore $1-A_{\mu,\gd,\gb}$ is invertible. Moreover, since $A_{\mu,\gd,\gb}$ is 
symmetric, it is also normal (i.e., $\| A_{\mu,\gd,\gb}\|_{\rm op} = \gl_{\gd,\gb}(\mu) < 1$), 
and so the inverse of $1-A_{\mu,\gd,\gb}$ equals $\sum_{n\in\N_0} A^n_{\mu,\gd,\gb}$. 
Thanks to Proposition~\ref{prop:gap} we also have $\tilde{\gl}_{\gd,\gb}(\mu) < 1$. Therefore 
$\cZ(\mu,\delta,\beta)$ is finite.


\subsubsection{$\bullet$ Case $0 \leq \mu < \mu(\delta,\beta)$} 

Suppose that $\mu(\delta,\beta) >0$ (otherwise there is nothing to prove). We show that 
$\cZ(\mu,\delta,\beta)$ in \eqref{Z+form} is infinite. Indeed, by Lemma \ref{lem:grcanspecloop},
\begin{equation}
\cZ(\mu,\gd,\gb) \geq \cZ^{\rm bridge}(\mu,\gd,\gb) = \sum_{x\in \N} (A^x_{\mu,\gd,\gb})(0,0).
\end{equation}
For $i,j\in\N_0$, let
\begin{equation}
\check{A}_{\mu,\gd,\gb}(i,j) = \frac{A_{\mu,\gd,\gb}(i,j)\nu_{\mu,\gd,\gb}(j)}
{\gl_{\gd,\gb}(\mu)\nu_{\mu,\gd,\gb}(i)},
\end{equation}
and observe that
\begin{equation}
(A^x_{\mu,\gd,\gb})(0,0) = [\gl_{\gd,\gb}(\mu)]^x (\check{A}^x_{\mu,\gd,\gb})(0,0).
\end{equation}
Since $\gl_{\gd,\gb}(\mu) \geq 1$, the proof is complete once we show that the sequence 
of positive numbers $\{(\check{A}^n_{\mu,\gd,\gb})(0,0)\}_{n\in\N}$ is bounded away from 
$0$. To that end, observe that $\check{A}$ is a transition matrix on $\N_0$ with
$\nu_{\mu,\gd,\gb}^2$ as invariant probability measure. Therefore, by the renewal 
theorem,
\begin{equation}
\lim_{n\to \infty} (\check{A}^n_{\mu,\gd,\gb})(0,0) = \nu_{\mu,\gd,\gb}(0)^2 >0,
\end{equation}
which completes the proof.


\subsection{Conclusion}
\label{ss:concl}

We are finally ready to conclude the proof of Theorem~\ref{thm:free.energy}. As shown 
above, $\cZ(\mu, \delta,\beta)$ is finite for $\mu > \mu(\delta,\beta)$ and infinite for 
$0 \leq \mu < \mu(\delta,\beta)$. Therefore we have proven that
\begin{equation}
\label{partfunclimsup}
\mu(\delta,\beta) = \limsup_{n\to\infty} \frac{1}{n} \log \bbZ_n^{*,\delta,\beta}. 
\end{equation}
Below we show that
\begin{equation}
\label{partfuncliminf}
\liminf_{n\to\infty} \frac{1}{n} \log \bbZ_n^{*,\delta,\beta} \geq \mu(\delta,\beta).
\end{equation}
Combining \eqref{partfunclimsup}--\eqref{partfuncliminf} and recalling \eqref{eq:ZnZn*rel}, 
we get the spectral representation in \eqref{varc} in Theorem~\ref{thm:free.energy}.
Since $\mu(\delta,\beta) \geq 0$, we also get the lower bound in \eqref{eq:Fineq}
in Theorem~\ref{thm:free.energy}. The fact that $(\delta,\beta) \mapsto F^*(\delta,\beta) 
= \mu(\delta,\beta)$ is convex is immediate from \eqref{eq:identity*}.

Consider the bridge version of the partition function defined in \eqref{partfuncbridge}. The 
sequence $(\log \bbZ_{n,\mathrm{bridge}}^{*,\delta,\beta})_{n\in\bbN}$ is super-additive 
because the concatenation of two bridges is again a bridge. Hence
\begin{equation}
\lim_{n\to\infty} \frac{1}{n} \log \bbZ_{n,\mathrm{bridge}}^{*,\delta,\beta} \quad \text{ exists}.
\end{equation}
It follows from Lemma~\ref{lem:grcanspecloop} that the limit equals $\mu(\delta,\beta)$.
Since $\bbZ_n^{*,\delta,\beta} \geq \bbZ_{n,\mathrm{bridge}}^{*,\delta,\beta}$, this settles 
\eqref{partfuncliminf}.


\section{General properties: proof of the main theorems}
\label{phasetransition}

In Section~\ref{ss:critcurve} we prove the qualitative properties of the excess free energy 
and the critical curve stated in Theorem~\ref{thm:critcurve}. In Section~\ref{ss:jointLDP} 
we prove the large deviation principle for the speed and the charge stated in 
Theorem~\ref{thm:jointLDP}, while in Section~\ref{ss:rfshape} we show how the shape
of the associated rate functions in Figs.~\ref{fig-rfspeed}--\ref{fig-rfcharge} come about.
In Section~\ref{CLTproof} we prove the central limit theorem for the speed and the charge 
stated in Theorem~\ref{thm:CLTspeedcharge}. In Section~\ref{ss:LLN} we prove the law 
of large numbers for the speed and the charge stated in Theorems~\ref{thm:speed}--\ref{thm:charge}.


\subsection{Critical curve}
\label{ss:critcurve}

In this section we prove Theorem~\ref{thm:critcurve}. Fix $\delta\in [0,\infty)$. Clearly, $\beta 
\to F^*(\delta,\beta)$ is non-increasing and convex on $(0,\infty)$ (see \eqref{eq:identity*}), 
and hence is continuous on $(0,\infty)$.

By Theorem~\ref{thm:free.energy}, we know that $F^*(\delta,\beta) \geq 0$. Since $\beta \mapsto 
F^*(\delta,\beta)$ is non-increasing and continuous, there exists a $\beta_c(\delta)=\sup\{\beta\in 
(0,\infty)\colon\,F^*(\delta,\beta)>0\}$ such that $F^*(\delta,\beta)>0$ when $0<\beta<\beta_c(\gd)$ 
and $F^*(\delta,\beta)=0$ when $\beta \geq \beta_c(\gd)$. Since $(\delta,\beta) \mapsto F^*(\delta,
\beta)$ is convex on $\cQ$, the level set $\{(\delta,\beta)\in \cQ \colon\,F^*(\delta,\beta) 
\leq 0\}$ is convex, and it follows that $\delta \mapsto \beta_c(\gd)$ (which coincides with the 
boundary of this level set) is also convex. 

Next we prove that $\beta_c(\delta)\in (0,\infty)$ for $\delta \in (0,\infty)$. By \eqref{eq:app.estG2}, $\lim_{\beta \to \infty} G_{\delta,\beta}^*(\ell)
=-\infty$ for $\delta\in [0,\infty)$ and $\ell\in \N$. Hence, by the definition of the Hilbert-Schmidt norm of 
the operator $A_{\mu,\delta,\beta}$ in \eqref{eq:Amubetadelta.HS}, we have $\lim_{\beta\to \infty} 
\|A_{\mu,\delta,\beta}\|_{\mathrm{HS}}=0$. But $\|A_{\mu,\delta,\beta}\|_{\mathrm{op}} \leq 
\|A_{\mu,\delta,\beta}\|_{\mathrm{HS}}$ and, since $A_{\mu,\delta,\beta}$ is normal, 
$\lambda_{\delta,\beta}(\mu) = \|A_{\mu,\delta,\beta}\|_{\mathrm{op}}$, so that $\lim_{\beta\to \infty} 
\lambda_{\delta,\beta}(0)=0$. Thus, $F^*(\delta,\beta)= \mu(\delta,\beta) =0$ for $\beta$ large 
enough, and so $\beta_c(\delta)<\infty$. Also, observe that $F^*(\delta,0)=-f(\delta)>0$ for 
$\delta \in (0,\infty)$, which yields that $\beta_c(\delta)>0$ for $\delta\in (0,\infty)$. Finally, 
since $F^*(0,\beta) = \mu(0,\beta) = 0$ for $\beta \in (0,\infty)$ (recall 
\eqref{eq:def.Qij}--\eqref{mudeltabetadef}), we get $\beta_c(0)=0$. The convexity of 
$\gd \mapsto \beta_c(\gd)$ and the fact that $\beta_c(\gd)>0$ for $\gd \in (0,\infty)$ imply that 
$\gd \mapsto \beta_c(\gd)$ is strictly increasing. The continuity of $\gd \mapsto \beta_c(\gd)$ 
follows from convexity and finiteness.

From \eqref{mudeltabetadef} and \eqref{varc} it follows that, for $\delta \in [0,\infty)$, 
$\beta_c(\gd)$ coincides with the unique solution of $\gl_{\gd,\gb}(0) = 1$. It therefore 
follows from Proposition~\ref{prop:regl}(i) and the implicit function theorem that $\delta 
\mapsto \beta_c(\delta)$ is analytic  on $(0,\infty)$. Finally, since $F^*(\delta,\beta)
=\mu(\delta,\beta)$, it follows from \eqref{mudeltabetadef}, Proposition~\ref{prop:regl}(i) 
and the implicit function theorem that $F^*$ is analytic on $\intr(\cB)$.


\subsection{Large deviation principles for the speed and the charge}
\label{ss:jointLDP}

In this section we prove Theorem~\ref{thm:jointLDP}.

\begin{proof}
The proof comes in 6 Steps.

\medskip\noindent
{\bf 1.}
We begin by introducing the joint moment-generating function for the speed and the charge.
Fix $(\delta,\beta) \in \cQ$ and $(\gamma,\gamma')\in\R^2$. Let
\begin{equation}
\label{Zngamma}
\begin{aligned}
\bbZ_n^{*,\delta,\beta}(\gamma,\gamma') 
&= (\bbE^\delta \times \bE)
\Big[e^{\gamma S_n + \gamma' \Omega_n}\,e^{-\beta H^\omega_n(S) - n f(\gd)}\Big]\\
&= (\bbE \times \bE)\Big[ e^{\gamma S_n + (\gd + \gamma')\gO_n - \gb H^\omega_n(S)} \Big]\\
&= \bE\Big[ e^{\gamma S_n + \sum_{y\in \bbZ} G^*_{\gd+\gamma',\gb}(L_n(y))} \Big].
\end{aligned}
\end{equation}
Then
\begin{equation}
\bbE_n^{\delta,\beta}\left[e^{\gamma S_n + \gamma' \Omega_n}\right]
= \frac{\bbZ_n^{*,\delta,\beta}(\gamma,\gamma')}{\bbZ_n^{*,\delta,\beta}(0,0)},
\end{equation}
where we recall that $\bbE_n^{\delta,\beta}$ is the expectation w.r.t.\ the annealed polymer 
measure of length $n$ defined in (\ref{eq:def.polmeasure}--\ref{eq:def.annpartfunc}). 
Next, let  
\begin{equation}
\cZ(\mu,\delta,\beta;\gamma,\gamma') 
= \sum_{n\in\N_0} e^{-\mu n} \bbZ_n^{*,\delta,\beta}(\gamma,\gamma').
\end{equation}
Then $\cZ(\mu,\delta,\beta;\gamma,\gamma')$ has a spectral representation similar to the one 
in Proposition~\ref{prop:grcanspec}. Indeed, the only difference is that $A_{\mu,\delta,\beta}$ 
must be replaced by $e^\gamma A_{\mu,\delta,\beta}$ and $\delta$ by $\delta+\gamma'$. 
The same is true for the bridge version of the moment-generating function, for which 
Lemma~\ref{lem:grcanspecloop} holds with the same replacement (see 
Proposition~\ref{prop:LDP.bridge}). Recall \eqref{mugammaid}. Repeating 
the argument in Sections~\ref{ss:spec}--\ref{ss:concl}, we obtain that
\begin{equation}
\label{Zngammaquotient}
\Lambda_{\delta,\beta}(\gamma,\gamma') 
= \lim_{n\to\infty} \frac{1}{n} \log 
\bbE_n^{\delta,\beta}\left[e^{\gamma S_n + \gamma' \Omega_n}\right]
= \lim_{n\to\infty} \frac{1}{n} \log 
\frac{\bbZ_n^{*,\delta,\beta}(\gamma,\gamma')}{\bbZ_n^{*,\delta,\beta}(0,0)} 
\end{equation}
is given by
\begin{equation}
\label{Lambdagammaidalt}
\Lambda_{\delta,\beta}(\gamma,\gamma') 
= \big[\mu(\delta+\gamma',\beta,\gamma)
\vee \tilde{\mu}(\delta+\gamma',\beta)\big] - \mu(\delta,\beta).
\end{equation} 
Here, the second term in the right-hand side comes from the denominator in \eqref{Zngammaquotient}, 
while the first term captures the crossover from a regime where the spectral radius of $e^\gamma 
A_{\mu,\delta+\gamma',\beta}$ controls the blow up of the generating function to a regime where the 
spectral radius of $\tilde{A}_{\mu,\delta+\gamma',\beta}$ does.

\medskip\noindent
{\bf 2.}
The result in \eqref{Lambdagammaidalt} allows us to apply the G\"artner-Ellis theorem 
of large deviation theory (den Hollander~\cite[Chapter V]{dH00}) and obtain that the
pair consisting of the empirical speed and the empirical charge satisfies the large deviation 
principle on $[0,\infty)$, with the associated rate function $I_{\delta,\beta}$ given by the 
Legendre transform of $\Lambda_{\delta,\beta}$, i.e.,
\begin{equation}
\label{rfjoint}
I_{\delta,\beta}(\theta,\theta') = \sup_{\gamma,\gamma'\in\R} 
\big[(\theta\gamma+\theta'\gamma')-\Lambda_{\delta,\beta}(\gamma,\gamma')\big].
\end{equation}
Actually, the G\"artner-Ellis theorem only gives us the large deviation principle in the regions 
where $I_{\delta,\beta}$ is ``exposed'', i.e., where $(\theta,\theta') \mapsto I_{\delta,\beta}
(\theta,\theta')$ is strictly convex. In the regions where $(\theta,\theta') \mapsto I_{\delta,\beta}
(\theta,\theta')$ is flat, it only gives a lower bound on the rate function and so we need to 
provide a matching upper bound. To prove the upper bound, we restrict ourselves to the 
marginal large deviation functions, which are obtained from $I_{\delta,\beta}(\theta,\theta')$ 
by setting $\gamma'=0$, respectively, $\gamma=0$. Substituting \eqref{Lambdagammaidalt} 
into \eqref{rfjoint}, we get the formulas for $I^v_{\delta,\beta}(\theta)$ and $I^\rho_{\delta,\beta}
(\theta')$ in \eqref{Ivdef}--\eqref{Irhodef}, where for the former we use that $\mu(\delta+\gamma',
\beta) \geq \tilde{\mu}(\delta+\gamma',\beta)$. 

\medskip\noindent
{\bf 3.}
Before embarking on the proof of the matching upper bounds for the flat pieces in the rate 
functions, we state two auxiliary propositions whose proofs are deferred to the end of Section \ref{ss:jointLDP}. Recall \eqref{partfuncbridge} and define the laws 
$\big(\bP_{n,\mathrm{bridge}}^{\gd,\gb}\big)_{n\in\N}$ by 
\begin{equation}
\frac{\dd\bP_{n,\mathrm{bridge}}^{\gd,\gb}}{\dd(\bbP^\gd \times \bP)}(\go,S) 
= \frac{1}{\bbZ_{n,{\rm bridge}}^{*,\gd,\gb}} e^{-\gb H_n^\go(S) -n f(\gd)}\,
\ind_{\big\{0< S_i \leq  S_n\,\,\forall\,0<i<n\big\}}.
\end{equation}
We also need a bridge version of \eqref{Zngamma}:
\begin{equation}
\label{Zngammabridge}
\bbZ_{n, {\rm bridge}}^{*,\delta,\beta}(\gamma,\gamma') 
= (\bbE^\delta \times E)
\left[e^{\gamma S_n + \gamma' \Omega_n}\,e^{-\beta H^\omega_n(S) - nf(\gd)} 
\ind_{\big\{0< S_i \leq  S_n\,\,\forall\,0<i<n\big\}}\right].
\end{equation}

\begin{proposition}
\label{prop:LDP.bridge}
{\rm (1)} The limit
\begin{equation}
\label{Zngammaquotientbridge}
\Lambda^{\rm bridge}_{\delta,\beta}(\gamma,\gamma') = \lim_{n\to\infty} \frac{1}{n} \log 
\frac{\bbZ_{n, {\rm bridge}}^{*,\delta,\beta}(\gamma,\gamma')}{\bbZ_{n, {\rm bridge}}^{*,\delta,\beta}(0,0)} 
\end{equation}
exists and is given by
\begin{equation}
\label{Lambdagammaidaltbridge}
\Lambda^{\rm bridge}_{\delta,\beta}(\gamma,\gamma') 
= \mu(\delta+\gamma',\beta,\gamma) - \mu(\delta,\beta).
\end{equation} 
{\rm (2)} Let $\hat v(\gd,\gb)$ be defined by 
\begin{equation}
\frac{1}{\hat{v}(\delta,\beta)} = \Big[-\frac{\partial}{\partial\mu} 
\log \lambda_{\delta,\beta}(\mu)\Big]_{\mu = 0},
\end{equation}
which satisfies $\hat{v}(\delta,\beta) \leq \tilde{v}(\delta,\beta)$ (see Fig.~{\rm \ref{fig-spectral}}
and \eqref{eq:speedspec*}). For every open set $O \subset (\hat{v}(\delta,\beta),\infty)$,
\begin{equation}
\label{eq:LDP.bridge.speed.lb}
\liminf_{n\to\infty}\frac{1}{n} \log \bP^{\gd,\gb}_{n,\mathrm{bridge}}(n^{-1}S_n \in O) 
\geq - \inf_{\theta \in O} I_{\gd,\gb}^{v, {\rm bridge}}(\theta),
\end{equation}
where
\begin{equation}
I_{\gd,\gb}^{v, {\rm bridge}}(\theta) 
= \sup_{\gamma\in \bbR} \big[\theta \gamma - \gL_{\gd,\gb}^{\rm bridge}(\gamma,0)\big],
\qquad \theta \in [0,\infty),
\end{equation}
and
\begin{equation}
\label{eq:eq.rate.bridge}
I_{\gd,\gb}^{v, {\rm bridge}}(\theta) = I_{\gd,\gb}^{v}(\theta),
\qquad \theta\in [\tilde v(\gd,\gb),\infty).
\end{equation}
{\rm (3)} For every open set $O \subset (\tilde\rho(\gb),\infty)$,
\begin{equation}
\label{eq:LDP.bridge.charge.lb}
\liminf_{n\to\infty}\frac{1}{n} \log \bP^{\gd,\gb}_{n,\mathrm{bridge}}(n^{-1}\gO_n \in O) 
\geq - \inf_{\theta' \in O} I_{\gd,\gb}^{\rho, {\rm bridge}}(\theta'),
\end{equation}
where
\begin{equation}
\label{eq:eq.rate.bridge2}
I_{\gd,\gb}^{\rho, {\rm bridge}}(\theta') 
= \sup_{\gamma'\in \bbR} \big[\theta' \gamma' - \gL_{\gd,\gb}^{\rm bridge}(0,\gamma')\big],
\qquad \theta' \in [0,\infty).
\end{equation}
and
\begin{equation}
\label{eq:eq.rate.bridge3}
I_{\gd,\gb}^{\rho, {\rm bridge}}(\theta')  = I_{\gd,\gb}^{\rho}(\theta'),
\qquad \theta' \in [0,\infty).
\end{equation}
\end{proposition}

\noindent
(With additional work, \eqref{eq:LDP.bridge.speed.lb} and \eqref{eq:LDP.bridge.charge.lb} 
can be turned into large deviation principles.)

Define the partition function restricted to loops:
\begin{equation}
\label{eq:def.loop.part.func}
\bbZ_{n,\mathrm{loop}}^{*,\gd,\gb} 
= \bE \Big[e^{\sum_{y\in\Z} G^*_{\delta,\beta}(L_n(y))}\,
\ind_{\{S_i\geq 0\,\,\forall\, 0<i<n,\, S_n = 0\}} \Big],\qquad n\in 2\N.
\end{equation}
\begin{proposition}
\label{prop:loop.part.func}
The sequence $(n^{-1}\log \bbZ_{n,\mathrm{loop}}^{*,\gd,\gb})_{n\in 2\bbN}$ converges.
\end{proposition}

\noindent
{\bf 4. Flat piece of $I^v_{\gd,\gb}$.} 
We start with the matching upper bound for $I^v_{\gd,\gb}$ on $(0,\tilde{v}(\gd,\gb))$ 
(see Fig.~\ref{fig-rfspeed}). For ease of notation, we omit to write integer parts. Let 
$\theta \in (0, \tilde{v}(\gd,\gb))$, and let $\gep>0$ be small enough so that 
$(\theta-\gep,\theta+\gep) \subseteq (0,\tilde{v}(\gd,\gb))$. It is enough to show that
\begin{equation}
\label{eq:lb.lin.piece.speed}
\begin{aligned}
&\liminf_{n\to \infty} \frac{1}{n} \log 
\bbP_n^{\gd,\gb}\Big(n^{-1}S_n\in (\theta-\gep,\theta+\gep) \Big)\\ 
&\qquad \geq \frac{\theta}{\tilde{v}(\gd,\gb)} I^v_{\gd,\gb}(0) 
+ \Big( 1- \frac{\theta}{\tilde{v}(\gd,\gb)} \Big) I^v_{\gd,\gb}(\tilde{v}(\gd,\gb)) - \bar\gep,
\end{aligned}
\end{equation}
for some $\bar\gep$ that tends to zero as $\gep \downarrow 0$. Note that
\begin{equation}
\label{eq:Ivgdgb0}
I^v_{\gd,\gb}(0) = \mu(\gd,\gb) - \inf_{\gamma\in\bbR} 
\{ \mu(\gd,\gb,\gamma) \vee \tilde{\mu}(\gd,\gb)\} 
= \mu(\gd,\gb) - \tilde{\mu}(\gd,\gb).
\end{equation}
In order to prove \eqref{eq:lb.lin.piece.speed}, we adopt the following strategy: the 
polymer moves to the right ballistically at speed $\tilde{v}(\gd,\gb)$ for a fraction of 
time $\theta/ \tilde{v}(\gd,\gb)$ and spends the rest of the time making a loop to the right. 
Recall \eqref{eq:def.loop.part.func}. As the reader can easily check, the method explained 
in Section~\ref{freeenergy} leads to the following representation of the grand-canonical 
partition function restricted to loops:
\begin{equation}
\sum_{n\in 2\N} e^{-\mu n} \bbZ_{n,\mathrm{loop}}^{*,\gd,\gb} 
= \Bigg[ \frac{1}{1-\tilde{A}_{\mu,\gd,\gb}}  \Bigg](0,0), \qquad \mu\geq 0.
\end{equation}
Note that $\spr(\tilde{A}_{\mu,\gd,\gb}) = \tilde{\gl}_{\gd,\gb}(\mu) < 1$ when 
$\mu > \tilde{\mu}(\gd,\gb)$, while $\tilde{\gl}_{\gd,\gb}(\mu) \geq 1$ when 
$\mu \in (0, \tilde{\mu}(\delta,\beta))$. By repeating the same argument as 
in Section~\ref{ss:spec}, we deduce that
\begin{equation}
\limsup_{ {n\to\infty} \atop {n\in 2\N} } \frac{1}{n} 
\log \bbZ_{n,\mathrm{loop}}^{*,\gd,\gb} = \tilde{\mu}(\gd,\gb).
\end{equation}
Moreover, by Proposition~\ref{prop:loop.part.func}, the limsup is actually a lim. Recall that
\begin{equation}
\lim_{n\to \infty} \frac{1}{n} \log \bbZ_{n}^{*,\gd,\gb} 
= \lim_{n\to \infty} \frac{1}{n} \log \bbZ_{n, {\rm bridge}}^{*,\gd,\gb} = \mu(\gd,\gb).
\end{equation}
Write $m_\mp = n\theta/\tilde v(\delta,\beta)(1\pm\gep)$ and abbreviate $E_m 
= \{m^{-1}S_m \in (\tilde v(\delta,\beta), \tilde v(\delta,\beta) (1+\gep^2))\}$.
Then the strategy above translates into
\begin{equation}
\begin{aligned}
&\bbP_n^{\gd,\gb}\Big(n^{-1}S_n \in (\theta(1-2\gep),\theta(1+2\gep)) \Big)\\
&\qquad = \frac{1}{\bbZ_n^{*,\gd,\gb}}
\bE\Big[e^{\sum_{y\in\Z} G^*_{\delta,\beta}(L_n(y))} 
\ind_{\{n^{-1}S_n \in (\theta(1-2\gep),\theta(1+2\gep))\}} \Big]\\
&\qquad \geq \sumtwo{m_- \leq m \leq m_+}{n-m \in 2\N}
\Bigg(\frac{\bbZ_{m,\mathrm{bridge}}^{*,\gd,\gb}(E_m)}{e^{m\mu(\gd,\gb)[1+o(1)]}} \Bigg) 
\Bigg(\frac{\bbZ_{n-m, \mathrm{loop}}^{*,\gd,\gb}}{e^{(n-m)\mu(\gd,\gb)[1+o(1)]}} \Bigg)\\
&\qquad \geq \sumtwo{m_- \leq m \leq m_+}{n-m \in 2\N} 
\bP_{m,\mathrm{bridge}}^{\gd,\gb}(E_m)\, 
\frac{e^{(n-m)\tilde\mu(\gd,\gb)[1+o(1)]}}{e^{(n-m)\mu(\gd,\gb)[1+o(1)]}}\\
&\qquad \geq C\, n \,\exp\Big[n[1+o(1)]
\quad \Big\{- \frac{\theta}{\tilde v(\gd,\gb) (1 - \gep)} 
I^v_{\gd,\gb}\big(\tilde v(\delta,\beta)(1+\gep^2)\big)\\
&\qquad\qquad\qquad\qquad  + \Big(1- \frac{\theta}{\tilde v(\delta,\beta) (1+\gep)}\Big) 
[\tilde\mu(\gd,\gb) - \mu(\gd,\gb)]\Big\} \Big],
\end{aligned}
\end{equation}
where $C$ is a positive constant that depends on $\gep, \theta, \tilde{v}$, and $\bbZ_{n,\mathrm{bridge}}^{*,\gd,\gb}(E)$ 
is short-hand notation for the bridge partition function restricted to the event $E$. We use \eqref{eq:eq.rate.bridge} 
for the last inequality.

By continuity of $\theta \mapsto I^v_{\gd,\gb}(\theta)$ at $\theta=\tilde v$ 
and \eqref{eq:Ivgdgb0}, this is the desired result. Note that the last inequality holds because of 
Proposition~\ref{prop:LDP.bridge}. Also note also that the $o(1)$'s are uniform on $m_- \leq m \leq m_+$ 
because $n-m \geq (1 - \theta/[\tilde v(\gd,\gb)(1-\gep)])n$.

\medskip\noindent
{\bf 5. Flat piece of $I^\rho_{\gd,\gb}$.} 
We now turn to the matching upper bound for $I^\rho_{\gd,\gb}$ on $(0,\tilde\rho(\gb))$ (see 
Fig.~\ref{fig-rfcharge}). We call a path $(S_i)_{0 \leq i \leq n}$ a \emph{half-bridge} when 
$S_i\geq S_0 (= 0)$ for all $0\leq i \leq n$, and we define
\begin{equation}
\bbZ^{*,\gd,\gb}_{n,\rm{hb}} = \bE\Big[e^{\sum_{y\in\bbZ} G^*_{\gd,\gb}(L_n(y))}
\ind_{\{S_i\geq0\,\,\forall\, 0 \leq i \leq n\}}\Big].
\end{equation}
Let $\theta' \in (0,\tilde\rho(\gb))$. The strategy is similar as above: fix $m \approx \theta'n/\tilde\rho(\gb)$, 
let the first $m$ charges have an empirical average close to $\tilde\rho(\gb)$, and let the remaining 
$n-m$ charges have an empirical average close to $0$. To be more precise, write $m_\pm = n\theta'
(1\pm\gep)/\tilde \rho(\gb)$ and abbreviate $\bar E_m = \{m^{-1}\gO_m \in (1,1+\gep^2)\tilde\rho(\gb))\}$ 
and $\tilde E_{n-m} = \{0\leq \gO_{n-m} \leq \gep(n-m)\}$. Estimate
\begin{equation}
\begin{aligned}
&\bbP^{\gd,\gb}_{n}\big(n^{-1}\gO_n \in (\theta' (1-2\gep), \theta' (1+2\gep))\big)\\
&\qquad = \frac{1}{\bbZ_n^{*,\gd,\gb}}\, \bbZ_n^{*,\gd,\gb}\big(n^{-1}\gO_n \in 
(\theta' (1-2\gep), \theta' (1+2\gep))\big)\\
&\qquad \geq \sum_{m_- \leq m \leq m_+} \frac{1}{\bbZ_n^{*,\gd,\gb}}\, 
\bbZ_{m, {\rm bridge}}^{*,\gd,\gb}(\bar E_m)\, 
\bbZ_{n-m,{\rm hb}}^{*,\gd,\gb}(\tilde E_{n-m})\\
&\qquad \geq e^{-n\gep} \sum_{m_- \leq m \leq m_+}
\frac{\bbZ_{m, {\rm bridge}}^{*,\gd,\gb}(\bar E_m)}{e^{m\mu(\gd,\gb)}} 
\frac{\bbZ_{n-m,{\rm hb}}^{*,\gd,\gb}(\tilde E_{n-m})}{e^{(n-m) \mu(\gd,\gb)}}\\
&\qquad \geq e^{-2n\gep} \sum_{m_- \leq m \leq m_+}
\bP_{m, {\rm bridge}}^{*,\gd,\gb}(\bar E_m) 
\frac{\bbZ_{n-m,{\rm hb}}^{*,\gd,\gb}(\tilde E_{n-m})}{e^{(n-m) \mu(\gd,\gb)}}.
\end{aligned}
\end{equation}
Thanks to \eqref{eq:eq.rate.bridge2}--\eqref{eq:eq.rate.bridge3} and the G\"artner-Ellis theorem, we know that
\begin{equation}
\liminf_{n\to\infty} \frac{1}{n} \log \bP_{n, {\rm bridge}}^{*,\gd,\gb}(\bar E_n) 
\geq \inf_{\theta'\in (1,1+\gep^2)\tilde\rho(\gb)}I^{\rho}_{\gd,\gb}(\theta').
\end{equation}
Below we will prove that
\begin{equation}
\label{eq:liminf.loop}
\liminf_{ {n\to\infty} } \frac{1}{n} 
\log \bbZ_{n,{\rm hb}}^{*,\gd,\gb}(\tilde E_n) \geq 0.
\end{equation}
Recalling that $I^{\rho}_{\gd,\gb}(0) = \mu(\gd,\gb)$ and using the continuity of $I^{\rho}_{\gd,\gb}$, 
we therefore obtain
\begin{equation}
\begin{aligned}
\liminf_{n\to \infty}  \frac{1}{n} \log & \bP^{\gd,\gb}_{n}\big(n^{-1}\gO_n 
\in (\theta' (1-\gep), \theta' (1+\gep))\big)  \\ &
\geq - \frac{\theta'}{\tilde\rho(\gb)} I^{\rho}_{\gd,\gb}(\tilde\rho(\gb)) 
- \Big(1-\frac{\theta'}{\tilde\rho(\gb)}\Big)I^{\rho}_{\gd,\gb}(0) - \tilde\gep,
\end{aligned}
\end{equation}
where $\tilde\gep \downarrow 0$ as $\gep \downarrow 0$, which is the desired result.

\medskip\noindent
{\bf 6.}
It remains to prove \eqref{eq:liminf.loop}. To that end, abbreviate $\Omega_n^x 
= \sum_{i=1}^n \go_i \ind_{\{S_i = x\}}$ and estimate
\begin{equation}
\bbZ_{n,{\rm hb}}^{*,\gd,\gb}(\tilde E_n) 
\geq (\bbE \times \bE) \Big[e^{-\gb \sum_{x\in\Z} (\Omega_n^x)^2} 
\ind_{\{\gO_n \in (0,\gep n)\}}\,\ind_{\{S_i \geq 0\,\,\forall\,0\leq i \leq n\}} \Big].
\end{equation}
The following strategy gives a lower bound. Fix $\ga\in (0,\tfrac12)$ and $C,M$ large, 
and define the events (to ease the notation we omit writing integer parts)
\begin{equation}
\begin{aligned}
\cA_n &= \{S_i \in [0, n^\ga] \,\,\forall\,0\leq i \leq n\},\\
\cB_n &= \{\ell_n(x) \leq C n^{1-\ga} \,\,\forall\,x \in [0,n^\ga]\},\\
\cC_n &= \Big\{ \Omega_n^x  \in (0,M] \,\,\forall\,x \in [0,n^\ga] \Big\}.
\end{aligned}
\end{equation}
On the event $\cA_n \cap \cB_n \cap \cC_n$, we have
\begin{equation}
\sum_{x\in\Z} (\Omega_n^x)^2 \leq M^2 (n^\ga + 1),
\qquad \Omega_n \in (0,\gep n) \text{ for } n \geq (M/2\gep)^{1/(1-\alpha)}.
\end{equation}
Therefore, for $n$ large enough,
\begin{equation}
\label{ZnAnBnCn}
\bbZ_{n,{\rm hb}}^{*,\gd,\gb}(\tilde E_n) 
\geq e^{-\gb M^2 (n^\ga+1)} (\bbE \times \bE)[\ind_{\cA_n \cap \cB_n \cap \cC_n}].
\end{equation}
Next, by the independence of the charges and the local central limit theorem, 
we have, for $M$ large enough,
\begin{equation}
\begin{aligned}
\bE\big[\bbP(\cC_n)\,\ind_{\{\cA_n \cap \cB_n\}}\big] 
&\geq \bE\left[ \big\{\bbP(\gO^0_\ell \in (0,M] \,\,\forall\, \ell \leq Cn^{1-\ga}\big\}^{n^\ga+1}\,
\ind_{\{\cA_n \cap \cB_n\}}\right] \\
&\geq \left( \frac{c_M}{\sqrt{Cn^{1-\ga}}} \right)^{n^\ga+1} \bP(\cA_n \cap \cB_n),
\end{aligned}
\end{equation}
where $c_M$ is a constant that depends on $M$. In addition, there exists a constant 
$c$ such that
\begin{equation}
\label{PnAnBnCn}
\bP(\cA_n \cap \cB_n) \geq e^{-c n^{1-2\ga}}.
\end{equation}
Indeed, the probability that a simple random walk stays inside $[0, n^\ga]$ up to time 
$n$ is
\begin{equation}
\label{estAn}
\bP(\cA_n) = \exp\Big[ - \frac{\gl_1 n}{(n^\ga)^2} (1+o(1))\Big], \qquad n\to\infty,
\end{equation}
where $\gl_1$ is the principal Dirichlet eigenvalue of the continuous Laplacian on $[0,1]$. 
Moreover, for $C$ large enough we have $\lim_{n\to\infty} \bP(\cB_n \mid \cA_n) = 1$,
because $(n^{-(1-\ga)}\,L_n(x n^{\ga}))_{x \in [0,1]}$ conditionally on $\cA_n$ converges 
in distribution to the square of the principal Dirichlet eigenfunction. Combining 
(\ref{ZnAnBnCn}--\ref{PnAnBnCn}), we arrive at
\begin{equation}
\bbZ_{n,{\rm hb}}^{*,\gd,\gb}(\tilde E_n) 
\geq \exp\big[-\gb M^2 (n^{\ga}+1) - c n^{1-2\ga} - c_{M,\alpha,C} n^{\ga} \log n\big],
\end{equation}
which yields \eqref{eq:liminf.loop}.
\end{proof}

\begin{proof}[Proof of Proposition~\ref{prop:LDP.bridge}]
The proof consists in a slight modification of the arguments at the beginning of 
this section, see \eqref{Zngamma}--\eqref{rfjoint}. Indeed, define 
\begin{equation}
\cZ^\mathrm{bridge}(\mu,\gd,\gb;\gamma,\gamma') 
= \sum_{n\in\N_0} e^{-\mu n} \bbZ^{*,\gd,\gb}_{n,\mathrm{bridge}}(\gamma,\gamma').
\end{equation}
A minor change in \eqref{Zbridgeform} gives
\begin{equation}
\cZ^\mathrm{bridge}(\mu,\gd,\gb;\gamma,\gamma') 
= \Big[ \frac{e^\gamma A_{\mu,\gd+\gamma',\gb}}
{1 - e^\gamma A_{\mu,\gd+\gamma',\gb}} \Big](0,0),
\end{equation}
from which we deduce \eqref{Lambdagammaidaltbridge}. Let
\begin{equation}
\label{eq:defgammac}
\gamma_c(\delta,\beta) = - \log \lambda_{\delta,\beta}(\tilde{\mu}(\delta,\beta)),
\end{equation}
as shown in Fig.~\ref{fig-spectralalt}. Observe that $\Lambda^\mathrm{bridge}_{\delta,\beta}
(\gamma,0) = \Lambda_{\delta,\beta}(\gamma,0)$ for $\gamma \geq \gamma_c(\gd,\gb)$, 
which yields \eqref{eq:eq.rate.bridge}. Finally, \eqref{eq:eq.rate.bridge2} is an immediate 
consequence of $\Lambda_{\gd,\gb}(0,\cdot) \equiv \Lambda^{\rm bridge}_{\gd,\gb}(0,\cdot)$. 
Both \eqref{eq:LDP.bridge.speed.lb} and \eqref{eq:LDP.bridge.charge.lb} follow from the 
G\"artner-Ellis theorem.
\end{proof}

\begin{proof}[Proof of Proposition~\ref{prop:loop.part.func}]
For $n\in 2\N$, let $\cL_n$ be the space of right-loops of length $n$, that is
\begin{equation}
\label{eq:def.loop}
\cL_n = \{S_i \geq 0,\,\,\forall\,0<i<n,\,S_n=0\}.
\end{equation}
For $m,n\in 2\N$, define $\cL_{m,n}$ as the following subset of $\cL_{m+n}$:
\begin{equation}
\begin{aligned}
\cL_{m,n} &= \big\{S\in \cL_{m+n}\colon\,\exists\, 0 \leq a \leq n\colon\, S_a = S_{m+a},\\
&\qquad S_i <S_a,\,\,\forall\,i \in [0,a) \cup (m+a,m+n],\,S_j \geq S_a,\,\,\forall\,j \in [a,a+m]\big\}.
\end{aligned}
\end{equation}
Note that: (i) there are at most $n+1$ ways of cutting a loop in $\cL_{m+n}$ 
to obtain two loops in $\cL_m$ and $\cL_n$ after gluing the two end-parts 
together; (ii) the loops obtained in this way involve two independent sets 
of charges. Therefore
\begin{equation}
\bbZ_{m+n,\mathrm{loop}}^{*,\gd,\gb} \geq 
\bE\Big[ e^{\sum_{y\in\Z} G^*_{\delta,\beta}(L_n(y))}\, 
\ind_{\{S\in \cL_{m,n}\}} \Big] \geq \frac{1}{n+1} 
\bbZ_{m,\mathrm{loop}}^{*,\gd,\gb} \bbZ_{n,\mathrm{loop}}^{*,\gd,\gb}
\end{equation}
and so
\begin{equation}
\log \bbZ_{m+n,\mathrm{loop}}^{*,\gd,\gb} \geq \log \bbZ_{m,\mathrm{loop}}^{*,\gd,\gb} 
+ \log \bbZ_{n,\mathrm{loop}}^{*,\gd,\gb} - \log [(m+1) \wedge (n+1)].
\end{equation}
Hence the sequence $\big(\log \bbZ_{n,\mathrm{loop}}^{*,\gd,\gb}\big)_{n\in 2\N}$ is almost 
super-additive, and an application of Hammersley~\cite[Theorem 1]{H62} 
gives the result.
\end{proof}

\subsection{Shape of rate functions}
\label{ss:rfshape}

In this section we show how Figs.~\ref{fig-rfspeed}--\ref{fig-rfcharge} come about. Recall 
from \eqref{mugammaid} that $\log\lambda_{\delta,\beta}(\mu(\delta,\beta,\gamma))
= -\gamma$ when $\gamma \geq -\log \gl_{\gd,\gb}(0)$. Differentiating this equation 
twice with respect to $\gamma$, we obtain
\begin{equation}
\begin{aligned}
\frac{\partial\mu}{\partial\gamma}(\delta,\beta,\gamma)
&= \left[- \frac{\partial}{\partial\mu}\log\lambda_{\delta,\beta}(\mu)
\right]^{-1}_{\mu=\mu(\delta,\beta,\gamma)},\\
\frac{\partial^2\mu}{\partial\gamma^2}(\delta,\beta,\gamma)
&= \left[\frac{\partial^2}{\partial\mu^2}\log\lambda_{\delta,\beta}(\mu)
\right]_{\mu=\mu(\delta,\beta,\gamma)}
\left[\frac{\partial\mu}{\partial\gamma}(\delta,\beta,\gamma)\right]^3.
\end{aligned}
\end{equation}
From this and the strict convexity (respectively, monotonicity) of $\mu \mapsto \log \gl_{\gd,\gb}(\mu)$ 
we see that $\gamma \mapsto \mu(\delta,\beta,\gamma)$ is strictly convex (respectively, increasing) 
for $\gamma \geq -\log \gl_{\gd,\gb}(0)$ (recall Fig.~\ref{fig-spectral}). Moreover, since $\mu(\gd,\gb) 
= F^*(\gd,\gb)$, we know that $\delta \mapsto \mu(\delta,\beta)$ is strictly increasing and strictly 
convex as well. 

\begin{figure}[htbp]
\begin{center}
\setlength{\unitlength}{0.35cm}
\begin{picture}(12,12)(0,-4)
\put(0,0){\line(1,0){12}}
\put(0,-6){\line(0,1){14}}
{\thicklines
\qbezier(0,6)(4,0)(12,-3)
\qbezier(0,1.5)(2,0)(10,-3.5)
}
\put(-1,-.5){$0$}
\put(12.5,-0.2){$\mu$}
\put(6,.7){$\mu(\delta,\beta)$}
\put(.6,-1.5){$\tilde{\mu}(\delta,\beta)$}
\put(13,-3){$\log\lambda_{\delta,\beta}$}
\put(11,-4.5){$\log\tilde{\lambda}_{\delta,\beta}$}
\put(-4.6,2.5){$-\gamma_c(\delta,\beta)$}
\put(0,6){\circle*{.4}}
\put(0,1.6){\circle*{.4}}
\put(6,0){\circle*{.4}}
\put(2.6,0){\circle*{.4}}
\qbezier[20](2.6,0)(2.6,1.3)(2.6,2.7)
\qbezier[20](2.6,2.7)(1.6,2.7)(0,2.7)
\end{picture}
\end{center}
\vspace{0.5cm}
\caption{\small Illustration of the crossover value $\gamma_c(\delta,\beta)<0$ when $(\delta,\beta)
\in \intr(\cB)$. This crossover value changes sign at the critical curve.}
\label{fig-spectralalt}
\end{figure}
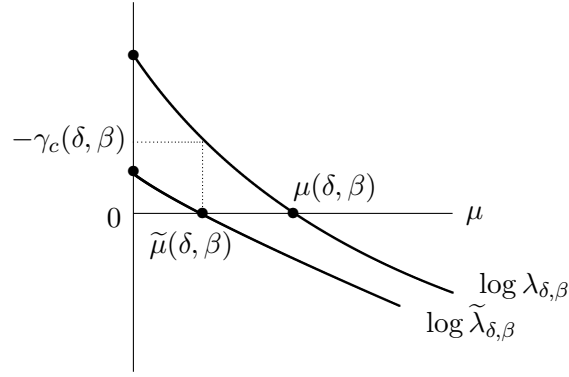

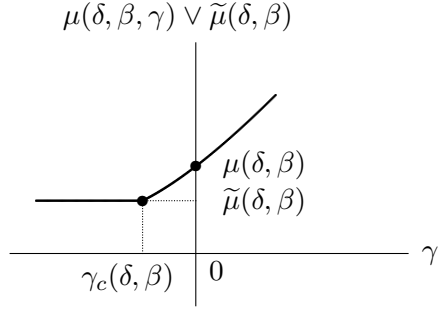
\begin{figure}[htbp]
\vspace{1cm}
\begin{center}
\setlength{\unitlength}{0.35cm}
\begin{picture}(12,12)(-4,-4.5)
\put(-7,0){\line(1,0){15}}
\put(0,-2){\line(0,1){10}}
{\thicklines
\qbezier(-6,2)(-4,2)(-2,2)
\qbezier(-2,2)(0,3)(3,6)
}
\qbezier[20](-2,0)(-2,1)(-2,2)
\qbezier[20](-2,2)(-1,2)(0,2)
\put(0.5,-1){$0$}
\put(8.5,-0.2){$\gamma$}
\put(-5,8.7){$\mu(\delta,\beta,\gamma) \vee \tilde{\mu}(\delta,\beta)$}
\put(-4.3,-1.2){$\gamma_c(\delta,\beta)$}
\put(1,3.1){$\mu(\delta,\beta)$}
\put(1,1.7){$\tilde{\mu}(\delta,\beta)$}
\put(-2,2){\circle*{.4}}
\put(0,3.3){\circle*{.4}}
\end{picture}
\end{center}
\vspace{-1cm}
\caption{\small Qualitative plot of $\gamma \mapsto \mu(\delta,\beta,\gamma) \vee 
\tilde{\mu}(\delta,\beta)$ for $(\delta,\beta) \in \intr(\cB)$. Both $\tilde{\mu}(\delta,\beta)>0$
and $\tilde{\mu}(\delta,\beta)=0$ are possible.}
\label{fig-maxmus}
\end{figure}

Recall \eqref{eq:defgammac}. By \eqref{mudeltabetadef} and \eqref{mugammaid},  
$\mu(\delta,\beta,\gamma)>\tilde{\mu}(\delta,\beta)$ when $\gamma>\gamma_c(\delta,\beta)$ 
and $\mu(\delta,\beta,\gamma) \leq \tilde{\mu}(\delta,\beta)$ when $\gamma \leq 
\gamma_c(\delta,\beta)$. Hence $\gamma \mapsto \mu(\delta,\beta,\gamma) \vee 
\tilde{\mu}(\delta,\beta)$ has the shape depicted in Fig.~\ref{fig-maxmus}. The right slope 
of this function at $\gamma=\gamma_c(\delta,\beta)$ is precisely the speed $\tilde{v}(\delta,\beta)$ 
defined in \eqref{eq:speedspec*}. Hence $\theta \mapsto I^v_{\delta,\beta}(\theta)$, which is 
given by the Legendre transform in \eqref{Ivdef}, is linear on $[0,\tilde{v}(\delta,\beta)]$ and 
strictly convex on $(\tilde{v}(\delta,\beta),\infty)$. For $\theta \in (0,\tilde{v}(\delta,\beta))$ the 
supremum over $\gamma$ in \eqref{Ivdef} is taken at $\gamma=\gamma_c(\delta,\beta)$, so 
that
\begin{equation}
I^v_{\delta,\beta}(\theta) = \mu(\delta,\beta) 
+ \big[\theta\gamma_c(\delta,\beta)-\tilde{\mu}(\delta,\beta)\big]
= \big[\mu(\delta,\beta) - \tilde{\mu}(\delta,\beta)\big] 
- \theta \log \lambda_{\delta,\beta}(\tilde{\mu}(\delta,\beta)). 
\end{equation}

\begin{figure}[htbp]
\begin{center}
\setlength{\unitlength}{0.35cm}
\begin{picture}(12,12)(0,-2)
\put(0,0){\line(12,0){12}}
\put(0,0){\line(0,8){8}}
{\thicklines
\qbezier(0,0)(5,0.5)(9,6.5)
}
\put(-1,-.5){$0$}
\put(12.5,-0.2){$\delta$}
\put(-0.3,8.5){$\beta$}
\put(0,0){\circle*{.4}}
\put(9,2){$\cB$}
\put(3,4){$\cS$}
\put(8.5,7.3){$\beta_c(\delta)$}
\put(4.7,-1.2){$\delta_c(\beta)$}
\put(7.8,-1.2){$\delta$}
\put(-1,2.7){$\beta$}
\put(8,3){\circle*{.4}}
\qbezier[30](8,3)(7,3)(6,3)
\qbezier[30](6,3)(6,1.5)(6,0)
\qbezier[30](8,3)(8,1.5)(8,0)
\qbezier[60](0,3)(4,3)(8,3)
\end{picture}
\end{center}
\vspace{0cm}
\caption{\small Illustration of the crossover value $\gamma'_c(\delta,\beta)<0$ when $(\delta,\beta)
\in \intr(\cB)$. This crossover value changes sign at the critical curve.}
\label{fig-critcurvealt}
\end{figure}
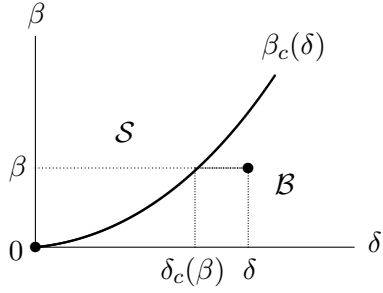

\begin{figure}[htbp]
\vspace{0.5cm}
\begin{center}
\setlength{\unitlength}{0.35cm}
\begin{picture}(12,12)(-4,-4)
\put(-7,0){\line(1,0){15}}
\put(0,-2){\line(0,1){10}}
{\thicklines
\qbezier(-6,0)(-4,0)(-2,0)
\qbezier(-2,0)(0,1)(2,5)
}
\put(.5,.5){$0$}
\put(8.5,-0.2){$\gamma'$}
\put(-3,8.7){$\mu(\delta+\gamma',\beta)$}
\put(-4,-1.4){$\gamma'_c(\delta,\beta)$}
\put(-2,0){\circle*{.4}}
\end{picture}
\end{center}
\vspace{-1cm}
\caption{\small Qualitative plot of $\gamma' \mapsto \mu(\delta+\gamma',\beta)$
for $(\delta,\beta) \in \intr(\cB)$.}
\label{fig-mu}
\end{figure}
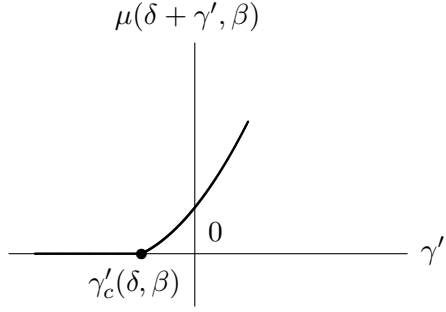

Let
\begin{equation}
\gamma'_c(\delta,\beta) = \delta_c(\beta) - \delta,
\end{equation}
as shown in Fig.~\ref{fig-critcurvealt}, where we recall that $\beta\mapsto \delta_c(\beta)$ 
is the inverse of the critical curve. By \eqref{mudeltabetadef}, $\mu(\delta+\gamma',\beta)
>0$ when $\gamma'>\gamma'_c(\delta,\beta)$ and $\mu(\delta+\gamma',\beta)=0$ when 
$\gamma'<\gamma'_c(\delta,\beta)$. Hence $\gamma' \mapsto \mu(\delta+\gamma',\beta)$ 
has the shape depicted in Fig.~\ref{fig-mu}. The right slope of this function at $\gamma
=\gamma'_c(\delta,\beta)$ is precisely $\tilde{\rho}(\beta)$ defined in \eqref{eq:speedspec*}. 
Hence $\theta' \mapsto I^\rho_{\delta,\beta}(\theta')$, which is given by the Legendre 
transform in \eqref{Irhodef}, is linear on $[0,\tilde{\rho}(\beta)]$ and strictly convex on 
$(\tilde{\rho}(\beta),\infty)$. For $\theta' \in (0,\tilde{\rho}(\beta))$ the supremum over 
$\gamma'$ in \eqref{Irhodef} is taken at $\gamma=\gamma'_c(\delta,\beta)$, so that
\begin{equation}
I^\rho_{\delta,\beta}(\theta') = \mu(\delta,\beta) 
+ \big[\theta'\gamma'_c(\delta,\beta)-\mu(\delta+\gamma'_c(\delta,\beta),\beta)\big]
= \mu(\delta,\beta) - \theta'[\delta-\delta_c(\beta)]. 
\end{equation}


\subsection{Central limit theorems for the speed and the charge}
\label{CLTproof}

In this section we prove Theorem~\ref{thm:CLTspeedcharge} for the speed. The 
extension of the argument to the charge is given in Appendix~\ref{AppB}.

\begin{proof}
The following sketch of the proof for the speed is inspired by K\"onig~\cite{K96}. 
The proof comes in 6 Steps.

\medskip\noindent
{\bf 1.} 
We begin with a probabilistic interpretation of $e^{-\mu(\gd,\gb)n}\Z_n^{*,\gd,\gb}
(S_n = x)$, $x\in\bbN$. We recall that the kernels $A, \tilde A$ and $\widehat A$
are defined in \eqref{eq:def.Abeta.p.r}, \eqref{eq:def.Abeta.p.r2} and \eqref{eq:hatA}, 
respectively.

Let $\bP^{x,\gd,\gb}$ be the joint law of two independent positive recurrent Markov 
chains on $\N_0$, denoted by $(M_y^-)_{y\geq 0}$ and $(M_y^+)_{y\geq 0}$, where 
$(M_y^+)_{0\leq y < x}$ has transition kernel $Q_{\gd,\gb}$ given in \eqref{eq:Qgdgb}
below, and $(M_y^-)_{y\geq 0}$ and $(M_y^+)_{y\geq x}$ have transition kernel $\tilde 
Q_{\gd,\gb}$ given in \eqref{eq:tildeQgdgb} below. The law of $M^+_0$ is the invariant 
law for $Q_{\gd,\gb}$, while the law of $M^-_0$ is the invariant law for $\tilde Q_{\gd,\gb}$ 
restricted to $\bbN$ and normalised by $\tilde\gl_\bbN$ ($0$ is an absorbing state). In 
particular,
\begin{equation}
\label{eq:M0pm}
\bP^{x,\gd,\gb}(M_0^+ = k) = \nu(k)^2 \,, \qquad
\bP^{x,\gd,\gb}(M_0^- = k) = \tilde \eta(k) \tilde \nu(k),\qquad k\in\N_0,
\end{equation}
with $\nu, \tilde \nu, \tilde \eta$ defined below. Let
\begin{equation}
\sigma = \inf\{i\in\N\colon\, M_i^- = 0\}, \qquad
\tau = \inf\{ i\in\N\colon\, M_{x-1+i}^+ = 0\},
\end{equation}
and define
\begin{equation}
\label{eq:Ypmj}
Y^- = 2 \sum_{y=0}^{\sigma - 1} M^-_y,\quad Y^+ 
= 2 \sum_{y=0}^{\tau - 1} M^+_{x +y}, \qquad Y_j 
= \sum_{y=0}^{j-1} (2M^+_y + 1), \qquad 1\leq j \leq x.
\end{equation}

\begin{lemma}
\label{lem:CLT1}
There exist explicit functions $f,g,h,k$, which are given in \eqref{eq:explicit} below, 
such that
\begin{equation}
\label{fghrep}
\begin{aligned}
&e^{-\mu(\gd,\gb)n}\Z_n^{*,\gd,\gb}(S_n = x)\\ 
&\qquad = \bE^{x,\gd,\gb}\Big[f(M^-_{\sigma-1})\,g(M_0^-, M_0^+)\, 
h(M^+_{x})\,k(M^+_{\tau-1})\,\ind_{\{Y^- + Y_x + Y^+ = n\}}\Big].
\end{aligned}
\end{equation}
\end{lemma}

\begin{proof}
An adaptation of the arguments leading to \eqref{eq:long3b} gives
\begin{equation}
\label{eq:CLT2}
\begin{split}
e^{-\mu n}\bbZ_n^{*,\gd,\gb}(S_n=x)
&= \sumtwo{(t,s)\in (\N_0)^2}{t\geq x} \,\,
\sum_{(m_y^+)_{0\leq y < t}, (m_y^-)_{0\leq y < s}\in \cS_{x,n}^{s,t}}\\
& \widehat A(m_0^-, m_0^+) \,
\prod_{y=0}^{x-1} A(m^+_y, m^+_{y+1})
\prod_{y= x}^{t-1} \tilde A'(m^+_y, m^+_{y+1})
\prod_{y = 0}^{s-1} \tilde A'(m^-_y, m^-_{y+1}),
\end{split}
\end{equation}
where
\begin{equation}
\begin{split}
\cS_{x,n}^{s,t} = \Big\{ &m_y^+\in \N_0,\, 0\leq y <x,\,m_y^+ \in \N,\, x\leq y < t,\\ 
&m_y^- \in \N,\, 0\leq y < s,\, \sum_{y=0}^{s-1} m_y^- + \sum_{y=0}^{t-1} 
m_y^+ = \tfrac12 (n-x) \Big\}
\end{split}
\end{equation}
and $m^+_t = m_s^- = 0$. Recall that the subscripts $\mu,\gd,\gb$ are suppressed from 
$\widehat A$, $A$ and $\tilde A'$, and the same observation as made below \eqref{eq:long3b} 
allows us to replace $\tilde A'$ by $\widetilde{A}$ in the line above. From now on we 
choose $\mu = \mu(\gd,\gb)$. Then, the spectral radius of $A$ is $1$. Writing $\nu$ for its 
normalized associated eigenvector (right or left, by symmetry), we note that
\begin{equation}
\label{eq:Qgdgb}
Q_{\gd,\gb}(i,j) = A(i,j)\,\frac{\nu(j)}{\nu(i)}, \qquad i,j\in\N_0,
\end{equation}
is a transition matrix with invariant probability distribution $\{\nu^2(i)\}_{i\in \N_0}$. 

Let us take a closer look at $\tilde A$. Note first that \eqref{eq:spr.Hadamard} is valid for $\tilde A$ 
and $\tilde A_{\N}$ (the restriction of $\tilde A$ to $\N\times\N$), since both matrices are 
Hilbert-Schmidt and therefore define bounded linear operators on $\ell_2(\N_0)$, respectively,
$\ell_2(\N)$. Next, by Proposition~\ref{prop:gap}, $\spr(\tilde A) = \tilde \lambda < 1$. Write 
$\tilde \gl_\N = \spr(\tilde A_\N)$ and observe that $\tilde \gl_\N \leq \tilde \gl$. Indeed, by 
\eqref{eq:spr.Hadamard}, for all $n\in\N$ we have
\begin{equation}
\tilde \gl_\N \leq \suptwo{u\in \ell_2(\N)}{\| u \|_2=1} \|\widetilde{A}_\bbN^n u\|_2^{1/n}
= \suptwo{u\in \ell_2(\N_0)}{\| u \|_2=1, u(0) = 0} \|\widetilde{A}^n u\|^{1/n} 
\leq  \|\widetilde{A}^n\|_{\rm op}^{1/n},
\end{equation}
and so we get the claim by letting $n\to \infty$. Therefore there exists a vector $\{\tilde\nu(i)\}_{i\in \N}$ 
with strictly positive entries such that
\begin{equation}
\sum_{j\in\N}  \tilde A(i,j) \frac{\tilde \nu(j)}{ \tilde \nu(i)} = \tilde \gl_\N < 1,
\qquad i \in \N.
\end{equation}
This defines a transition matrix $\tilde Q_{\gd,\gb}$ on $\N_0 \times \N_0$ given by
\begin{equation}
\label{eq:tildeQgdgb}
\begin{aligned}
& \tilde Q_{\gd,\gb}(i,j) = \tilde A(i,j)\, \frac{\tilde \nu(j)}{ \tilde \nu(i)},\quad i,j\in \N,\\
& \tilde Q_{\gd,\gb}(i,0) = 1- \tilde\gl_\N >0,\quad i\in\N,\\
& \tilde Q_{\gd,\gb}(0,0) = 1.
\end{aligned}
\end{equation}
Note that $0$ is an absorbing state for $\tilde Q_{\gd,\gb}$. Let $\{\tilde \eta(i)\}_{i\in\N}$ be 
a left eigenvector of $\tilde A_\N$. We may normalise $\tilde \nu$ and $\tilde \eta$ such 
that $\sum_{i\in \N} \tilde \eta (i) \tilde \nu (i) = 1$. 
Recalling \eqref{eq:M0pm}, from \eqref{eq:CLT2} we obtain
\begin{equation}\label{eq:explicit}
\begin{split}
&e^{-\mu(\gd,\gb) n}\bbZ_{n,\gd,\gb}^{*,\gd,\gb}(S_n=x)\\
&= \frac{1}{(1-\tilde\gl_\N)^2} \bE^{x,\gd,\gb}\left[\frac{\tilde A(M^-_{\sigma -1},0)}{\tilde\nu (M^-_{\sigma-1})} 
\frac{\widehat A(M_0^-,M_0^+)}{\tilde \eta(M_0^-)\nu(M_0^+)}
\frac{\tilde \nu (M_x^+)}{\nu (M_x^+)} \frac{\tilde A(M^{+}_{\tau-1},0)}{\tilde \nu(M^{+}_{\tau-1})} 
\ind_{\{Y^- + Y_x + Y^+ = n\}}\right],
\end{split}
\end{equation}
which completes the proof of \eqref{fghrep}.
\end{proof}

\medskip\noindent
{\bf 2.} 
Henceforth we denote by $\bP^{\gd,\gb}$ the law of the time-homogeneous Markov chain 
$(M^+_y)_{y \geq 0}$ with transition kernel $Q_{\gd,\gb}$. We then have the following 
representation.

\begin{lemma}
\label{lem:CLT3}
There exists an explicit function $u$ such that, for all $\bar x \in \N$,
\begin{equation}
\label{eq:lem.CLT3}
\begin{aligned}
\sum_{x\geq \bar x} 
&e^{-\mu(\gd,\gb)n}\Z_n^{*,\gd,\gb}(S_n = x) = \sum_{a,b,n_0^-,n_0^+} u(n_0^-, n_0^+, a, b)\\
&\times \bP^{\gd,\gb}\Big(Y_{\bar x} \leq n-n_0^--n_0^+,\, 
\exists\, j\in\N\colon\, Y_j = n - n_0^- - n_0^+,\, M_{j-1}^+ = b
~\Big|~ M_0^+ = a \Big).
\end{aligned}
\end{equation}
\end{lemma}

\begin{proof}
Define
\begin{equation}\label{eq:udef}
\begin{split}
u(n_0^-,n_0^+,a,b) &= \frac{1}{(1-\tilde\gl_\N)^2}
\bE^{\gd,\gb}\left[ \frac{\tilde A(M^-_{\sigma-1},0)}{\tilde \nu(M^-_{\sigma-1})} 
\frac{\widehat A(M_0^-,a)}{\tilde \eta (M_0^-)\nu(a)} \ind_{\{Y^- = n_0^-\}} \right]
\bP^{\gd,\gb}(M_0^+ = a)\\ 
&\qquad\qquad \times \bE^{x,\gd,\gb}\left[ \frac{\tilde \nu(M_x^+)}{\nu(M_x^+)} 
\frac{\tilde A(M^+_{\tau -1},0)}{\tilde \nu(M^+_{\tau-1})} 
\ind_{\{Y^+ = n_0^+\}} ~\Big|~ M^+_{x-1}=b \right] .
\end{split}
\end{equation}
(Note that the right-hand side in the line above does not actually depend on $x$,
because $(M^+_y)_{y\ge x}$ is by definition a Markov chain with kernel 
$\tilde Q_{\gd,\gb}$ given in \eqref{eq:tildeQgdgb}.)
By the Markov property and Lemma~\ref{lem:CLT1},
\begin{equation}
\begin{split}
&e^{-\mu(\gd,\gb) n}\bbZ_{n}^{*,\gd,\gb}(S_n=x)\\
&= \sum_{a,b,n_0^-,n_0^+} u(n_0^-,n_0^+,a,b) \bP^{\gd,\gb}\big(Y_x = n - n_0^- - n_0^+, 
M_{x-1}^+ = b \big| M_0^+ = a \big).
\end{split}
\end{equation}
The increments of $(Y_j)_{j\in\N}$ are strictly positive, which means that the events in the line 
above are disjoint for different values of $x$. Therefore, summing over $x$ we get the claim.
\end{proof}

\medskip\noindent
{\bf 3.} 
From Lemma~\ref{lem:CLT3} we see that the fluctuations of $S_n$ are related to the fluctuations 
of $(Y_j)_{j\in\N_0}$. The key ingredient of our proof is the following lemma.

\begin{lemma}
\label{lem:CLT4}
For every $a\in\N_0$, the process $(Y_j)_{j\in\N_0}$ under $\bP^{\gd,\gb}(\,\cdot\,|M_0^+=a)$
satisfies the CLT with mean and variance
\begin{equation}
\mu_Y = -\left[\frac{\partial}{\partial\mu} 
\log \lambda_{\gd,\gb}(\mu)\right]_{\mu = \mu(\gd,\gb)},
\qquad 
\sigma_Y^2 = \left[\frac{\partial^2}{\partial\mu^2}
\log\lambda_{\delta,\beta}(\mu)\right]_{\mu=\mu(\delta,\beta)}, 
\end{equation}
i.e., $(Y_N - N \mu_Y)/(\sigma_Y \sqrt{N}) \to N(0,1)$ in law as $N\to\infty$.
\end{lemma}

\begin{proof}
Recall that under $\bP^{\gd,\gb}$ the process $(M^+_y)_{y\in\N_0}$ is a positive recurrent Markov 
chain, with transition kernel $Q_{\gd,\gb}$ and invariant probability measure $\{\nu^2(i)\}_{i\in\N_0}$.
Following \cite[Lemma 4.1]{K96}, we are going to apply the CLT for Markov chains (see 
Chung~\cite[Theorem 16.1]{Chung}).

The strategy is as follows. For $i\in\N_0$, denote by $T_i$ the $i$-th return time to the state $a$
of $(M^+_y)_{y\in\N_0}$ (with $T_0 = 0$) and set $\ell_j := \max\{i\in\N_0: \ T_i \le j\}$. Then, we 
can decompose
\begin{equation} 
\label{eq:CLT?}
Y_N - N \mu_Y
= M_0^+ - M_N^+ + \sum_{i=1}^{\ell_N} \Big( U_i - \mu_Y (T_i - T_{i-1}) \Big) 
+  \Big(U^{\mathrm{inc}}_{\ell_N+1} - \mu_Y (N - T_{\ell_N}) \Big),
\end{equation}
where
\begin{equation}
U_i = \sum_{y= T_{i-1}}^{T_i-1} (M^+_y + M^+_{y+1} + 1) \,, \qquad
U_{\ell_N+1}^{\mathrm{inc}} 
= \sum_{y= T_{\ell_N}}^{N-1} (M^+_y + M^+_{y+1} + 1).
\end{equation}
By the strong Markov property, under $\bP^{\gd,\gb}(\,\cdot\,|M_0^+=a)$ the random variables 
$(T_i - T_{i-1})_{i\in\N}$ are i.i.d.\ with finite expectation $\mu_T = \bE^{\gd,\gb}(T_1|M_0^+=a) 
= \nu(a)^{-2}$. Consequently, by the strong law of large numbers, $\lim_{N\to\infty} \ell_N / N =
\mu_T^{-1}$ a.s. Also the random variables $(U_i)_{i\in\N}$ are i.i.d., again by the strong Markov 
property. If we show that
\begin{equation} 
\label{eq:topr1}
\mu_Y = -\bigg[\frac{\partial}{\partial\mu} \log \lambda_{\gd,\gb}(\mu)\bigg]_{\mu = \mu(\gd,\gb)} 
= \frac{\bE^{\gd,\gb}(U_1|M_0^+=a)}{\mu_T}
= \frac{\bE^{\gd,\gb}(U_1|M_0^+=a)}{\bE^{\gd,\gb}(T_1|M_0^+=a)},
\end{equation}
then the random variables $U_i - \mu_Y (T_i - T_{i-1})$ are i.i.d.\ and are centered. If we pretend 
that we may replace the upper index $\ell_N$ by $\mu_T^{-1} N$ in \eqref{eq:CLT?}, 
and if we ignore the 
contribution of the ``boundary term'' $U^{\mathrm{inc}}_{\ell_N+1} - \mu_Y (N - T_{\ell_N})$, then
the standard CLT yields $(Y_N - N \mu_Y)/(\sigma_Y \sqrt{N}) \to N(0,1)$ with
\begin{equation} 
\label{eq:topr2}
\sigma_Y^2 
= \frac{\bE^{\gd,\gb}\big((U_1 - \mu_Y T_1)^2 \big| M_0^+=a \big)}{\mu_T}
= \frac{\bE^{\gd,\gb}\big((U_1 - \mu_Y T_1)^2 \big|
M_0^+=a \big)}{\bE^{\gd,\gb}(T_1|M_0^+=a)}.
\end{equation}
But this is indeed justified by Chung~\cite[Theorem 16.1]{Chung}. The boundary term $M_0^+ - M_N^+$ in \eqref{eq:CLT?} is harmless. It only remains to show
that $\mu_Y$ and $\sigma_Y^2$, as defined in the statement of Lemma~\ref{lem:CLT4},
(are finite and) satisfy relations \eqref{eq:topr1}-\eqref{eq:topr2}.

Recalling the definition \eqref{eq:Qgdgb} of the transition kernel $Q_{\gd,\gb}(i,j)$, we can write
\begin{equation}
\label{1relation}
1 = \bP^{\gd,\gb}(T_1 < \infty | M_0^+ = a) 
= \sum_{n\in\N}\,\, \sumtwo{m_1,\ldots,m_{n-1}\neq a}{m_0 = m_n =a} \,\,
\prod_{i=1}^n \frac{A_{\mu,\gd,\gb}(m_{i-1},m_i)}{\gl_{\gd,\gb}(\mu)}
\end{equation}
with $\mu = \mu(\gd,\gb)$. Recalling the definition \eqref{eq:def.Abeta.p.r} of the kernel 
$A_{\mu,\gd,\gb}$, when we differentiate this relation with respect to $\mu$ term by term and 
set $\mu = \mu(\gd,\gb)$, we get precisely relation \eqref{eq:topr1}. Differentiating it a second 
time, we get \eqref{eq:topr2}. We only need an argument analogous to \cite[Eq.\ (4.6)]{K96} to 
interchange differentiation and summation. This follows from an adaptation of 
\cite[Eq.\ (4.7)--(4.11)]{K96}. In particular, \cite[Eq.\ (4.8)]{K96} is replaced in our context by
\begin{equation}
K_R = \sup_{\mu\colon\, |\mu - \mu(\gd,\gb)|\leq \gep} \, 
\sup_{\ell\geq R} \, e^{-\mu \ell + G^*_{\gd,\gb}(\ell)} \downarrow 0, \qquad R\to\infty,
\end{equation}
for $\gep$ small enough, which uses the fact that $G^*_{\gd,\gb}$ is bounded from above and 
$\mu(\gd,\gb)>0$.
\end{proof}

\medskip\noindent
{\bf 4.} 
Fix $c\in \bbR$ and set
\begin{equation}
\label{eq:def.xn}
x_n 
= \frac{n}{\mu_Y} + c\sqrt{n} 
= nv(\gd,\gb)+ c\sqrt{n}
\end{equation}
with $\mu_Y= -[\frac{\partial}{\partial\mu} \log \lambda_{\gd,\gb}(\mu)]_{\mu = \mu(\gd,\gb)}$ as 
in Lemma~\ref{lem:CLT4} (the second equality in \eqref{eq:def.xn} follows from the definition 
in \eqref{eq:speedspec} of the speed $v(\gd,\gb)$).

Let us next look at the probability in \eqref{eq:lem.CLT3} with $\bar x = x_n$. For simplicity, 
we first forget about the constraint $\{\exists\, j\in\N\colon\,Y_j = n - n_0^- - n_0^+, 
M_{j-1}^+ = b\}$. Recall \eqref{eq:def.xn}. We prove that, for fixed $n_0^-,n_0^+\in \N$ and 
$a\in\N_0$,
\begin{equation}
\lim_{n\to \infty} \bP^{\gd,\gb}\big(Y_{x_n} \leq n - n_0^- - n_0^+
~\big|~ M_0^+ = a \big) = \theta(c), \qquad
\theta(c) = \int_c^{+\infty} 
\frac{e^{-z^2/[2\sigma_v(\gd,\gb)^2]}}{\sqrt{2\pi \sigma_v(\gd,\gb)^2}}\,\dd z,
\end{equation}
where $\sigma_v(\gd,\gb)$ is defined in \eqref{eq:formula.variance}. Indeed,
\begin{equation}
\begin{split}
& \bP^{\gd,\gb}\big(Y_{x_n} \leq n - n_0^- - n_0^+ ~|~ M_0^+ = a \big)  \\
& \qquad = \bP^{\gd,\gb}\Big(Y_{x_n} - x_n v(\gd,\gb)^{-1}
\leq -c v(\gd,\gb)^{-3/2} \sqrt{x_n}\,\,[1+o(1)] ~\Big|~ M_0^+ = a \Big),
\end{split}
\end{equation}
and so the claim follows from Lemma~\ref{lem:CLT4} and the fact that $\sigma_v(\gd,\gb)^2 
= v(\gd,\gb)^3 \sigma_Y^2$, by \eqref{eq:formula.variance}.

\medskip\noindent
{\bf 5.} 
Adapting the argument below \cite[Eq.\ (4.20)]{K96}, we show that, for fixed $n_0^-$,
$n_0^+$ and $a,b$,
\begin{equation}
\label{eq:CLTstep5}
\bP^{\gd,\gb}\big(Y_{x_n} \leq n-n_0^--n_0^+,\, \exists\,j\in\N\colon\, 
Y_j = n - n_0^- - n_0^+, 
M_{j-1}^+ = b | M_0^+ = a \big) \stackrel{n\to\infty}{\sim} \theta(c) m(b),
\end{equation}
for an explicit positive constant $m(b)$. Indeed, by the standard ergodic theorem for 
positive recurrent aperiodic Markov chains, $\bP^{\gd,\gb}(M_n^+ = b \ |\ M_0^+ = a) 
\to \nu^2(b)$ as $n\to\infty$, which ensures in particular that $(M_n^+)$ is tight, since
$\sum_{b\in\N_0} \nu^2(b) = 1$. As a consequence, the analogue of \cite[Eq.\ (4.20)]{K96} 
holds, i.e., $\limsup_{b\to\infty} \sup_{n\in\N} \bP^{\gd,\gb}(M_n^+ = b\ |\ M_0^+ = a) = 0$.

\medskip\noindent
{\bf 6.} We may now conclude the proof by showing that 
\begin{equation}
\label{eq:CLTresult}
\lim_{n\to\infty} \bbP_n^{\gd,\gb}\big(S_n \geq n v(\gd,\gb) + c\sqrt{n}\big) = \theta(c).
\end{equation}
First, note that
\begin{equation}
\bbP_n^{\gd,\gb}(S_n = x) 
= \frac{e^{-\mu(\gd,\gb)n}\Z_n^{*,\gd,\gb}(S_n = x)}{e^{-\mu(\gd,\gb)n}\Z_n^{*,\gd,\gb}}, 
\end{equation}
so that, by Lemma~\ref{lem:CLT3},
\begin{equation} \label{eq:CLTstep6}
\begin{aligned}
&\bbP_n^{\gd,\gb}\big(S_n \geq n v(\gd,\gb) + c\sqrt{n}\big) 
= \frac{\sum_{x\geq x_n}\bbP_n^{\gd,\gb}(S_n = x)}{\sum_{x\in \N}\bbP_n^{\gd,\gb}(S_n = x)}\\
&= \frac{\begin{array}{ll} 
&\sum_{a,b,n_0^-,n_0^+} u(n_0^-, n_0^+, a, b) \bP^{\gd,\gb}\big(Y_{x_n} \leq n-n_0^--n_0^+,\\ 
&\qquad\qquad\qquad\qquad \exists\, j\in\N\colon\, Y_j = n - n_0^- - n_0^+,
M_{j-1}^+ = b \,\big|\, M_0^+ = a \big)
\end{array}}
{\sum_{a,b,n_0^-,n_0^+} u(n_0^-, n_0^+, a, b) \bP^{\gd,\gb}\big(\exists\, j\in\N\colon\, 
Y_j = n - n_0^- - n_0^+, M_{j-1}^+ = b \,\big|\, M_0^+ = a \big)}.
\end{aligned}
\end{equation}
Equation \eqref{eq:CLTresult} follows from \eqref {eq:CLTstep5} and \eqref {eq:CLTstep6} by 
letting $n\to\infty$. If we set
\begin{equation}
\label{Jndef}
J(n)  = \sup\{j\in\N\colon\, Y_j \leq n\} - 1,
\end{equation}
then the limit is justified by the following points:
\begin{enumerate}
\item\label{it:1}
$\sum_{n_0^-,\ n_0^+,\ a} u(n_0^-, n_0^+, a, b)  < \infty$
for all $b$,
\item\label{it:2}
$\displaystyle\limsup_{B\to\infty} \sup_{n\in\N} \sumtwo{n_0^-, n_0^+, a}{b > B} 
u(n_0^-, n_0^+, a, b) 
\bP^{\gd,\gb}(n-n_0^- - n_0^+ \in Y, M^+_{J(n)-1}=b \,|\, M_0^+=a) = 0$.
\end{enumerate}
Note that \eqref{it:2} allows us to truncate the 
sums in the numerator and denominator of \eqref{eq:CLTstep6} to $b \le B$, after which \eqref{it:1} allows us to apply 
dominated convergence.
To see why \eqref{it:1} and \eqref{it:2} hold, 
we first note that, by the definition \eqref{eq:def.Abeta.p.r2} of $\tilde A$,
\begin{equation}
	\tilde A(n,0) \stackrel{n\to\infty}{\asymp} e^{-n(\mu + \log 2)} \,,
\end{equation}
where the $\log 2$ comes from $\tilde Q(n,0)$, cf.\ \eqref{eq:def.Qij}, and we note that 
$G^*$ gives a negligible contribution, by Proposition~\ref{pr:asymptGstar}. Since, by 
Lemma~\ref{lem:nurate} below,
\begin{equation}
\label{eq:belo}
\tilde\nu(n) \stackrel{n\to\infty}{\asymp} \nu(n) \stackrel{n\to\infty}{\asymp} e^{-rn}, 
\qquad \mu = \log\cosh(r) > r - \log 2,
\end{equation}
it follows that
\begin{equation}
\sup_{n\in\N} \frac{\tilde A(n,0)}{\tilde\nu(n)} < \infty,
\end{equation}
from which, recalling \eqref{eq:udef}, we deduce that for some positive constant $C$,
\begin{equation}\label{eq:doubl}
\begin{aligned}
&\sum_{n_0^-, n_0^+} u(n_0^-, n_0^+, a, b) \leq C \, \phi(a,b)
\, \bP^{\gd,\gb}(M_0^+ = a),\\
&\text{where} \quad
\phi(a,b) = \bE^{\gd,\gb}\Big[ \frac{\hat A(M_0^-, a)}{\tilde\eta (M_0^-) \nu(a)} \Big]\, 
\bE^{x,\gd,\gb}\Big[ \frac{\tilde \nu(M^+_x)}{\nu(M^+_x)} \Big| M^+_{x-1} = b \Big].
\end{aligned}
\end{equation}
We remind that the second expectation does not actually depend on $x$, by
the Markov property. Moreover, we claim that for every $\epsilon > 0$
there exists $C < \infty$ such that
\begin{equation} 
\label{eq:claimju}
\bE^{x,\gd,\gb}\Big[ \frac{\tilde \nu(M^+_x)}{\nu(M^+_x)} \Big| M^+_{x-1} = b \Big]
\le C e^{\epsilon b} \,, \qquad b \in \N_0 \,.
\end{equation}
In fact, recalling the definition \eqref{eq:Qgdgb} of the transition kernel $Q_{\gd,\gb}$
of $(M_n^+)$,
\begin{equation}
\bE^{x,\gd,\gb}\Big[ \frac{\tilde \nu(M^+_x)}{\nu(M^+_x)} \Big| M^+_{x-1} = b \Big]
= \sum_{k\in \N_0} \frac{\tilde \nu(k)}{\nu(k)} \, A(b,k) \, \frac{\nu(k)}{\nu(b)}
= \frac{1}{\nu(b)} \sum_{k\in \N_0} \tilde \nu(k) \, A(b,k) \,.
\end{equation}
Since $\nu(b) \ge c \, e^{-(r+\epsilon/2)b}$ (recall \eqref{eq:belo}), it suffices to show 
that $\sum_{k\in\N_0} \tilde \nu(k) \, A(b,k) \le C \, e^{-(r-\epsilon/2)b}$. To this 
end, we show that
\begin{equation}
\label{eq:toshowthat}
\sum_{b\in \N_0} e^{(r-\epsilon/2)b} 
\bigg( \sum_{k\in \N_0} \tilde \nu(k) \, A(b,k) \bigg) < \infty \,.
\end{equation}
Recalling the definition \eqref{eq:def.Abeta.p.r} of $A$, we note that, since $G^*$ 
is bounded from above (by Proposition~\ref{pr:asymptGstar}),
\begin{equation}
A(b,k) \le C \, e^{-\mu(b+k)} \, \binom{b+k}{b} \, \frac{1}{2^{b+k+1}}
= \frac{C}{2} \, e^{-\mu(b+k)} \, \bP(S_{b+k}=b-k),
\end{equation}
with $S_n$ the simple symmetric random walk on $\Z$. Since $\tilde\nu(k) \le C 
e^{-(r-\epsilon/2)k}$ (recall \eqref{eq:belo}), setting $i = b-k$ and $j = b+k$, we get
\begin{equation}
\begin{split}
\sum_{b\in \N_0} e^{(r-\epsilon/2)b} 
& \bigg( \sum_{k\in \N_0} \tilde \nu(k) \, A(b,k) \bigg)
\le C \sum_{j\in\N_0} \sum_{i\in\Z}
e^{(r-\epsilon/2)i} e^{-\mu j} \bP(S_j=i)\\
& = C \sum_{j\in\N_0} e^{-\mu j} \bE[e^{(r-\epsilon/2)S_j}] 
= C \sum_{j\in\N_0} e^{(\log \cosh(r-\epsilon/2)-\mu )j} < \infty,
\end{split}
\end{equation}
because $\mu = \log\cosh(r) > \log\cosh(r-\epsilon)$ (recall \eqref{eq:belo}).
We have proved \eqref{eq:toshowthat}, and hence \eqref{eq:claimju}.

To prove \eqref{it:1}, recall \eqref{eq:M0pm} and \eqref{eq:hatA}. Bound
$\widehat A(\ell, a) \le C e^{-\mu(\ell + a)}$ (because $Q(\ell,a) \le 1$ and 
$G^*$ is bounded from above by Proposition~\ref{pr:asymptGstar}), to  
obtain via \eqref{eq:doubl} that
\begin{equation}
\begin{aligned}
\sum_{a\in\N_0} \bE^{\gd,\gb}\Big[ \frac{\widehat A(M_0^-, a)}{\tilde\eta (M_0^-) \nu(a)} \Big] 
\bP^{\gd,\gb}(M_0^+ = a) 
&\leq \sum_{a,\ell \in \N_0} \frac{\widehat A(\ell, a)}{\tilde\eta (\ell) \nu(a)} 
\nu(a)^2 \tilde\eta(\ell) \tilde\nu(\ell)\\
&\leq C \sum_{a,\ell \in \N_0} e^{-\mu(a+\ell)} \nu(a) \tilde\nu(\ell) < \infty.
\end{aligned}
\end{equation}
To prove \eqref{it:2}, we first note that, by Cauchy-Schwarz,
\begin{equation}
\begin{split}
& \bP^{\gd,\gb}(M_0^+ = a, n-n_0^- - n_0^+\in Y, M^+_{J(n-n_0^- - n_0^+)}=b) \\
& \qquad \le \sqrt{\bP^{\gd,\gb}(M_0^+ = a)}
\sqrt{\bP^{\gd,\gb}(n-n_0^- - n_0^+\in Y, M^+_{J(n-n_0^- - n_0^+)}=b)}
\le \nu(a) \, \nu(b),
\end{split}
\end{equation}
because $\bP^{\gd,\gb}(M_0^+ = a) = \nu(a)^2$ (recall \eqref{eq:M0pm}), and the
relation
\begin{equation}
\label{eq:magic}
\sup_{k\in\N} \bP^{\gd,\gb}(k \in Y, M^+_{J(k)}=b) = \nu(b)^2
\end{equation}
proved below.
Recalling \eqref{eq:udef} and \eqref{eq:doubl}--\eqref{eq:claimju}, we can finally estimate 
\begin{equation}
\begin{aligned}
&\sumtwo{n_0^-, n_0^+, a}{b > B} u(n_0^-, n_0^+, a, b) 
\bP^{\gd,\gb}(n-n_0^- - n_0^+ \in Y, M^+_{J(n)}=b |\, M_0^+=a)\\
& \qquad \leq C \sum_{a\in\bbN, b>B} \phi(a,b) \, \nu(a) \, \nu(b)
\leq C \Bigg(\sum_{a\in\bbN_0}
\bE^{\gd,\gb}\Big[ \frac{\hat A(M_0^-, a)}{\tilde\eta (M_0^-)} \Big] \Bigg)
\Bigg( \sum_{b > B} e^{\epsilon b} \nu(b) \Bigg) \,.
\end{aligned}
\end{equation}
Since $\nu(b) \asymp e^{-r b}$ (recall \eqref{eq:belo}), the sum in the second parenthesis 
converges, and vanishes as $B\to\infty$. To complete the proof of \eqref{it:2}, it suffices to 
show that the first parenthesis is finite. Similarly as above, recalling \eqref{eq:M0pm}, 
\eqref{eq:belo} and bounding $\hat A(i,j) \le C \, e^{-\mu(i+j)}$ (recall \eqref{eq:hatA}),
we get
\begin{equation}
\sum_{a\in \N_0} \bE^{\gd,\gb}\Big[ \frac{\hat A(M_0^-, a)}{\tilde\eta (M_0^-)} \Big] 
\leq C \sum_{a,\ell \in \N_0} e^{-\mu(a+\ell)}\tilde\nu(\ell) < \infty.
\end{equation}

It only remains to prove \eqref{eq:magic}. If we define $p^{(n)}(b) = \sup_{k \leq n} 
\bP^{\gd,\gb}(k\in Y, M^+_{J(k)}=b)$, then
\begin{equation}
\label{eq:supequalsnu}
\sup_{n \in \N_0} \bP^{\gd,\gb}(n\in Y, M^+_{J(n)}=b) = p^{(\infty)}(b) = \lim_{n\to\infty} p^{(n)}(b).
\end{equation}
Note that
\begin{equation}
p^{(\infty)}(b) \geq  p^{(2b+1)}(b) = \nu(b)^2,
\end{equation}
while, for all $b\geq 1$,
\begin{equation}
\begin{aligned}
&\bP^{\gd,\gb}(n\in Y, M^+_{J(n)}=b) \\
&\quad = \left\{
\begin{array}{ll}
0, & \mbox{if}\ 2b+1>n\\
\nu(b)^2, & \mbox{if}\ 2b+1=n\\
\sum_{1\leq a \leq n-2b-1} \bP(n-2b-1\in Y, M^+_{J(n-2b-1)} = a) Q^{\gd,\gb}(a,b), & \mbox{if}\ 2b+1 <n.
\end{array}
\right.
\end{aligned}
\end{equation}

If $2b+1<n$, then we have in particular that $p^{(n)}(b) \leq \sum_{1\leq a \leq n-b} p^{(n-1)}(a) 
Q^{\gd,\gb}(a,b)$. Therefore
\begin{equation}
\label{eq:ineq_pn0}
p^{(n)}(b) \leq \nu(b)^2 \ind_{\{2b+1=n\}} + (p^{(n-1)}Q^{\gd,\gb})(b)\ind_{\{2b+1<n\}},
\end{equation}
from which we get for $n\in \N$ (with the convention $p^{(0)} = 0$),
\begin{equation}
\label{eq:ineq_pn}
\begin{aligned}
\|p^{(n)}\|_1 &\leq \nu(n)^2  + \|p^{(n-1)}Q^{\gd,\gb}\|_1\\
&\leq \nu(n)^2  + \|p^{(n-1)}\|_1,
\end{aligned}
\end{equation}
because $Q^{\gd,\gb}$ is stochastic. Since $\|p^{(1)}\|_1 = \nu(1)^2$, we get $\|p^{(n)}\|_1 \leq 
\sum_k \nu(k)^2 =1$, for $n\geq1$, so $\|p^{(\infty)}\|_1 \leq 1$ by Fatou's lemma. By letting $n\to\infty$ 
in \eqref{eq:ineq_pn0}, we obtain $p^{(\infty)}(b) \leq (p^{(\infty)}Q^{\gd,\gb})(b)$ for all $b \geq1$. 
But $\|p^{(\infty)}Q^{\gd,\gb}\|_1 \leq \|p^{(\infty)}\|_1$, which implies that $p^{(\infty)} = p^{(\infty)}
Q^{\gd,\gb}$, and there exists a constant $c$ such that $p^{(\infty)} = c\nu^2$. Necessarily, 
$c\leq 1$, so $p^{(\infty)}(b) \leq \nu(b)^2$. This completes the proof of \eqref{eq:supequalsnu},
and hence the proof of the central limit theorem for the speed.
\end{proof}

\medskip 
In Appendix~\ref{AppB} we list the key ingredients necessary to extend the above argument to 
prove the central limit theorem for the charge.


\subsection{Laws of large numbers for the speed and the charge}
\label{ss:LLN}

In this section we prove Theorems~\ref{thm:speed}--\ref{thm:charge}.

\begin{proof}
If $\beta \neq \beta_c(\delta)$ (``off the critical curve''), then the following hold:
\begin{itemize} 
\item
The rate function $I^v_{\gd,\gb}$ for the speed in \eqref{Ivdef} has a unique zero 
at $v(\delta,\beta) = \frac{\partial\mu}{\partial\gamma}(\gd,\gb,0)$ (corresponding to 
$\gamma=0$), which equals \eqref{eq:speedspec}.
\item
The rate function $I^\rho_{\gd,\gb}$ for the charge in \eqref{Irhodef} has a unique 
zero at $\rho(\delta,\beta)=\frac{\partial\mu}{\partial\delta}(\gd,\gb)$ (corresponding to 
$\gamma'=0$), which equals \eqref{eq:chargespec}.  
\end{itemize}
Hence the laws of large numbers follow from the large deviation principles. (Note: 
The condition $S_n>0$ in \eqref{LLNspeed} is put in to fix a direction for the speed: 
by the symmetry of the simple random walk the same large deviation principle holds to the left.) 
If, on the other hand, $\beta = \beta_c(\delta)$ (``on the critical curve''), then the rate 
functions have a horizontal piece, and hence the laws of large numbers need a separate 
argument.

We first note that $v(\gd,\gb_c(\gd))>0$, as it appears from \eqref{eq:speedspec} and Lemma \ref{lem:slopecrit1}. Thanks to the strict convexity of $I^v_{\gd,\gb_c(\gd)}$ at the right of $v(\gd,\gb_c(\gd))$, 
it suffices to prove that
\begin{equation}
\lim_{n\to\infty} \bbP_n^{\gd,\gb}\big(S_n \leq (1-\gep)v(\gd,\gb_c(\gd))n \mid S_n >0\big) = 0
\qquad \forall\,\gep>0.
\end{equation} 
We will use the representation of $\bbP_n^{\gd,\gb}(S_n = x)$, $x\in \bbZ$, as developed in the 
proof of Theorem~\ref{thm:CLTspeedcharge}. Reproducing Steps 1--6 with $x_n = (1-\gep)
v(\gd,\gb_c(\gd))n$, we are left with controlling
\begin{equation}
\bP^{\gd,\gb}(Y_{x_n} \geq n) \quad \mbox{ with } \quad Y_k = \sum_{y=0}^{k-1} (2 M^+_y + 1),
\end{equation}
where $(M_y^+)_{y\in\N}$ is the stationary Markov chain whose transition matrix is given by 
\eqref{eq:Qgdgb}, for the choice of parameters $\beta = \gb_c(\gd)$ and $\mu = \mu(\gd,
\gb_c(\gd)) =0$. It will be shown in the proof of Lemma~\ref{lem:slopecrit1} below that 
$\sum_{i\in \N_0} i \nu(i)^2<\infty$, which means that the Markov chain has finite mean. 
Furthermore, $\bE(Y_n) = \frac{n}{v(\gd,\gb_c(\gd))}$. Since $n\geq \frac{1+\gep}
{v(\gd,\gb_c(\gd))} x_n$, we only need that
\begin{equation}
\lim_{n\to\infty} \bP(Y_n - \bE[Y_n] \geq \gep n) = 0,
\end{equation}
but this follows from the law of large numbers for stationary Markov chains.
\end{proof}


\section{Asymptotic properties: proof of the main theorems}
\label{s:asymptotics}

Section~\ref{ss:asympcritcurve} contains the proof of the scaling of $\beta_c(\delta)$
for $\delta\to\infty$ stated in Theorem~\ref{thm:asymp.cc}(2). Section~\ref{orderphtr}
explains why the phase transition is first order as claimed in Theorem~\ref{thm:orderphtr}.
Section~\ref{ss:WIL} contains the proof of the scaling of $\beta_c(\delta)$ for $\delta 
\downarrow 0$ stated in Theorem~\ref{thm:asymp.cc}(1), and also deals with the weak 
interaction limit $\delta,\beta \downarrow 0$ in Theorem~\ref{thm:Fscalbeta}.


\subsection{Scaling of the critical curve}
\label{ss:asympcritcurve}

In this section we give the proof of Theorem~\ref{thm:asymp.cc}(2). We begin by stating a rough 
but helpful lemma. Recall 
\eqref{eq:rel.G.star}: 
\begin{equation}
\label{reann}
G^*_{\delta,\beta}(\ell) = \log \bbE\left[e^{\delta \Omega_\ell-\beta\Omega_\ell^2}\right]
= \log \bbE\left[e^{ \delta \Omega_\ell (1-\frac{\beta}{\delta} \Omega_\ell)}\right].
\end{equation}

\begin{lemma}
\label{SBcond}
The following hold:\\
{\rm (1)} If $G^*_{\delta,\beta}(\ell)\leq 0$ for all $\ell\in\N$, then $\beta \geq \beta_c(\gd)$.\\
{\rm (2)} If there exists $\ell \in \N$ such that $\ell^{-1}G^*_{\delta,\beta}(\ell) > \log 2$,
then $\beta < \beta_c(\gd)$.
\end{lemma}

\begin{proof}
(1) This is immediate from \eqref{eq:Fdef} and (\ref{eq:identity*}--\ref{eq:ZnZn*rel}).\\
(2) Let $\ell \in \N$ such that $\ell^{-1}G^*_{\delta,\beta}(\ell) > \log 2$. We restrict the partition 
function $\bbZ_n^{*,\delta,\beta}$ to the trajectory $s=(s_i)_{i\in\N_0}$ defined by $s_0=0$ 
and
\begin{equation}
\begin{aligned}
s_{2a\ell+2b+1}
&=2a+1, &a\in \N_0, \,  0 \leq b < \ell,\\
s_{2a\ell+2(b+1)}
&=2a+2, &a\in \N_0, \,  0 \leq b <  \ell.
\end{aligned}
\end{equation}
For $a\in \N$, we can estimate
\begin{equation}
\label{dert}
\bbZ_{2a\ell}^{*,\delta,\beta} \geq e^{2a G^{*}_{\delta,\beta}(\ell)}\,  
\bP(S_{[0,2a\ell]}=s_{[0,2a\ell]}).
\end{equation}
Take $\frac{1}{2a\ell} \log$ on both sides of \eqref{dert}, let $a\to \infty$ and recall \eqref{eq:identity*},  
to obtain that $F^*(\delta,\beta) \geq \ell^{-1}G^{*}_{\delta,\beta}(\ell)-\log 2>0$.
\end{proof}

With the help of Lemma~\ref{SBcond} the asymptotics of $\beta_c(\delta)$ for $\delta\to\infty$ is 
proved as follows.

\medskip\noindent 
\emph{Lattice case}:
Recall \eqref{eq:Tlat}. We prove that there exists a $c_1>0$ such that
\begin{equation}
\label{resfi}
\frac{\delta}{T}-\frac{c_1}{T^2}\leq \beta_c(\delta)\leq \frac{\delta}{T} \qquad \forall\,\delta>0.
\end{equation}

To prove the upper bound in \eqref{resfi}, note that $\bbP(\Omega_\ell\in (0,T))=0$ for all 
$\ell\in \N$. Therefore, choosing $\beta=\delta/T$ in the right-hand side of \eqref{reann}, 
we obtain that $\delta \Omega_\ell (1-\frac{1}{T} \Omega_\ell) \leq 0$ for all $\ell\in\N$ and 
$\bbP$-a.e.\ $\omega$. Consequently, $G^*_{\delta,\delta/T}(\ell)\leq 0$ for $\delta>0$ 
and $\ell\in \N$ which, by Lemma~\ref{SBcond}(1), implies that  $\beta_c(\delta)\leq \delta/ T$
for $\delta>0$.

To prove the lower bound in \eqref{resfi}, note that there exists an $\ell\in \N$ such that 
$\bbP(\Omega_\ell=T)>0$ (Durrett~\cite[Theorem 3.5.2]{Du91}). Therefore, choosing 
$\beta = \beta(\delta,T) =\delta/T - c_1/T^2$ in the right-hand side of \eqref{reann}, we obtain 
\begin{align}
\label{reann1}
G^*_{\delta,\beta(\delta,T)}(\ell)
&\geq \log \bbE\left[e^{ \delta \Omega_\ell \left(1-\frac{\beta(\delta,T)}{\delta} \Omega_\ell\right)}
\ind_{\{\Omega_\ell=T\}}\right]\\ \nonumber
&= \log \bbP(\Omega_\ell=T) + \delta T\left(1-\tfrac{\beta(\delta,T)}{\delta} T\right)
= \log \bbP(\Omega_\ell=T)+c_1,
\end{align}
hence
\begin{equation}
\label{reann2}
\frac{1}{\ell}\, G^*_{\delta,\beta(\delta,T)}(\ell)-\log 2 
\geq \frac{1}{\ell}  \log\bbP(\Omega_{\ell}=T)+\frac{c_1}{\ell} - \log 2.
\end{equation}
The right-hand side of \eqref{reann2} is strictly positive for $c_1$ large enough, uniformly in 
$\delta$. Therefore, by Lemma~\ref{SBcond}(2), we have $(\delta,\beta(\delta,T))\in \cB$, 
which completes the proof of \eqref{resfi}.

\medskip\noindent 
\emph{Non-lattice case}: 
We show that 
\begin{equation}
\label{nlc}
\liminf_{\delta\to \infty} \beta_c(\delta)/\delta=\infty.
\end{equation}
Pick $C>0$. The proof of \eqref{nlc} will be complete once we show that $(\delta,C\delta)\in\cB$ 
for $\delta$ large enough. To that end we note that for any $\eta > 0$ there exists an $\ell \in \N$ 
such that $\bbP(\Omega_\ell \in (0,\eta)) > 0$ (by the non-lattice assumption; recall \eqref{eq:Tlat} 
with $T=0$). Choosing $\eta = 1/C$, we get
\begin{align}
\label{drt}
\begin{split}
G^*_{\delta,C\delta}(\ell)
 & \geq \log \bbE\left[e^{\delta\Omega_\ell(1 - C\Omega_\ell)} 
\ind_{\big\{\Omega_\ell\in (0,\frac{1}{C})\big\}}\right]  \\
&\geq \log \bbP\big(\Omega_\ell \in (0,\tfrac{1}{C})\big) 
+ \delta \, \bbE\big[\Omega_\ell(1-C\Omega_\ell) \mid \Omega_\ell \in (0,\tfrac{1}{C})\big].
\end{split}
\end{align}
Since $\bbE[\Omega_\ell(1-C\Omega_\ell) \mid \Omega_\ell \in (0,\tfrac{1}{C})] > 0$,
the right-hand side tends to $\infty$ as $\delta \to \infty$, and hence
$\ell^{-1} G^*_{\delta,C\delta}(\ell) > \log 2$ for $\delta$ large enough.
This completes the proof of \eqref{nlc} by Lemma~\ref{SBcond}.

\medskip\noindent
\emph{Non-lattice case with density}:  
Suppose that $\omega_1$ has a density $g_1$ with respect to the Lebesgue measure $\lambda$. 
Then, for all $\ell\in \N$, $\Omega_\ell$ has a density $g_\ell$ with respect to $\lambda$.

\begin{lemma}
\label{lem:dc1}
$\liminf_{\delta\to \infty} \beta_c(\delta)/(\delta^2/\log \delta) \geq \frac{1}{4}$.
\end{lemma}

\begin{proof}
Using \eqref{reann}, we may write
\begin{equation}
G^*_{\delta,\beta}(\ell) \geq \log \bbE\left[ e^{\delta \Omega_\ell (1- \frac{\beta}{\delta}\gO_\ell)} 
\ind_{\{\gO_\ell \in (0,\frac{\delta}{\beta})\}} \right].
\end{equation}
Suppose that $g_1$ has a finite $L^\infty(\lambda)$-norm. Then $g_\ell$ is bounded and 
continuous for $\ell\in\N\backslash\{1\}$, with $g_\ell(0) > 0$ for $\ell$ large enough,
by the local central limit theorem (see Feller~\cite[Theorem 2 in Section XV.5]{Fe71}).
By continuity, for every $\epsilon \in (0,1)$ there exists a $\gamma > 0$ such that $(1-\epsilon) 
g_\ell(0) \leq g_\ell(x) \leq (1+\epsilon) g_\ell(0)$ for all $x \in (0,\gamma)$. If $\frac{\delta}{\beta} 
< \gamma$, then we can write
\begin{equation}
\begin{aligned}
G^*_{\delta,\beta}(\ell) &= \log \int_\R e^{\delta s (1- \frac{\beta}{\delta}s)} 
\ind_{\{s\in (0,\frac{\delta}{\beta})\}} g_\ell(s)\dd s.\\
& \geq \log \Big[ (1-\epsilon)g_\ell(0) \int_\R e^{\delta s (1- \frac{\beta}{\delta}s)} 
\ind_{\{s\in (0,\frac{\delta}{\beta})\}} \dd s\Big]\\
&\geq \log[(1-\epsilon)g_\ell(0)] + \log \Big[ \int_\R e^{\frac{\delta^2}{\beta} 
s(1-s)} \tfrac{\delta}{\beta} \ind_{\{s\in (0,1)\}} \dd s \Big].
\end{aligned}
\end{equation}
Choosing $\beta = C \delta^2/\log \delta$, with $C > 0$ fixed and $\delta$ large enough, 
we get that the condition $\delta/\beta < \gamma$ is satisfied and
\begin{equation}
G^*_{\delta,\beta}(\ell) \geq  \log \Big[ \int_\R \delta^{\frac{s(1-s)}{C}} 
\ind_{\{s\in (0,1)\}}\dd s \Big] - \log \delta + \log \log \delta + \log[(1-\epsilon)g_\ell(0)] - \log C.
\end{equation}
For every $\epsilon\in (0,1)$, there exists a $C_\gep>-\infty$ such that $\log[\int_\R\delta^{\frac{s(1-s)}{C}} 
\ind_{\{s\in (0,1)\}}\dd s] \geq \frac{1-\gep}{4C} \log \delta + C_\gep$, and so $\lim_{\delta\to\infty} 
G^*_{\delta,\beta}(\ell)=\infty$ when $\frac{1-\gep}{4C}>1$. Since $\epsilon \in (0,1)$ is arbitrary, it 
follows from Lemma~\ref{SBcond} that $(\delta,C \delta^2/\log\delta)\in\cB$ for any $C<\frac14$
and $\delta$ large enough.
\end{proof}

\begin{lemma}
\label{lem:dc2}
$\limsup_{\delta\to \infty} \beta_c(\delta)/(\delta^2/ \log \delta) \le \frac{1}{4}$.
\end{lemma}

\begin{proof}
Here we assume that the density $g_1$ of $\omega_1$ is bounded, i.e., $\|g_1 \|_\infty < \infty$. 
Then $\| g_\ell \|_\infty \le \|g_1\|_\infty$ for all $\ell\in\N$. Note that, for every $\epsilon > 0$,
\begin{equation}
\begin{split}
\bbE\left[e^{ \delta \Omega_\ell (1-\frac{\beta}{\delta} \Omega_\ell)}
\ind_{\{\Omega_\ell \not\in (0,\frac{\delta}{\beta})\}}\right]
& \le e^{-\delta \epsilon} 
+ \bbP(\Omega_\ell \in [-\epsilon,0] \cup [\tfrac{\delta}{\beta},\tfrac{\delta}{\beta}+\epsilon])
\le e^{-\delta \epsilon} + 2\epsilon\|g_1\|_\infty,
\end{split}
\end{equation}
because $\delta \Omega_\ell (1-\frac{\beta}{\delta} \Omega_\ell) \le - \delta \epsilon$
when $\Omega_\ell < -\epsilon$ or $\Omega_\ell > \frac{\delta}{\beta}+\epsilon$,
and $\delta \Omega_\ell (1-\frac{\beta}{\delta} \Omega_\ell) \le 0$ on the whole domain
of integration $\{\Omega_\ell \not\in (0,\frac{\delta}{\beta})\}$. Choosing $\epsilon 
= \frac{1}{\sqrt{\delta}}$, we get
\begin{equation}
\lim_{\delta \to \infty}
\bbE\left[e^{ \delta \Omega_\ell (1-\frac{\beta}{\delta} \Omega_\ell)}
\ind_{\{\Omega_\ell \not\in (0,\frac{\delta}{\beta})\}}\right] = 0
\end{equation}
uniformly in $\ell\in\N$. On the other hand, since $\max_{0 \le z \le \frac{\delta}{\beta}}
 \delta z(1-\frac{\beta}{\delta}z) = \frac{\delta^2}{4\beta}$, we can write
\begin{equation}
\bbE\left[e^{ \delta \Omega_\ell (1-\frac{\beta}{\delta} \Omega_\ell)}
\ind_{\{\Omega_\ell \in (0,\frac{\delta}{\beta})\}}\right] 
\le e^{\frac{\delta^2}{4\beta}} \, \bbP(\Omega_\ell \in (0,\tfrac{\delta}{\beta}))
\le e^{\frac{\delta^2}{4\beta}} \, \|g_1\|_\infty \, \frac{\delta}{\beta}.
\end{equation}
We now choose $\beta= C \delta^2/\log \delta$ with $C > 0$, to get
\begin{equation}
\bbE\left[e^{ \delta \Omega_\ell (1-\frac{\beta}{\delta} \Omega_\ell)}
\ind_{\{\Omega_\ell \in (0,\frac{\delta}{\beta})\}}\right] 
\le \|g_1\|_\infty \, \delta^{\frac{1}{4C}} \, \frac{\log \delta}{C \delta},
\end{equation}
and when $C > \frac{1}{4}$ the right-hand side vanishes as $\delta\to\infty$
uniformly in $\ell\in\N$. Altogether, we have shown that $\lim_{\delta \to \infty} 
G^*_{\delta,\beta}(\ell) = -\infty$ for any $C > \frac{1}{4}$ and $\beta = C 
\delta^2/\log \delta$, uniformly in $\ell\in\N$. It follows from Lemma~\ref{SBcond} 
that $(\delta,C \delta^2/\log\delta)\in \cS$ for any $C>\frac14$ and $\delta$ 
large enough.
\end{proof}


\subsection{Order of the phase transition}
\label{orderphtr}
 
In this section we give the proof of Theorem~\ref{thm:orderphtr}. It follows from 
\eqref{mudeltabetadef} that
\begin{equation}
\frac{\partial}{\partial\mu} \lambda_{\delta,\beta}(\mu(\delta,\beta))
\,\frac{\partial\mu}{\partial\beta}(\delta,\beta) 
+ \frac{\partial}{\partial\beta} \lambda_{\delta,\beta}(\mu(\delta,\beta))
= 0, \qquad (\delta,\beta) \in \intr(\cB).
\end{equation}
Since $F^*(\delta,\beta)= \mu(\delta,\beta)$ by \eqref{varc}, and $F^*(\delta,\beta_c(\delta))=0$,
Lemmas~\ref{lem:slopecrit1}--\ref{lem:slopecrit2} below imply \eqref{Cidscal}--\eqref{Cid} in
Theorem~\ref{thm:orderphtr}.

\begin{lemma} 
\label{lem:slopecrit1}
For all $\delta \in (0,\infty)$,
\begin{equation}
\left[-\frac{\partial \lambda_{\delta,\beta_c(\delta)}(\mu)}{\partial \mu} \right]_{\mu=0} \in (0,\infty).
\end{equation}
\end{lemma}

\begin{proof}
For $(i,j) \in \N_0^2$,
\begin{equation} \label{eq:derA}
\left[- \frac{\partial}{\partial \mu}A_{\mu,\delta,\beta}(i,j) \right]_{\mu=0} 
= (i+j+1) A_{0,\delta,\beta}(i,j).
\end{equation}
In the following, $\delta$ is kept fixed and $\beta = \beta_c(\delta)$, which is the value of $\beta$ 
for which $\lambda_{\delta,\beta}(0) = 1$. Let $\chi=\chi_{0,\delta,\beta_c(\delta)}$ be the right 
eigenvector of $A_{0,\delta,\beta_c(\delta)}$ associated with the eigenvalue $1$. Recall 
Proposition~\ref{pr:prop.A.lambda.mu}. Since $A_{0,\delta,\beta_c(\delta)}$ is Hilbert-Schmidt 
and symmetric, $\chi \in \ell_2(\N_0)$ and $\chi$ is a left eigenvector as well. We choose $\chi$ 
such that $\| \chi \|_2 = 1$. Our starting point is the relation
\begin{equation}
\label{partlambdarel}
\left[ -\frac{\partial \lambda_{\delta,\beta_c(\delta)}(\mu)}{\partial \mu} \right]_{\mu=0} 
= \sum_{i,j \in \N_0} \chi(i)\, (i+j+1) A_{0,\delta,\beta_c(\gd)}(i,j) \, \chi(j),
\end{equation}
which follows by \eqref{eq:derA} and by $\langle \chi, \frac{\partial}{\partial \mu}\chi\rangle
= \frac{1}{2} \frac{\partial}{\partial \mu} \|\chi\|_2 = 0$.
From Proposition~\ref{pr:asymptGstar} in Appendix~\ref{AppA}, we know that
\begin{equation}
\label{eq:G*bd}
e^{G^*_{\delta,\beta}(i+j+1)} \leq \frac{C}{\sqrt{(i+j+1)}}\,
\qquad i,j\in\N_0,
\end{equation}
for some constant $C$. Therefore, recalling \eqref{eq:def.Abeta.p.r} and \eqref{eq:Q}, 
we need to show that
\begin{equation}
\label{sumfinite}
\sum_{i,j\in \N_0} \chi(i) \chi(j) \sqrt{i+j+1}\, \bP(S_{i+j} = i-j) < \infty.
\end{equation}
This is done in 5 steps. In Steps 1--4 we derive successively stronger tail estimates on 
$\chi$. In Step 5 we use these to prove \eqref{sumfinite}. In what follows, $C$ is a
constant that may change from line to line (and depend on other choices of constants).

\medskip\noindent 
{\bf 1.} 
Estimate
\begin{equation}
\label{est1chi}
\begin{aligned}
\sum_{i \geq k} \chi(i)^2 
&= \sum_{i \geq k} \sum_{j\in\N_0} A_{0,\delta,\beta}(i,j) \chi(i)\chi(j)\\
&\leq C \sum_{i \geq k} \chi(i) \sum_{j \in \N_0} \frac{1}{\sqrt{i+j}}\,\bP(S_{i+j}= i-j)\, \chi(j) \\
&\leq C \left[\sum_{i \geq k} \chi(i)^2\right]^{1/2}  
\left[ \sum_{i \geq k} \left( \sum_{j\in\N_0} 
\frac{1}{\sqrt{i+j}}\,\bP(S_{i+j}= i-j)\, \chi(j) \right)^2\, \right]^{1/2}
\end{aligned}
\end{equation}
and
\begin{equation}
\label{est2chi}
\begin{aligned}
&\sum_{i \geq k} \left( \sum_{j\in\N_0} 
\frac{1}{\sqrt{i+j}}\, \bP(S_{i+j}= i-j)\, \chi(j) \right)^2
\leq \sum_{i \geq k} \sum_{j\in\N_0} \frac{1}{i+j}\, \bP(S_{i+j}= i-j)^2\\
&\leq \sum_{ {u \geq k} \atop {v\in\Z} } \frac{1}{u} \bP(S_u = v)^2
\leq \sum_{u \geq k} \frac{1}{u} \bP(S_{2u} = 0)
\leq \sum_{u \geq k} Cu^{-\tfrac32} \leq C  k^{-\tfrac12}.
\end{aligned}
\end{equation}
Combining \eqref{est1chi}--\eqref{est2chi}, we get 
\begin{equation}
\label{est3chi}
\sum_{i \geq k} \chi(i)^2 \leq C k^{-\tfrac12}.
\end{equation}
Abbreviating $\sigma_2(k) = \sum_{i \geq k} \chi(i)^2$, we find that, for any 
$\ga \in (0,\tfrac12)$, 
\begin{equation}
\label{est4chi}
\begin{aligned}
&\sum_{i\in\N_0} i^\ga \chi(i)^2 = \sum_{i\in\N_0} i^\ga [\sigma_2(i) - \sigma_2(i+1)]\\
&\qquad = \sum_{i\in\N} [i^\ga - (i-1)^\ga] \sigma_2(i) \leq C \sum_{i\in\N} i^{\ga-1} \sigma_2(i) 
\leq C \sum_{i\in\N} i^{\alpha - \tfrac32} < \infty.
\end{aligned}
\end{equation}

\medskip\noindent
{\bf 2.} 
Next we use \eqref{est4chi} to prove that $\chi\in\ell_r(\N_0)$ for all $r>\tfrac43$. 
Indeed, let $\gep \in (0,\tfrac13)$ and $r=\tfrac43+\gep$, and use H\"older's inequality to 
estimate
\begin{equation}
\label{est2}
\sum_{i\in\N_0} \chi(i)^r \leq \Big( \sum_{i\in\N_0} \chi(i)^{rp}\, i^{\tfrac13 p} \Big)^{1/p} 
\Big( \sum_{i\in\N_0} i^{-\tfrac13 q} \Big)^{1/q}
\end{equation}
with $p = (\frac23 + \frac{\gep}{2})^{-1}$ and $q = (\frac13-\frac{\gep}{2})^{-1}$. Since $rp=2$, $\tfrac13 p
<\tfrac12$ and $\tfrac13 q>1$, we can use \eqref{est4chi} to get that both sums in the 
right-hand side are finite.

\medskip\noindent
{\bf 3.} 
Next we prove that $\sum_{i \geq k} \chi(i)^{4-\gep} \leq Ck^{-\tfrac52+\gep}$ for all 
$\gep \in (0,3)$. Indeed, since
\begin{equation}
\chi(i) = \sum_{j\in\N_0} A(i,j)\,\chi(j) 
\leq C \sum_{j\in\N_0} \frac{1}{\sqrt{i+j}}\, \bP(S_{i+j}= i-j)\, \chi(j),
\end{equation}
we can use H\"older's inequality with $p=4-\gep$ and $q=\frac{4-\gep}{3-\gep}> \tfrac43$ 
to estimate
\begin{equation}
\chi(i) \leq \Big( \sum_{j\in\N_0} (i+j)^{-2+\frac{\gep}{2}}\, 
\bP(S_{i+j} = i-j)^{4-\gep} \Big)^{1/(4-\gep)}
\Big( \sum_{j\in\N_0} \chi(j)^{(4-\gep)/(3-\gep)} \Big)^{(3-\gep)/(4-\gep)}.
\end{equation}
Hence, using that the second sum is finite as shown in \eqref{est2}, we get 
\begin{equation}
\label{est3}
\begin{aligned}
&\sum_{i \geq k} \chi(i)^{4-\gep} \leq C \sum_{i \geq k} 
\sum_{j\in\N_0} (i+j)^{-2+\frac{\gep}{2}}\,\bP(S_{i+j}=i-j)^{4-\gep}\\ 
&\leq C \sum_{ {u \geq k} \atop {v\in\Z} } u^{-2+\frac{\gep}{2}}\,\bP(S_u = v)^{4-\gep} 
\leq C \sum_{u \geq k} u^{-\tfrac72+\gep} \leq  C k^{-\tfrac52+\gep},
\end{aligned}
\end{equation}
where in the next-to-last inequality we use that
\begin{equation}
\begin{aligned}
&\bP(S_u = v) \leq \bP(S_u \in \{-1,0,1\}) \leq C\,u^{-\tfrac12} \quad \forall\, u\in\N,\,v\in\Z,\\
&\sum_{v\in\bbZ} \bP(S_u = v)^2 = \bP(S_{2u}=0) \quad \forall\,u \in \N.
\end{aligned}
\end{equation} 

\medskip\noindent
{\bf 4.} Next we prove that $\sum_{i\in\N_0} \sqrt{i}\, \chi(i)^2<\infty$. Let $\gep \in (0,\tfrac12)$. 
By H\"older's inequality with $p=(2-\frac{\gep}{2})/(1-\tfrac{\gep}{2})$ and $q=2-\tfrac{\gep}{2}$,
\begin{equation}
\sum_{i\in\N} \sqrt{i}\, \chi(i)^2 \leq \left(\sum_{i\in\N} 
i^{-(1+\tfrac{\gep}{4})/(1-\tfrac{\gep}{2})}\right)^{(1-\tfrac{\gep}{2})/(2-\tfrac{\gep}{2})} 
\left(\sum_{i\in\N} i^2 \chi(i)^{4-\gep}\right)^{1/(2-\tfrac{\gep}{2})}.
\end{equation}
The first term converges. Abbreviating $\sigma_{4-\gep}(k) = \sum_{i \geq k} \chi(i)^{4-\gep}$,
we find that 
\begin{equation}
\sum_{i\in\N_0} i^2 \chi(i)^{4-\gep} = \sum_{i\in\N} \big[i^2-(i-1)^2\big] \sigma_{4-\gep}(i)
\leq C \sum_{i\in\N} i \sigma_{4-\gep}(i) < C \sum_{i\in\N} i^{-\tfrac32+\gep}< \infty,
\end{equation}
where we use the estimate in \eqref{est3}.

\medskip\noindent
{\bf 5.} We can now give the proof of \eqref{sumfinite}. The change of variables $u = i+j$
and $v = i - j$ yields
\begin{equation}
\begin{aligned}
&\sum_{i,j\in \N_0} \chi(i) \chi(j) \sqrt{i+j+1}\, \bP(S_{i+j} = i-j)\\
&= \sum_{u \in \N_0} \sqrt{u+1} \sum_{v\in\bbZ} 
\chi\Big(\frac{u+v}{2}\Big) \chi\Big(\frac{u-v}{2}\Big) \bP(S_u = v)\\
&= \sum_{u \in \N_0} \sqrt{u+1}\,\, 
\bE\Big[\chi\Big(\frac{u+S_u}{2}\Big) \chi\Big(\frac{u-S_u}{2}\Big)\Big].
\end{aligned}
\end{equation}
By the Cauchy-Schwarz inequality and the symmetry of the simple random walk, the right-hand 
side is bounded from above by
\begin{equation}
\sum_{u \in \N_0} \sqrt{u+1}\,\,\bE\Big[\chi\Big(\frac{u+S_u}{2}\Big)^2\Big] 
= \sum_{u \in \N_0} \sqrt{u+1} \sum_{v\in \bbZ} \chi\Big(\frac{u+v}{2}\Big)^2 \bP(S_u = v).
\end{equation}
Split the sum into two parts: $|v|> \tfrac12 u$ and $|v| \leq \tfrac12 u$. The first part 
can be estimated by (recall that $\|\chi\|_2=1$)
\begin{equation}
\begin{aligned}
&\sum_{u \in \N_0} \sqrt{u+1} \sum_{|v|> \tfrac12 u} 
\chi\Big(\frac{u+v}{2}\Big)^2 \bP(S_u = v)\\ 
&\leq \sum_{u \in \N_0} \sqrt{u+1}\, \bP(|S_u| > \tfrac12 u) 
\leq \sum_{u \in \N_0} \sqrt{u+1}\, e^{-C u} < \infty.
\end{aligned}
\end{equation}
The second part can be estimated by
\begin{equation}
\begin{aligned}
&\sum_{u \in \N_0} \sqrt{u+1} \sum_{|v| \leq \tfrac12 u} 
\chi\Big(\frac{u+v}{2}\Big)^2 \bP(S_u = v)\\ 
&\leq C \sum_{u \in \N_0} \sum_{|v| \leq \tfrac12 u} \sqrt\frac{u+v}{2}\,  
\chi\Big(\frac{u+v}{2}\Big)^2 \bP(S_u = v)\\
&\leq C \sum_{i \in \N_0} \sqrt{i}\, \chi(i)^2 \sum_{j\in\N_0} \bP(S_{i+j} = i-j)\\
&= 2C \sum_{i \in \N_0} \sqrt{i}\, \chi(i)^2,
\end{aligned}
\end{equation}
where we use \eqref{eq:QandP} in the last equality.
\end{proof}

\begin{lemma}
\label{lem:slopecrit2} 
For all $\delta \in (0,\infty)$,
\begin{equation}
\left[ - \frac{\partial \lambda_{\delta,\beta}(0)}{\partial \beta} 
 \right]_{\beta = \beta_c(\delta)} \in (0,\infty).
\end{equation}
\end{lemma}

\begin{proof}
Since we fix $\|\chi\|_2 = 1$, we have
$\langle \chi, \frac{\partial}{\partial \beta}\chi\rangle
= \frac{1}{2} \frac{\partial}{\partial \beta} \|\chi\|_2 = 0$, hence
\begin{equation}
\left[ \frac{\partial \lambda_{\delta,\beta}(0)}{\partial \beta} \right]_{\beta = \beta_c(\delta)} 
= \sum_{i,j \in \N_0} \chi(i) \chi(j) \left[ \frac{\partial}{\partial \beta}G^*_{\delta,\beta}(i+j-1) \right] 
A_{0,\beta_c(\delta),\delta}(i,j).
\end{equation}
Note that, for $\ell \in \N_0$,
\begin{equation}
\label{eq:partial.beta.G}
\frac{\partial}{\partial \beta}G^*_{\delta,\beta}(\ell) 
= - \frac{\bbE[\Omega_\ell^2\,e^{\delta \gO_\ell -\beta\Omega_\ell^2}]}
{\bbE[e^{\delta \gO_\ell -\beta\Omega_\ell^2}]}.
\end{equation}
Use Lemma~\ref{lem:der.beta.G} below to conclude that the sum in the numerator is finite.
\end{proof}

\begin{lemma}
\label{lem:der.beta.G}
$\inf_{\ell\to\infty} \frac{\partial}{\partial \beta}G^*_{\delta,\beta}(\ell) > -\infty$.
\end{lemma}

\begin{proof}
Recall \eqref{eq:partial.beta.G} above. An argument similar to \eqref{eq:app.estG1} gives
\begin{equation}
\bbE[e^{\delta \gO_\ell -\beta\Omega_\ell^2}] \geq C\, \ell^{-1/2},
\end{equation}
and an argument similar to \eqref{eq:app.estG2} gives
\begin{equation}
\bbE[\Omega_\ell^2\,e^{\delta \gO_\ell -\beta\Omega_\ell^2}] \leq C\, 
\Big(\sum_{k\in\N} k^2e^{\gd k - \gb k^2}\Big)\, \ell^{-1/2},
\end{equation}
which ends the proof.
\end{proof}


\subsection{Weak interaction limit}
\label{ss:WIL}

In this section we give the proof of Theorems~\ref{thm:asymp.cc}(1) and \ref{thm:Fscalbeta}.
Recall \eqref{eq:def.Abeta.p.r} and \eqref{eq:def.Gbeta.p}. The proof comes in 10 Steps.

\medskip\noindent
{\bf 1.}
To study limits of variational formulas, we make use of the notion of \emph{epi-convergence} 
(which was used in the derivation of scaling limits for weakly self-avoiding walks as well; see 
van der Hofstad and den Hollander~\cite{vdHdH95}).

\begin{definition}
\label{def:epi}
Let $(Z,\tau)$ be a metrizable topological space, and let $Z' \subset Z$ be dense in $Z$.
Given $U_\beta\colon\,Z \to \R$ with $\beta \in (0,\infty)$ and $U\colon\,Z \to \bar{\R}$, the 
family $(U_\beta)_{\beta \in (0,\infty)}$ is said to be epi-convergent to $U$ on $Z'$, written
\begin{equation}
\lim_{\beta \downarrow 0} U_\beta \epi{=} U \text{ on } Z',
\end{equation}
when the following properties hold:
\begin{equation}
\begin{aligned}
&\forall\,z_\beta \stackrel{\tau}{\to} z \text{ in } Z'\colon 
&\limsup_{\beta \downarrow 0} U_\beta(z_\beta) \leq U(z),\\  
&\exists\,z_\beta \stackrel{\tau}{\to} z \text{ in } Z'\colon 
&\liminf_{\beta \downarrow 0} U_\beta(z_\beta) \geq U(z).  
\end{aligned}
\end{equation}
\end{definition}

\noindent
The importance of the notion of epi-convergence is contained in the following proposition, for 
which we refer to Attouch~\cite[Theorem 1.10 and Proposition 1.14]{At84}.

\begin{proposition}
\label{prop:epi}
Suppose that
\begin{itemize}
\item[(I)] 
$\lim_{\beta \downarrow 0} U_\beta \epi{=} U$ on $Z'$.
\item[(II)] 
For all $\beta \in (0,\infty)$, $U_\beta$ is continuous on $Z$ and has a unique
maximiser $\bar{z}_\beta \in Z$.
\item[(III)]
There exists a $K \subset Z'$ such that $K$ is $\tau$-relatively compact in $Z$, $U$ has a 
unique maximizer $\bar{z} \in \bar{K}$, and there exists a sequence $(z_\beta)_{\beta \in 
(0,\infty)}$ in $\bar{K}$ such that $z_\beta-\bar{z}_\beta \stackrel{\tau}{\to} 0$ and $U_\beta(z_\beta) 
- U_\beta(\bar{z}_\beta) \to 0$ as $\beta \downarrow 0$.
\end{itemize}
Then, as $\beta \downarrow 0$,
\begin{equation}
\sup_{z \in Z} U_\beta(z) \to \sup_{z \in Z} U(z),
\qquad \bar{z}_\beta \stackrel{\tau}{\to} \bar{z}.
\end{equation}
\end{proposition}

\noindent
Below we will apply Proposition~\ref{prop:epi} with the following choices:
\begin{equation}
\label{epichoices}
\begin{aligned}
Z &= \big\{f \in L^2((0,\infty))\colon\,f \geq 0,\,\|f\|_2 = 1\big\},\\
Z' &= \cC,\\
\tau &= \text{ topology induced by the $L^2$-norm},\\
K &= K_c = \big\{f \in \cC\colon\,U(f) \geq -c\big\},\\
U_\beta &= \gb^{-\eta}\,I_{\gd,\gb},\\
U &= a U_1 - b U_2 - U_3,
\end{aligned}
\end{equation} 
where $\cC$ is the set defined in \eqref{Csetdef}, $U_1, U_2, U_3$ are the functions 
defined in \eqref{eq:defA123} below, $I_{\gd,\gb}$ is the functional defined in \eqref{var1} 
below, $\eta = \tfrac23$ (or $\eta=\tfrac13$) for Theorem~\ref{thm:asymp.cc}(1) (or 
Theorem~\ref{thm:Fscalbeta}), while $c$ is a constant chosen large enough so that 
$K_c \neq \emptyset$. The constants $a$ and $b$ are chosen as we go along.

\medskip\noindent
{\bf 2.} 
From the Rayleigh formula, we have
\begin{equation}
\label{lambdavardiscrete}
\gl_{\gd,\gb}(\mu) - 1 = \suptwo{v\in \ell_2(\N_0)}{v\geq0,\, \|v\|_2 =1} 
\Big\{ \sum_{i,j\in\N_0} A_{\mu,\delta,\beta}(i,j) v(i) v(j) - \sum_{i\in \N_0} v(i)^2 \Big\},
\end{equation}
We begin by showing that the supremum can actually be taken over functions in $L^2((0,\infty))$. 
Indeed, for $\eta>0$ write
\begin{equation}
R(h) = \frac{1}{\gb^\eta} \int_0^\infty \dd x \int_0^\infty \dd y\, h(x) h(y) \,
A_{\mu,\delta,\beta}\Big(\Big\lfloor \frac{x}{\gb^\eta} \Big\rfloor, 
\Big\lfloor \frac{y}{\gb^\eta} \Big\rfloor\Big), \quad h\in L^2((0,\infty)). 
\end{equation}
Let $f\in L^2((0,\infty))$ with $\| f\|_2=1$. Define the piecewise constant function
\begin{equation}
g(x) = \sum_{i\in \N} c_i \ind_{\{\gb^\eta (i-1) < x \leq \gb^\eta i\}}, \quad x \in (0,\infty),
\qquad c_i = \frac{1}{\gb^\eta} \int_{\gb^\eta (i-1)}^{\gb^\eta i} f(x)\,\dd x, \quad i\in\N.
\end{equation}
Then $\|g\|_2^2 = \int_0^\infty g(x)^2 \dd x = \gb^\eta \sum_{i\in \N} c_i^2$ which, by Jensen's 
inequality, is smaller than or equal to $\int_{0}^\infty f(x)^2 \dd x =1$. Next, denote the renormalised 
version of $g$ by $\hat g = g/\|g\|_2$. Then
\begin{equation}
R(\hat g) - 1  = \frac{1}{\| g\|^2_2} R(g) - 1 = \frac{1}{\| g\|^2_2} R(f) - 1 \geq R(f) - 1.
\end{equation}
Therefore we may write
\begin{equation}
\label{var1}
\gl_{\gd,\gb}(\mu) - 1 = \suptwo{f\in L^2((0,\infty))}{f\geq 0,\, \| f\|_2=1} I_{\gd,\gb}(f)
\end{equation}
with
\begin{equation} 
I_{\gd,\gb}(f) = \frac{1}{\gb^\eta} \int_0^\infty \dd x \int_0^\infty \dd y\, f(x) f(y)\, 
A_{\mu,\delta,\beta}\Big(\left\lfloor\frac{x}{\gb^\eta}\right\rfloor,
\left\lfloor\frac{y}{\gb^\eta}\right\rfloor\Big) - \int_0^\infty \dd x f(x)^2.
\end{equation}

\medskip\noindent
{\bf 3.} 
Next, recall \eqref{eq:def.Qij}. Using that, for all $x$,
\begin{equation}
\frac{1}{\gb^\eta}\int_0^\infty Q\Big(\Big\lfloor \frac{x}{\gb^\eta} \Big\rfloor 
+ 1, \Big\lfloor \frac{y}{\gb^\eta} \Big\rfloor\Big)\, \dd y = 1,
\end{equation}
we decompose the variational formula in \eqref{var1} as
\begin{equation}
\label{eq:var.for}
I_{\gd,\gb}(f) = I^1_{\gd,\gb}(f) - I^2_{\gd,\gb}(f),
\end{equation}
where
\begin{equation}
\label{I1I2def}
\begin{aligned}
I^1_{\gd,\gb}(f) &= \int_0^\infty \dd x \int_0^\infty \dd y\, f(x)^2  
\Big[ \frac{1}{\gb^\eta} A_{\mu,\delta,\beta}\left(\left\lfloor\frac{x}{\gb^\eta}\right\rfloor,
\left\lfloor\frac{y}{\gb^\eta}\right\rfloor\right) 
- \frac{1}{\gb^\eta} Q\Big(\left\lfloor\frac{x}{\gb^\eta}\right\rfloor+1,
\left\lfloor\frac{y}{\gb^\eta}\right\rfloor\Big)\Big],\\
I^2_{\gd,\gb}(f) &= \frac12 \int_0^\infty \dd x \int_0^\infty \dd y\, [f(x) - f(y)]^2 \frac{1}{\gb^\eta} 
A_{\mu,\delta,\beta}\Big(\left\lfloor\frac{x}{\gb^\eta}\right\rfloor,
\left\lfloor\frac{y}{\gb^\eta}\right\lfloor\Big).
\end{aligned}
\end{equation}
We are interested in the behaviour of the quantity in \eqref{eq:var.for} as $\beta \downarrow 0$. 
Define
\begin{equation}
\label{eq:defA123}
U_1(f) = \int_0^\infty \dd x\, (2x) f(x)^2, \quad U_2(f) = \int_0^\infty \dd x\, (2x)^2 f(x)^2, 
\quad U_3(f) = \int_0^\infty \dd x\, x[f'(x)]^2,
\end{equation}
with $U_3(f)=\infty$ when $f'$ does not exist everywhere. Our key observation is the following lemma. 
Parts (1) and (2) settle requirement in (I) in Proposition~\ref{prop:epi}, Part (3) settles requirement (III), 
while requirement (II) follows from the fact that \eqref{lambdavardiscrete} has a unique maximiser 
and hence so does \eqref{var1}.  

\begin{lemma}
\label{lem:epiconvergence}
{\rm (1)} 
Pick $B,C\in \bbR$ and put $\mu = B\gb^{4/3}$, $\frac12 \gd^2 - \gb = C(\tfrac12 \gd^2)^{4/3}$, 
$\eta = \tfrac23$. Then
\begin{equation}
\label{I1I2scal(1)}
\begin{aligned}
\lim_{\gb \downarrow 0} \frac{1}{\gb^{2/3}}\, I^1_{\gd,\gb}(f) 
&\epi{=} (C-B)\, U_1(f) - \, U_2(f),\\
\lim_{\gb \downarrow 0} \left[-\frac{1}{\gb^{2/3}}\, I^2_{\gd,\gb}(f)\right] 
&\epi{=} -U_3(f).
\end{aligned}
\end{equation}
{\rm (2)} 
Pick $B\in\bbR$ and put $\mu = -f(\gd) - B \gb^{2/3}$, $\eta = \tfrac13$. Then
\begin{equation}
\label{I1I2scal(2)}
\begin{aligned}
\lim_{\gb \downarrow 0} \frac{1}{\gb^{1/3}}\, I^1_{\gd,\gb}(f) 
&\epi{=} B\, U_1(f) - \rho_\gd U_2(f), \\
\lim_{\gb \downarrow 0} \left[-\frac{1}{\gb^{1/3}}\, I^2_{\gd,\gb}(f)\right] 
&\epi{=} - U_3(f),
\end{aligned}
\end{equation}
where $\rho_\gd = \bbE^\gd[\go_1]$.\\
{\rm (3)} 
Fix $\mu,\gd,\gb\in(0,\infty)$ and let $K_c = \{f\in\cC\colon\, U(f) \geq -c\}$. 
Let $\bar{f}_\beta \in Z$ be the unique maximizer of $U_\gb$ in $Z$ defined
in \eqref{epichoices}. Then there exist $f_\gb \in \overline{K_c}$, $\beta \in (0,\infty)$, 
such that 
\begin{equation}
\lim_{\gb \downarrow 0} \| f_\gb - \bar{f}_\gb \|_2 =0,\qquad 
\lim_{\gb \downarrow 0} |U_\gb(f_\gb) - U_\gb(\bar{f}_\gb)| = 0.
\end{equation}
The same holds for $\gd = \gd(\gb)$ satisfying $\frac12\gd^2 - \gb = C(\frac12 \gd^2)^{4/3}$, $C\in (0,\infty)$.

\end{lemma}

\begin{proof}[Sketch of the proof]
We give a brief sketch, the details of which will be worked out in Steps 7--9 below. 
Using Proposition~\ref{pr:G.conv-to-Brownian} in Appendix~\ref{AppA} we get,
for $\beta \downarrow 0$,
\begin{equation}
\label{Ascalprop}
\begin{aligned}
&\frac{1}{\gb^\eta} A_{\mu,\delta,\beta}\Big(\left\lfloor\frac{x}{\gb^\eta}\right\rfloor,
\left\lfloor\frac{y}{\gb^\eta}\right\rfloor\Big) 
- \frac{1}{\gb^\eta} Q\Big(\left\lfloor\frac{x}{\gb^\eta}\right\rfloor +1,
\left\lfloor\frac{y}{\gb^\eta}\right\rfloor\Big)\\ 
&\qquad = \Big[e^{G^*_{\delta,\beta}(\frac{x+y}{\gb^\eta})-\mu (\frac{x+y}{\gb^\eta})} - 1 \Big] 
\frac{1}{\gb^\eta} Q\Big(\left\lfloor\frac{x}{\gb^\eta}\right\rfloor +1,
\left\lfloor\frac{y}{\gb^\eta}\right\rfloor\Big)\\
&\qquad \epi{\sim} \Big[(\tfrac12\gd^2 - \gb - \mu) \gb^{-\eta}(x+y) 
- \gb^{1-2\eta}\gd^2 (x+y)^2 \Big] \frac{1}{\gb^\eta} 
Q\Big(\left\lfloor\frac{x}{\gb^\eta}\right\rfloor +1,
\left\lfloor\frac{y}{\gb^\eta}\right\rfloor\Big).
\end{aligned}
\end{equation}
Inserting \eqref{Ascalprop} into the first line of \eqref{I1I2def} and using Proposition~\ref{lem:asymptP} 
in Appendix~\ref{AppA}, we find the first lines of \eqref{I1I2scal(1)}--\eqref{I1I2scal(2)}. Inserting 
\eqref{Ascalprop} into the second line of \eqref{I1I2def} and using Proposition~\ref{lem:asymptP} 
in Appendix~\ref{AppA}, we find the second lines of \eqref{I1I2scal(1)}--\eqref{I1I2scal(2)}. Part (3)
will be achieved by taking for $f_\beta$ the linear interpolation of $\bar{f}_\beta$. 
\end{proof}

\medskip\noindent
{\bf 4.\ Proof of Theorem \ref{thm:asymp.cc}(1).} 
We look at the scaling of $\gb_c(\gd) - \tfrac12\gd^2$ as $\gd \downarrow 0$. Put
\begin{equation}
\label{eq:setvar}
\eta = \tfrac23, \qquad \mu = B\gb^{\tfrac43}, \qquad \beta 
= \tfrac12\gd^2 - C (\tfrac12\gd^2)^{\tfrac43}.
\end{equation}
Then \eqref{eq:var.for}, Proposition~\ref{prop:epi} and Lemma~\ref{lem:epiconvergence}(1,3) 
imply that
\begin{equation}
\label{lambdarholim}
\lim_{\gd \downarrow  0} \gb^{-\tfrac23}\,[\gl_{\gd,\gb}(\mu)-1] = \chi(C-B,1),
\end{equation}
where
\begin{equation}
\label{varforSL1}
\chi(a,1) = \sup_{f\in\cC} 
\left\{ \int_0^\infty \dd x\, \big[ f(x)^2 [a(2x) - (2x)^2] - \tfrac12 f'(x)^2 (2x) \big] \right\},
\qquad a\in\bbR.
\end{equation}
with $\cC$ the set defined in \eqref{Csetdef}. Via integration by parts, the variational 
formula in \eqref{varforSL1} can be rewritten as
\begin{equation}
\label{varforSL2}
\chi(a,1) = \sup_{f\in\cC} 
\langle f, \cL^{a,1} f \rangle,\qquad (\cL^{a,1}f)(x) = xf''(x)  + f'(x) + [a(2x) - (2x)^2] f(x),
\end{equation}
which is the variational representation of the largest eigenvalue $\chi(a,1)$ of the 
Sturm-Liouville operator $\cL^{a,1}$ introduced in \eqref{SL}--\eqref{SLeig}. Pick
$B=0$ in \eqref{lambdarholim} and use that $C \mapsto \chi(C,1)$ changes sign 
at $C=a^*$ according to \eqref{SLprop}, to obtain that \eqref{lambdarholim} yields 
the scaling for the critical curve given in \eqref{eq:betacasympzero}.

\medskip\noindent
{\bf 5.\ Proof of Theorem~\ref{thm:Fscalbeta}(2).} 
This a consequence of \eqref{lambdarholim} and the shape of $b\mapsto \chi(a,b)$ (recall 
Fig.~\ref{fig-chiab}). Indeed, let $\beta$ be as in \eqref{eq:setvar}. Then
\begin{equation}
F^*\Big(\gd,\tfrac12\gd^2-C(\tfrac12\gd^2)^{\tfrac43}\Big) 
\sim [C - a^*(1)](\tfrac12\gd^2)^{\tfrac43}, \qquad \gd,\gb \downarrow 0.
\end{equation}
The right-hand side is equivalent to $\gb_c(\gd) - \gb$ as soon as $\beta_c(\gd) - \beta \asymp 
\delta^{8/3}$.

\medskip\noindent
{\bf 6.\ Proof of Theorem \ref{thm:Fscalbeta}(1).}
Set $\mu = -f(\gd) - B \gb^{\tfrac23}$. Then \eqref{eq:var.for}, Proposition~\ref{prop:epi} 
and Lemma~\ref{lem:epiconvergence}(2,3) imply that
\begin{equation}
\lim_{\gd \downarrow 0} \gb^{-\tfrac13}\,[\gl_{\gd,\gb}(\mu)-1] = \chi(B,\rho_\gd),
\end{equation}
from which we get $\mu(\gd,\gb) = -f(\gd) - a^*(\rho_\gd) \gb^{\tfrac23} [1+o(1)]$. 
Because $F(\gd,\gb)=\mu(\gd,\gb)+f(\gd)$, this proves the asymptotics for the free
energy in the first part of \eqref{Fvrhoscal}. As to the asymptotics for the speed, 
recall \eqref{eq:speedspec}, which reads
\begin{equation}
\label{varvalt}
v(\delta,\beta) 
= \left[-\frac{\partial}{\partial\mu} 
\lambda_{\delta,\beta}(\mu)\right]^{-1}_{\mu = \mu(\delta,\beta)}
\end{equation}
because $\gl_{\gd,\gb}(\mu(\gd,\gb))=1$. We have just proven that
\begin{equation}
\label{lambdascalingformula}
\gl_{\gd,\gb}\Big(-f(\gd) - B\gb^{\tfrac23}\Big) = 1 + \chi(B,\rho_\gd) \gb^{\tfrac13} [1+o(1)],
\qquad \beta \downarrow 0.
\end{equation}
By convexity, we may take the derivative of \eqref{lambdascalingformula} with respect to 
$B$, to get
\begin{equation}
-\gb^{\tfrac23}\left[\frac{\partial}{\partial\mu} 
\lambda_{\delta,\beta}(\mu)\right]_{\mu = -f(\gd) - B\gb^{\tfrac23}} 
= \left[\frac{\partial}{\partial B}\chi(B,\rho_\gd)\right] \gb^{\tfrac13} [1+o(1)].
\end{equation}
Using monotonicity of $\mu \mapsto \frac{\partial}{\partial\mu} \lambda_{\delta,\beta}(\mu)$ 
and continuity of $B \mapsto \frac{\partial}{\partial B}\chi(B,\rho_\gd)$, we in turn deduce that
\begin{equation}
\begin{aligned}
\left[\frac{\partial}{\partial\mu} \lambda_{\delta,\beta}(\mu)\right]_{\mu = \mu(\gd,\gb)} 
&= -\left[\frac{\partial}{\partial B}\chi(B,\rho_\gd) \right]_{B = a^*(\rho_\gd)} \gb^{-\tfrac13} [1+o(1)],
\end{aligned}
\end{equation}
which via \eqref{varvalt} proves the second part of \eqref{Fvrhoscal}. To get the asymptotics 
for the charge we note that
\begin{equation}
\rho(\gd,\gb)-\rho_\delta = \frac{\partial}{\partial\gd} [\mu(\gd,\gb)+f(\delta)]
= \frac{\partial}{\partial\gd} F(\delta,\beta). 
\end{equation}
We know that $F(\delta,\beta) \sim -a^*(\rho_\delta)\beta^{2/3}$ as $\beta \downarrow 0$. 
Hence we get the third part of \eqref{Fvrhoscal}, provided we show that the differentiation 
w.r.t.\ $\delta$ and the limit $\beta \downarrow 0$ may be interchanged. This can be justified 
as follows. Fix $\delta \in (0,\infty)$. For $\gd_1,\gd_2 \in (0,\infty)$, estimate
\begin{equation}
\begin{aligned}
\Big|\gb^{-\tfrac23} \frac{\partial}{\partial\gd}F(\gd,\gb) 
- \frac{\dd}{\dd\gd}[-a^*(\rho_\gd)]\Big| 
&\leq \gb^{-\tfrac23} \Big|\frac{\partial}{\partial\gd}F(\gd,\gb)
- \frac{F(\gd_1,\gb)- F(\gd_2,\gb)}{\gd_1 - \gd_2}\Big|\\
&\qquad + \Big| \frac{[\gb^{-\tfrac23}F(\gd_1,\gb) + a^*(\rho_{\gd_1})] 
- [\gb^{-\tfrac23} F(\gd_2,\gb) + a^*(\rho_{\gd_2})]}{\gd_1 - \gd_2}\Big|\\
&\qquad + \Big|\frac{a^*(\rho_{\gd_2})-a^*(\rho_{\gd_1})}{\gd_1 - \gd_2} 
- \frac{\dd}{\dd\gd}[-a^*(\rho_\gd)]\Big|.
\end{aligned}
\end{equation}
The second term in the right-hand side tends to zero as $\beta \downarrow 0$ for every choice of 
$\delta_1,\delta_2$. The third term does not depend on $\beta$, and tends to zero as $\delta_1, 
\delta_2 \to \delta$. To control the first term it is enough to prove that for some $\gb_0 \in 
(0,\infty)$,
\begin{equation}
\lim_{\gd_1,\gd_2 \to \gd} \sup_{\gb \in (0,\gb_0)} 
\gb^{-\tfrac23} \Big|\frac{\partial}{\partial\gd}F(\gd,\gb) 
- \frac{F(\gd_1,\gb) - F(\gd_2,\gb)}{\gd_1-\gd_2} \Big| = 0.
\end{equation}
A sufficient condition for the latter is that there exist $\gb_0,\gep_0 \in (0,\infty)$ such that
\begin{equation}
\sup_{\gb \in (0,\gb_0)} \gb^{-\tfrac23} \sup_{|\gd' - \gd|\leq \gep_0} 
\frac{\partial^2}{\partial\gd^2}F(\gd',\gb) < \infty.
\end{equation}
Finally, from \eqref{eq:formula.variance}, $\frac{\partial^2}{\partial\gd^2}F(\gd,\gb) 
= \frac{\partial^2}{\partial\gd^2}\mu(\gd,\gb) + \frac{\dd^2}{\dd\gd^2}f(\gd) 
= \sigma_\rho^2(\gd,\gb) - \sigma_\rho^2(\gd,0)$.

\medskip\noindent
{\bf 7.}
In the remaining steps we prove Lemma~\ref{lem:epiconvergence}. Along the way we need two 
technical lemmas.

\begin{lemma}
\label{lem:lim_moment}
As $\ell\to\infty$,
\begin{equation}
\begin{array}{lll}
&\forall m\in\N_0\colon &\bbE[\gO_\ell^{2m}] 
= c_m\ \ell^m [1+o(1)] \mbox{ with } c_m= 1 \times 3 \times \ldots \times (2m-1),\\[0.2cm]
&\forall m\in\N\colon &\bbE[\gO_\ell^{2m+1}] \leq 
c'_m\ \ell^m \mbox{ for some constant $c'_m$ depending on $m$.}
\end{array}
\end{equation}
\end{lemma}

\begin{proof}
The computations are straightforward and are left to the reader.
\end{proof}

\begin{lemma}
\label{lem:momentsQij}
For every $\ga\geq 1$ there exists a constant $c_\ga$ such that
\begin{equation}
\sum_{j\in \bbN_0} (i+j)^\ga Q(i,j) \leq c_\ga\, i^\ga,\qquad i\in\bbN_0.
\end{equation}
\end{lemma}

\begin{proof}
Note that $(i+j)^\ga \leq 2^{1-\ga}(i^\ga + j^\ga)$, and use \eqref{eq:Q} to show that 
$\sum_{j\in\N_0} j^\ga Q(i,j) \leq C i^\ga$ for some constant $C$.
\end{proof}

\medskip\noindent
{\bf 8.}
Lemmas~\ref{lem:epi.one.sup}--\ref{lem:epi.two.inf} below prove the epi-convergence 
claimed in Lemma~\ref{lem:epiconvergence}(1,2), which is requirement (I) in 
Proposition~\ref{prop:epi}.

\begin{lemma}
\label{lem:epi.one.sup}
Let $\beta = \tfrac12\gd^2 - C(\tfrac12\gd^2)^{4/3}$ and $\mu = B\gb^{4/3}$. 
Then, for all $f_\beta\stackrel{L^2}{\to} f$ in $\cC$, 
\begin{equation}
\limsup_{\gb \downarrow 0} \gb^{-\tfrac23} I^1_{\gd,\gb}(f_\beta) \leq (C-B) U_1(f) - U_2(f).
\end{equation}
\end{lemma}

\begin{proof}
For simplicity, we start with the case $B=0$. For $\ell \in \N$, let $e_{\gd,\gb}(\ell) 
= \gd \gO_\ell - \gb \gO_\ell^2$. Note that
\begin{equation}
\sup_{\ell\in \N} e_{\gd,\gb}(\ell) \leq \tfrac14\frac{\gd^2}{\gb} = \tfrac12[1+o(1)],
\qquad \gb \downarrow 0.
\end{equation}
We expand $e^{e_{\gd,\gb}(\ell)}-1$ to fifth order in $\gb$, so that the expansion 
includes the limiting term. There exists a constant $c$ such that
\begin{equation}
\label{eq:egdgb_sup}
e^{e_{\gd,\gb}(\ell)}-1 \leq e_{\gd,\gb}(\ell) +\tfrac12 e_{\gd,\gb}(\ell)^2 
+ \tfrac16 e_{\gd,\gb}(\ell)^3 + \tfrac{1}{24} e_{\gd,\gb}(\ell)^4 
+ c e_{\gd,\gb}(\ell)^5 \ind_{\{e_{\gd,\gb}(\ell)\geq 0\}}.
\end{equation}
Keep in mind that $\ell$ will later be replaced by $i+j$, where $0\leq i,j \leq N \gb^{-\frac23}$. 
Reordering the terms of the expansion according to $\gb$, we get
\begin{equation}
\label{eq:egdgb_sup1}
\bbE[e^{e_{\gd,\gb}(\ell)}-1] \leq \big( \tfrac12 \gd^2 - \gb \big) 
\bbE[\gO_\ell^2] + \big( \tfrac12 \gb^2 - \tfrac12 \gd^2 \gb 
+ \tfrac{1}{24} \gd^4 \big) \bbE[\gO_\ell^4] + R_{\gd,\gb}(\ell),
\end{equation}
where $R_{\gd,\gb}(\ell)$ is a remainder term given by
\begin{equation}
\begin{split}
R_{\gd,\gb}(\ell) = \big(\tfrac16 \gd^3 - \gd \gb\big)\bbE[\gO_\ell^3] 
+ \big(\tfrac12 \gd \gb^2 - \tfrac16 \gd^3 \gb \big)\bbE[\gO_\ell^5] 
+ \tfrac14 \gd^2 \gb^2 \bbE[\gO_\ell^6] - \tfrac16 \gd\gb^3 \bbE[\gO_\ell^7]\\ 
+ \tfrac{1}{24} \gb^4 \bbE[\gO_\ell^8] + \gd^5 \bbE[\gO_\ell^5] 
+ 10 \gd^3 \gb^2 \bbE[\gO_\ell^7] + 5 \gd \gb^4 \bbE[\gO_\ell^9].
\end{split}
\end{equation}
Recall the first line of \eqref{I1I2def}, and set $\eta = \tfrac23$. Estimate 
$I^1_{\gd,\gb}(f_\gb) \leq I^{1,1}_{\gd,\gb}(f_\gb) + I^{1,2}_{\gd,\gb}(f_\gb)$, 
where
\begin{equation}
\begin{split}
&I^{1,1}_{\gd,\gb}(f_\gb) 
= \gb^{-\frac23} \int_0^\infty \dd x\int_0^\infty \dd y\, f_\gb^2(x)\,   
Q\left(\left\lfloor x\gb^{-2/3}\right\rfloor +1, \left\lfloor y\gb^{-2/3}\right\rfloor\right)\\
&\times \left\{\big(\tfrac12\gd^2-\gb\big) \bbE\left[\gO_{\left\lfloor x\gb^{-2/3}\right\rfloor
+ \left\lfloor y\gb^{-2/3}\right\rfloor}^2\right] 
+ \big(\tfrac12 \gb^2-\tfrac12 \gd^2 \gb + \tfrac{1}{24}\gd^4\big) 
\bbE\left[\gO_{\left\lfloor x\gb^{-2/3}\right\rfloor 
+ \left\lfloor y\gb^{-2/3}\right\rfloor }^4 \right] \right\}
\end{split}
\end{equation}
and
\begin{equation}
\begin{aligned}
I^{1,2}_{\gd,\gb}(f_\gb) &= \gb^{-\frac23} \int_0^\infty \dd x\int_0^\infty \dd y\, f_\gb^2(x)\,   
Q\left(\left\lfloor x\gb^{-2/3}\right\rfloor + 1, \left\lfloor y\gb^{-2/3}\right\rfloor\right)\\ 
&\qquad \times R_{\gd,\gb}\left(\left\lfloor x\gb^{-2/3}\right\rfloor + \left\lfloor y\gb^{-2/3}\right\rfloor\right).
\end{aligned}
\end{equation}
Let us first deal with $I^{1,1}_{\gd,\gb}(f_\gb)$. We cut the integrals over $x$ and $y$ at $\gep>0$. 
Using Lemmas~\ref{lem:lim_moment}--\ref{lem:momentsQij} we get that, for all $\gep>0$ small 
enough, there exists a constant $c$ such that
\begin{equation}
\begin{aligned}
I^{1,1}_{\gd,\gb}(f_\gb) 
&\leq c\gep \gb^{\tfrac23} + [1+o(1)]\,\gb^{-\tfrac23}
\int_\gep^\infty \dd x \int_\gep^\infty \dd y\, f_\gb^2(x)\,   
Q\left(\left\lfloor x\gb^{-2/3}\right\rfloor + 1, \left\lfloor y\gb^{-2/3}\right\rfloor\right)\\ 
&\qquad \qquad \qquad \times [C(x+y) - (x+y)^2]\,\gb^{\frac23}.
\end{aligned}
\end{equation}
Therefore,
\begin{equation}
\begin{aligned}
\limsup_{\gb\downarrow 0} \gb^{-\tfrac23}\,I^{1,1}_{\gd,\gb}(f_\gb) 
&\leq 
\limsup_{\gb\downarrow 0} \beta^{-\tfrac23} \int_0^\infty \dd x\int_0^\infty \dd y f_\gb^2(x)\,   
Q\left(\left\lfloor x\gb^{-2/3}\right\rfloor+1, \left\lfloor y\gb^{-2/3}\right\rfloor\right)\\
&\qquad \qquad \qquad \times  [C(x+y) - (x+y)^2].
\end{aligned}
\end{equation}
Recall \eqref{eq:QandP}. Making the change of variables $u=x+y$ and $v=x-y$, we obtain 
\begin{equation}
\begin{aligned}
\label{eq:cvgI_oneone}
&\beta^{-\tfrac23}\int_0^\infty \dd x\int_0^\infty \dd y\, f_\gb^2(x)\,   
Q\left(\left\lfloor x\gb^{-2/3}\right\rfloor +1, \left\lfloor y\gb^{-2/3}\right\rfloor\right) 
[C(x+y) - (x+y)^2]\\
& \leq \tfrac12 \gb^{-\tfrac23} \int_0^\infty \dd u \int_{-\infty}^\infty \dd v\, 
f_\gb^2\,\Big(\frac{u+v}{2}\Big)\,(Cu - u^2)\,\bP\Big(S_{\lfloor u\gb^{-2/3}\rfloor} 
= \lfloor v\gb^{-\tfrac23}\rfloor\Big)\\
& = \int_0^\infty \dd \left(\frac{u}{2}\right)\, (Cu - u^2)\, 
\bE\Big[ f^2_\gb\Big( \frac{u + \gb^{2/3}S_{\lfloor u\gb^{-2/3}\rfloor}}{2} \Big) \Big] \\
& \stackrel{\gb \downarrow 0}{\longrightarrow} \int_0^\infty \dd \left(\frac{u}{2}\right)
\,(Cu - u^2)\,f\left(\frac{u}{2}\right)^2.
\end{aligned}
\end{equation}
We next deal with the remainder term $I^{1,2}_{\gd,\gb}(f_\gb)$ and show it is $o(\gb^{2/3})$. For 
brevity we deal with the first term of $R_{\gd,\gb}(\ell)$ only, and leave the reader to check that the other terms 
in $R_{\gd,\gb}(\ell)$ can be handled in the same way. Using Lemmas~\ref{lem:lim_moment}--\ref{lem:momentsQij}, 
we get
\begin{equation}
\begin{aligned}
&\int_0^\infty \dd x \int_0^\infty \dd y\, \gb^{-2/3}\,f_\gb^2(x)\, 
Q\left(\left\lfloor x\gb^{-2/3}\right\rfloor+1,\left\lfloor y\gb^{-2/3}\right\rfloor\right)\,
\big(\tfrac16 \gd^3 - \gd\gb \big)\,\bbE\left[\gO^3_{\left\lfloor (x+y)\gb^{-2/3}\right\rfloor}\right]\\
&\leq c \gb^{3/2} \int_0^\infty \dd x\, f_\gb^2(x) \int_0^\infty \dd (\gb^{-2/3}y)\, 
Q\left(\left\lfloor x \gb^{-2/3}\right\rfloor +1,\left\lfloor y\gb^{-2/3}\right\rfloor\right)\,
\big[(x+y)\gb^{-2/3}\big]\\
&\leq c \gb^{3/2} \int_0^\infty \dd x\, f_\gb^2(x)\,\big[x\gb^{-2/3}\big] 
\stackrel{\gb \downarrow 0}{\sim} c \gb^{5/6} \int_0^\infty \dd x\, xf^2(x). 
\end{aligned}
\end{equation}
We next indicate how to deal with the case $B>0$. The left-hand side of \eqref{eq:egdgb_sup} 
has to be replaced by $e^{e_{\gd,\gb}(\ell)-\mu\ell}-1$. Since $\mu \ell$ is of order $\gb^{2/3}$, 
the first term in the right-hand side of \eqref{eq:egdgb_sup1} becomes $(\tfrac12 \gd^2 - \gb - \mu) 
\bE(\gO_\ell^2)$ (recall that $\bE(\gO_\ell^2)=\ell$). Moreover, $\tfrac12 \gd^2 - \gb - \mu$ is equivalent to $(C-B)\gb^{4/3}$, so we may repeat the computations for the case $B=0$ after replacing $C$ by $C-B$.
\end{proof}

\begin{lemma}
\label{lem:epi.one.inf}
Let $\beta = \tfrac12\gd^2 - C(\tfrac12\gd^2)^{4/3}$ and $\mu = B\gb^{4/3}$. Then, 
for all $f\in \cC$, 
\begin{equation}
\liminf_{\gb \downarrow 0} \gb^{-2/3} I^1_{\gd,\gb}(f) \geq (C-B) U_1(f) - U_2(f).
\end{equation}
\end{lemma}

\begin{proof}
Instead of \eqref{eq:egdgb_sup}, use
\begin{equation}
\label{eq:egdgb_inf}
e^{e_{\gd,\gb}(\ell)}-1 \geq e_{\gd,\gb}(\ell) + \tfrac12 e_{\gd,\gb}(\ell)^2 
+ \tfrac16 e_{\gd,\gb}(\ell)^3 + \tfrac{1}{24} e_{\gd,\gb}(\ell)^4 
+ \tfrac{1}{120} e_{\gd,\gb}(\ell)^5.
\end{equation}
The analysis for small $\beta$ in Lemma~\ref{lem:epi.one.sup} essentially carries 
over.
\end{proof}

\begin{lemma}
\label{lem:epi.two.sup}
Let $\beta = \tfrac12\gd^2 - C(\tfrac12 \gd^2)^{4/3}$ and $\mu = B\gb^{4/3}$. Then,
for all $f_\beta\stackrel{L^2}{\to} f\in \cC$, 
\begin{equation}
\limsup_{\gb \downarrow 0} \Big\{ -\gb^{-2/3} I^2_{\gd,\gb}(f_\gb) \Big\} \leq - U_3(f).
\end{equation}
\end{lemma}

\begin{proof}
We need a lower bound for
\begin{equation}
\begin{aligned}
I^2_{\gd,\gb}(f_\gb) &= \tfrac12 \int_0^\infty \dd x \int_0^\infty \dd y\, 
[f_\gb(x)-f_\gb(y)]^2\\
&\qquad\qquad \times \gb^{-2/3}\,
e^{G^*_{\gd,\gb}(\lfloor(x+y)\gb^{-2/3}\rfloor)-\mu(\lfloor(x+y)\gb^{-2/3}\rfloor)}\, 
Q\left(\left\lfloor x\gb^{-2/3}\right\rfloor+1,\left\lfloor y\gb^{-2/3}\right\rfloor\right).
\end{aligned}
\end{equation}
Fix $M>1$ and observe that
\begin{equation}
e^{G^*_{\gd,\gb}(\ell)} \geq e^{-\gb M^2 \ell - \gd M \sqrt{\ell}}\, 
\bbP\left(\left|\frac{\gO_\ell}{\sqrt{\ell}}\right| \leq M \right).
\end{equation}
Fix $0<\gep<N$. Then, restricting the integral over $x$ to the interval $(\gep, N)$ and the 
integral over $y$ to $(x\pm \gb^\frac13 N)$, we obtain for $\gep$ and $\gb$ small enough
\begin{equation}
\begin{aligned}
&\gb^{-2/3} I^2_{\gd,\gb}(f_\gb) \geq e^{-3\gb^{1/3}M^2 N - 3\gb^{1/6}M\sqrt{N} - 3B\gb^{2/3}N} 
\, \inf_{\ell \geq \frac{\gep}{2} \gb^{-2/3}}\, \bbP\left( \left|\frac{\gO_\ell}{\sqrt{\ell}}\right| \leq M \right)\\
&\qquad \times \tfrac12 \int_\gep^N \dd x \int_{x-\gb^{1/3}N}^{x+\gb^{1/3}N} \dd y\, 
\gb^{-4/3} [f_\gb(x)-f_\gb(y)]^2\, Q\left(\left\lfloor x\gb^{-2/3}\right\rfloor +1, 
\left\lfloor y\gb^{-2/3}\right\rfloor\right).
\end{aligned}
\end{equation}
For fixed $\gep,N,M$, the exponential term goes to $1$ and the infimum tends to 
$\bbP(|\cN(0,1)|\leq M)$ as $\gd,\gb \downarrow 0$. Put $\go = (y-x)\gb^{-1/3}$ and 
use \eqref{eq:Qasymp}, so that the integral becomes
\begin{equation}
[1+O(\gb^{2/9})] \int_\gep^N \dd x\, \int_{-N}^N\, \dd\go\, \gb^{-2/3}\, 
[f_\gb(x)-f_\gb(x+\gb^{1/3}\go)]^2\, \frac{1}{\sqrt{2\pi (2x)}}e^{-\frac{\go^2}{2(2x)}}.
\end{equation}
We are now in the same situation as in the proof of \cite[Lemma 7, Eq.\,(2.14)]{vdHdH95}. 
We refer to \cite[Eqs.\ (2.17)--(2.26)]{vdHdH95} to show that the limit of this integral as 
$\gb\downarrow 0$ is the integral with $[\go f'(x)]^2$ in place of $\gb^{-2/3}\,[f_\gb(x)
-f_\gb(x+\gb^{1/3}\go)]^2$ in the integrand. Letting $\gep \downarrow 0$ and $N, M \to \infty$, 
we get the desired result.
\end{proof}

\begin{lemma}
\label{lem:epi.two.inf}
Let $\beta = \tfrac12 \gd^2 - C(\tfrac12 \gd^2)^{4/3}$ and $\mu = B\gb^{4/3}$. Then, for all 
$f\in \cC$, 
\begin{equation}
\liminf_{\gb \downarrow 0} \Big\{ -\gb^{-2/3} I^2_{\gd,\gb}(f) \Big\} \geq - U_3(f).
\end{equation}
\end{lemma}

\begin{proof}
Since $\mu>0$, a first upper bound on $\gb^{-2/3} I^2_{\gd,\gb}(f)$ is 
\begin{equation}
\begin{aligned}
&\gb^{-2/3} I^2_{\gd,\gb}(f) \leq \tfrac12 \int_0^\infty \dd x \int_0^\infty \dd y\, 
\gb^{-4/3}\, [f(x) - f(y)]^2\\ 
&\qquad \qquad \times e^{G^*_{\gd,\gb}((x+y)\gb^{-2/3})}\, 
Q\left(\left\lfloor x\gb^{-2/3}\right\rfloor +1,\left\lfloor y\gb^{-2/3}\right\rfloor\right).
\end{aligned}
\end{equation}
Recall that
\begin{equation}
\label{eq:expGupperbound}
e^{G^*_{\gd,\gb}(\ell)} = \bbE\Big[ e^{-\gb \gO_\ell^2 + \gd \gO_\ell}\Big] 
\leq e^{\gd^2/4\gb} = \sqrt{e}[1+o(1)],\quad \mbox{as } \gb \downarrow 0,
\end{equation}
and that the maximum of $z\mapsto -\gb z^2 + \gd z$ is achieved at $z = \gd/2\gb \sim 
1/\sqrt{2\gb}$, as $\beta\downarrow 0$. As in the proof of \cite[Lemma 8]{vdHdH95}, we split the integral into three 
parts (note that $\gb$ there is $\gb^2$ here). Fix $c>0$. Part 1 corresponds to $x > c\gb^{-1/3}$ 
or $y > c\gb^{-1/3}$. We may use \eqref{eq:expGupperbound} and \cite[Eqs.\ (2.28)--(2.29)]{vdHdH95} 
to show that this part is negligible. In Part 2 we integrate over $x,y\leq c \gb^{-1/3}$ and 
$|x-y|>\gb^{1/12}$. Again, \eqref{eq:expGupperbound} and \cite[Eq.\ (2.32)]{vdHdH95} are 
enough to conclude. Finally, in Part 3 we integrate over $x,y\leq c \gb^{-1/3}$ and $|x-y|\leq \gb^{1/12}$. 
We only need to prove that the factor $\exp(G^*_{\gd,\gb}(\lfloor(x+y)\gb^{-2/3})\rfloor)$ 
is harmless. Indeed, let $\kappa < 1/\sqrt{2}$. Abbreviating $\ell = (x+y)\gb^{-2/3}$, we get
\begin{equation}
\begin{aligned}
& \bbE\Big[e^{-\gb \gO_\ell^2 + \gd \gO_\ell} 
\ind_{\left\{\gO_\ell \leq \kappa/\sqrt{\gb}\right\}} \Big] 
\leq e^{-\kappa^2 + \kappa\sqrt{2}},\\
& \bbE\Big[e^{-\gb \gO_\ell^2 + \gd \gO_\ell} 
\ind_{\left\{\gO_\ell > \kappa/\sqrt{\gb}\right\}} \Big] 
\leq [\sqrt{e} + o(1)]\, \bbP\Big(\gO_\ell > (\kappa/\sqrt{2c}) \sqrt{\ell} \Big).
\end{aligned}
\end{equation}
Therefore we get the result by first letting $\gb\downarrow 0$ (see \cite[Eq.\ (2.34)--(2.43)]{vdHdH95}), 
then $c\downarrow 0$, and finally $\kappa\downarrow 0$.
\end{proof}

\medskip\noindent
{\bf 9.} 
The analogues of Lemmas~\ref{lem:epi.one.sup}--\ref{lem:epi.two.inf} for $\gb \downarrow 0$
and $\gd \in (0,\infty)$ fixed are proved in the same manner as in Steps 8 and 9. The details are 
left to the reader. Thus we have completed the proof of Lemmas~\ref{lem:epiconvergence}(1,2). 

\medskip\noindent
{\bf 10.}
The proof of Lemma~\ref{lem:epiconvergence}(3), which is requirement (III) in Proposition~\ref{prop:epi}, 
is given in Appendix~\ref{AppC}.

\appendix


\section{Properties of the weight function}
\label{AppA}

We prove three properties of the function $G^*_{\delta,\beta}$ defined in \eqref{eq:rel.G.star} 
that were used in Sections~\ref{subsec:eigenvalue},  \ref{ss:critcurve} and \ref{ss:WIL}.
 
\begin{proposition}
\label{pr:asymptGstar}
For $(\gd,\gb)\in \cQ$,
\begin{equation}
G^*_{\delta,\beta}(\ell) = - \tfrac12 \log \ell + O(1), \qquad \ell \to \infty.
\end{equation}
\end{proposition}

\begin{proof}
Recall that $\Omega_\ell = \sum_{k=1}^\ell \omega_k$. With the help of the local limit theorem
we can estimate
\begin{equation}
\label{eq:app.estG1}
\begin{aligned}
G^*_{\delta,\beta}(\ell) &\geq \log\bbE\left[e^{\delta \Omega_\ell-\beta \Omega_\ell^2} 
\,\ind_{\{\Omega_\ell \in [0,1]\}}\right] \\
&\geq C(\delta,\beta) + \log \bbP\big(\Omega_\ell \in [0,1]\big) \sim -\tfrac12 \log \ell,
\qquad \ell\to\infty,
\end{aligned}
\end{equation}
and 
\begin{equation}
\label{eq:app.estG2}
G^*_{\delta,\beta}(\ell) \leq \log \left[ \sum_{m\in \bbZ} e^{\delta m-\beta m^2} 
\bbP\big(\Omega_\ell \in (m, m+1]\big)\right] \leq \log(C \ell^{-1/2}),
\end{equation}
from which the claim follows.
\end{proof}

\begin{proposition}
\label{pr:G.conv-to-Brownian}
Let $D_{\gd,\gb}(\ell) = \exp\{G^*_{\gd,\gb}(\ell)\}-1$. For $\gb = \tfrac12\gd^2 
- C(\tfrac12 \gd^2)^{4/3}$ and $\ell = \lfloor u \gb^{-2/3} \rfloor$,
\begin{equation}
D_{\gd,\gb}(\ell) \sim (Cu - u^2) \gb^{2/3}, \qquad \gb \downarrow 0.
\end{equation}
\end{proposition}

\begin{proof}
Abbreviate $Z_\ell = \gO_\ell / \sqrt{\ell}$, $a=-\gb \ell$ and $b=\gd \sqrt{\ell}$. 
Note that $ab = O(\sqrt{\beta})$, while
$\bbE[(Z_\ell)^3] = O(\ell^{-1/2}) \to 0$ and $\bbE[(Z_\ell)^4]\to 3$. 
By expanding 
the exponential below to fourth order and keeping only the terms that are of order lower than 
or equal to $\gb^{2/3}$, we obtain
\begin{equation}
\begin{aligned}
D_{\gd,\gb}(\ell) &= \bbE\Big( e^{-\gb\ell Z_\ell^2 + \gd \sqrt{\ell} Z_\ell} \Big) - 1\\
&= (a+ \tfrac12 b^2) + \tfrac32 a^2 + \tfrac32 a b^2 + \tfrac18 b^4 + o(\gb^{2/3})\\
& = C\gb^{4/3}\ell - \gb^2 \ell^2 + o(\gb^{2/3}),
\end{aligned} 
\end{equation}
which is the desired result.
\end{proof}

\begin{proposition}
\label{pr:Gbetadelta.monotonicity}
For $\ell\in\N$, $\beta \mapsto G^*_{\delta,\beta}(\ell)$ is strictly decreasing on $[0,\infty)$ 
and $\delta \mapsto G^*_{\delta,\beta}(\ell)$ is strictly convex on $[0,\infty)$.
\end{proposition}

\begin{proof}
For $\beta \in (0,\infty)$, compute
\begin{equation}
\frac{\partial}{\partial\beta} G^*_{\delta,\beta}(\ell) 
= -\frac{\bbE\Big[ \Omega_\ell^2 e^{\delta \Omega_\ell-\beta \Omega_\ell^2} \Big]}
{\bbE\Big[ e^{\delta \Omega_\ell-\beta \Omega_\ell^2} \Big]},
\end{equation}
which is strictly negative because $\omega_0$ is non-degenerate. For $\delta \in (0,\infty)$, 
compute
\begin{equation}
\frac{\partial^2}{\partial\delta^2} G^*_{\delta,\beta}(\ell) 
= \langle\Omega_\ell^2\rangle - \langle \Omega_\ell \rangle^2 > 0,
\qquad \langle \cdot \rangle = \frac{\bbE[\,(\cdot)\, e^{\delta \Omega_\ell-\beta \Omega_\ell^2}]}
{\bbE[e^{\delta \Omega_\ell-\beta \Omega_\ell^2}]}.
\end{equation}
Therefore $\delta \mapsto G^*_{\delta,\beta}(\ell)$ is stricly convex on $(0,\infty)$. \end{proof}

\begin{proposition} 
{\rm (van der Hofstad and den Hollander~\cite[Lemma 3]{vdHdH95})}
\label{lem:asymptP}
\begin{equation}
\label{eq:Qasymp}
Q(i+1,j) = \frac{1}{\sqrt{2\pi (i+j)}} \exp\left\{ -\frac{(i-j)^2}{2(i+j)} \right\}\,
\left(1 + O\left( \frac{1}{(i+j)^{1/3}} \right) \right)
\end{equation}
for $i,j\to\infty$ with $i-j = O((i+j)^{2/3})$.
\end{proposition}


\section{Key ingredients for the charge central limit theorem}
\label{AppB}

In Section~\ref{CLTproof} we gave the proof of the central limit theorem for the speed.
In this appendix we list the key ingredients necessary to extend the argument to the 
charge. This comes in 3 Steps.

\medskip\noindent
{\bf 1.} 
Write, for $\lambda \in \R$ (recall \eqref{eq:def.polmeasure}--\eqref{eq:def.annpartfunc}),
\begin{equation}
\begin{aligned}
&\bE_n^{\gd,\gb}\Big[e^{\gl[\gO_n - n\rho(\gd,\gb)]/\sqrt{n}} \Big] 
= e^{-\gl\rho(\gd,\gb)\sqrt{n}}\,\frac{\bbZ_n^{*,\gd+ (\gl/\sqrt{n}),\beta}}{\bbZ_n^{*,\gd,\beta}}\\
&\qquad = e^{-\gl\rho(\gd,\gb)\sqrt{n}+\big[\mu(\gd+(\gl/\sqrt{n}),\gb) - \mu(\gd,\gb)\big]n}
\, \frac{e^{-\mu(\gd+(\gl/\sqrt{n}),\gb)n}
\,\bbZ_n^{*,\gd+(\gl/\sqrt{n}),\beta}}{e^{-\mu(\gd,\gb)n}\bbZ_n^{*,\gd,\beta}}.
\end{aligned}
\end{equation}
By \eqref{eq:chargespec} and the analyticity of $\delta \mapsto \mu(\gd,\gb)$, as 
$n\to\infty$ 
the first factor converges to
\begin{equation}
\exp\Big[\tfrac12\gl^2 \frac{\partial^2}{\partial\delta^2}\,\mu(\gd,\gb)\Big],
\end{equation} 
which is the desired limit, see \eqref{eq:formula.variance}. Therefore we need to prove that the ratio in the second factor 
converges to $1$. An adaptation of Lemma~\ref{lem:CLT3} gives 
\begin{equation}
e^{-\mu(\gd,\gb)n} \bbZ_n^{*,\gd,\gb} = \sum_{a,b,n_1,n_2} u(n_1,n_2,a,b)\, 
\bP^{\gd,\gb}\big(\exists\, j\colon\, Y_j = n-n_1-n_2, M^+_{j-1}=b ~\big|~
M_0^+ = a \big).
\end{equation}
Recalling points \eqref{it:1}-\eqref{it:2} on page \pageref{it:1}, it is enough to show that 
the probabilities
\begin{equation}
\bP^{\gd,\gb}\big(\exists\, j\colon\, Y_j = n, M^+_{j-1}=b ~\big|~ M_0^+ = a \big),
\qquad  
\bP^{\gd+(\gl/\sqrt{n}),\gb}\big(\exists\, j\colon Y_j = n, M^+_{j-1}=b ~\big|~
M_0^+ = a\big),
\end{equation}
both converge to $m(b)$ as $n\to\infty$ (recall \eqref{eq:CLTstep5}). The difficulty is that in the 
second probability one of the parameters depends on $n$. To handle this, we need the following 
uniform renewal theorem.

\begin{lemma} 
\label{lem:appr.ren.thm}
Suppose that $(P^{(n)})_{n\in\N}$ is a sequence of inter-arrival time distributions, each with 
finite mean, converging to an inter-arrival time distribution $P$, also with finite mean. Suppose 
further that there exists a constant $c \in (0,\infty)$ such that
\begin{equation}
\label{tailass}
\sup_{n\in\N} P^{(n)}(\tau_1 = k) \leq e^{-ck}, \qquad k \in \N.
\end{equation}
Then, $\lim_{n\to\infty} 
|P^{(n)}(n\in\tau) - 1/E[\tau_1]| = 0$, where $\tau=(\tau_i)_{i\in\N}$ denotes
the sequence of arrival times.
\end{lemma}

\begin{proof}
By \eqref{tailass} and dominated convergence, we have $\lim_{n\to\infty} E^{(n)}[\tau_1] = E[\tau_1]$. 
According to Ney~\cite{N81}, there exists a constant $c_0$ such that
\begin{equation}
|P^{(n)}(n\in\tau) - E^{(n)}[\tau_1]^{-1}| \leq c_0\, P^{(n)}\left( \sum_{i=1}^N Z_i > n \right),
\end{equation}
where the $Z_i$'s are i.i.d.\ with 
$P^{(n)}(Z_1 >k) \leq c_1 e^{-c k}$, and $P^{(n)}(N > k) \leq c_2 (1-\gd)^k$ 
for all $k\in\N$ with $\gd \in (0,1)$. Moreover,
condition \eqref{tailass} ensures that the 
constants $c_0, c_1, c_2$ are uniform in $n\in\N$, and that $\gd$ is bounded away from $0$
uniformly in $n\in\N$.
\end{proof}

\medskip\noindent
{\bf 2.}
Abbreviate $\bP^{(n)} = \bP^{\gd+\gl/\sqrt{n},\gb}$. Apply Lemma~\ref{lem:appr.ren.thm} to 
the renewal process whose inter-arrival times have law
\begin{equation}
\cK^b_n(\ell) = \sum_{i\in\N} \bP^{(n)} (M_1^+ \neq b, \ldots,  M_{i-1}^+ \neq b, M_i^+ = b, 
Y_i = \ell\, |\, M_0^+ = b), \qquad \ell\in\N,
\end{equation}
delayed by a first inter-arrival time with law 
\begin{equation}
\cK^{a,b}_n(\ell) = \sum_{i\in \N} \bP^{(n)} (M_1^+ \neq b, \ldots,  
M_{i-1}^+ \neq b, M_i^+ = b, Y_i = \ell\, |\, M_0^+ = a), \qquad \ell\in\N.
\end{equation}
We can now explain why \eqref{tailass} holds. For simplicity, we restrict to the case without delay, i.e., 
$a=b$. Let $K\in\N$ be such that $b \leq K$, and put $\sigma = \inf\{i\in\N\colon\, M_i^+ \leq K\}$. 
Define $\tilde M_1 = \sum_{1 \leq i \leq \sigma} 2M_i^+ + 1$. We rely on the following lemma, whose
proof is given in Step 3 below.

\begin{lemma}
\label{lem:RFSS}
There exists a $c>0$ (depending on $K$) such that
\begin{equation}
\sup_{1\leq i \leq K} \sup_{n\in\N} \bE^{(n)}\Big(e^{c \tilde M_1} \left| M_0^+ = i \right.\Big) 
< \infty.
\end{equation}
This implies that there exists a $C >0$ such that 
\begin{equation}
\sup_{1\leq i \leq K} \bP^{(n)}(\tilde M_1 \geq n\, |\, M_0^+ = i) 
\leq C\, e^{-c n}, \qquad n\in\N.
\end{equation}
\end{lemma}

\noindent
Define, recursively, $\sigma_0 = 0$ and $\sigma_i = \inf\{k > \sigma_{i-1}\colon\, M^+_k \leq K\}$, 
$i\in\N$, as well as the random processes
\begin{equation}
\tilde M_k = \sum_{\sigma_{k-1} < i \leq \sigma_k} (2 M_i^+ + 1),\qquad \gL_k 
= M_{\sigma_k}^+, \qquad k\in\N.
\end{equation}
Note that the pair $(\tilde M_k, \gL_k)_{k\in\N}$ constitutes a Markov renewal process, 
and that
\begin{equation} 
\eta = \inf_{1\leq i, j \leq K}
\, \inf_{n\in\N} \, \bP^{(n)}( \gL_1 = i \,|\, \gL_0 = j) > 0.
\end{equation}
We can now derive an exponential upper bound on $\cK_n^b(\ell)$. Indeed, write
\begin{equation}
\cK^b_n(\ell) = \sum_{i\in\N} \bP^{(n)}\Big(\tilde M_1 + \ldots + \tilde M_i = \ell,\, 
\gL_1 \neq b,\ldots, \gL_{i-1} \neq b, \gL_i = b ~\Big|~ M_0^+ = b\Big).
\end{equation}
Split the sum according to $i \leq \gamma \ell$ and $i> \gamma \ell$, with $\gamma \in (0,1)$ 
a constant to be determined later. We have
\begin{equation}
\cK^b_n(\ell) \leq \bP^{(n)}\Big( \sum_{i=1}^{\gamma \ell} \tilde M_i \ge \ell \Big) 
+ \bP^{(n)}(\gL_1 \neq b, \ldots, \gL_{\gamma \ell} \neq b) \leq e^{-\ga \ell} 
\bE^{(n)}\Big(e^{\ga \sum_{i=1}^{\gamma\ell} \tilde M_i} \Big) + (1-\eta)^{\gamma \ell},
\end{equation}
where $\ga>0$. Using Lemma~\ref{lem:RFSS}, we know that 
\begin{equation}
\bE^{(n)}\Big(e^{\ga \sum_{i=1}^{\gamma\ell} \tilde M_i} \Big)  
\leq \exp\Big( \gamma \ell \log \sup_{1\leq i\leq K} \sup_{n\in\N} 
\bE^{(n)}[e^{\ga \tilde M_1} | M_0^+ = i]  \Big),
\end{equation}
which is finite for $\ga$ small enough. Therefore, choosing $\gamma$ small enough, we 
find a $c>0$ such that
\begin{equation}
\sup_{n\in\N} \cK_n^b(\ell) \leq e^{-c\ell}, \quad \ell \in \N.
\end{equation}

\medskip\noindent
{\bf 3.} 
We conclude by giving the proof of Lemma~\ref{lem:RFSS}.

\begin{proof}
For the moment, ignore the dependence on $n$ and write $\bP$, $\bE$ instead of $\bP^{(n)}$, $\bE^{(n)}$ (see the comment at the end of the proof). 
As we will see, the difficulty lies in the lack of uniform exponential decay for the one-step 
transition probability of $(M^+_n)_{n\in\N_0}$ w.r.t.\ the initial state. Instead, we will analyse 
the exponential decay of $Q_{\gd,\gb}(k, ak)$ as $k\to\infty$, and prove that the maximum 
is attained at $a_0 <1$ (Step I below). Consequently, even when the Markov chain starts 
at a large initial state, it {\it quickly} returns to a predetermined finite subset of its state space
(Step II below).

\medskip\noindent
{\bf Step I.} 
By \eqref{eq:Qgdgb} and \eqref{eq:def.Abeta.p.r} (recall 
also \eqref{eq:QandP})
\begin{equation}
\label{eq:hatQestim}
\begin{split}
Q_{\gd,\gb}(i,j) & = e^{G^*(i+j+1) - \mu(i+j+1)} Q(i+1,j)\, \frac{\nu(j)}{\nu(i)} \\
& = e^{G^*(i+j+1) - \mu(i+j+1)}\,\frac{\bP(S_{i+j}=i-j)}{2}\, \frac{\nu(j)}{\nu(i)},
\end{split}
\end{equation}
where $Q(\cdot,\cdot)$ was defined in \eqref{eq:def.Qij}, we denote by $(S_n)_{n\ge 0}$ the 
simple symmetric random walk on $\Z$, and we suppress the dependence on $\gd$ and $\gb$ 
on $G^*$ and $\nu$. We first need to control the exponential decay of $\nu$. This is the content 
of the following lemma.

\begin{lemma}
\label{lem:nurate}
The limit
\begin{equation}
r = -\lim_{n\to\infty} \frac{1}{n} \log \nu(n),
\end{equation}
exists and is the positive solution of $\log \cosh(r) = \mu$. The same holds when $\nu$ is 
replaced by $\tilde\nu$.
\end{lemma}

\begin{proof}
Let
\begin{equation}
r = -\limsup_{n\to\infty} \frac{1}{n} \log \nu(n), \qquad 
\cV(x) = \sum_{n\in\N} e^{xn} \nu(n), \qquad x\in\R.
\end{equation}
Let $\gep>0$. Since $G^*$ is bounded, by Proposition~\ref{pr:asymptGstar},
by \eqref{eq:hatQestim} we get, for a constant $c_\gep>0$,
\begin{equation}
\nu(n) \leq c_\gep \sum_{m\in\N} e^{-\mu (m+n)} \bP(S_{n+m} = n-m) e^{-(r-\gep) m}.
\end{equation}
Therefore, using the change of variables $k = n+m$ and $\ell = n-m$, we obtain
\begin{equation}
\begin{aligned}
\cV(x) &\leq c_\gep \sum_{k\in\N,\ell\in\Z} e^{x(\frac{k+\ell}{2}) 
- \mu k - (r-\gep)(\frac{k-\ell}{2})} \bP(S_k = \ell),\\
& \leq c_\gep \sum_{k\in\N} e^{k(\frac{x}{2} -\mu - \frac{r-\gep}{2})} 
\bE[e^{S_k(\frac{x}{2} + \frac{r-\gep}{2})}],\\
& \leq c_\gep \sum_{k\in\N} \exp\Big\{k\Big[\frac{x}{2}
- \mu - \frac{r-\gep}{2} + \log \cosh \Big(\frac{x+r-\gep}{2}\Big)\Big]\Big\},
\end{aligned}
\end{equation}
from which we deduce, by evaluating at $x=r+\gep$, that $\log \cosh(r) \geq \mu - \gep$. 
Letting $\gep\downarrow 0$, we get
\begin{equation}
\label{eq:logchr}
\log \cosh(r) \geq \mu.
\end{equation}
Let us next prove that $r$ is a limit. Define
\begin{equation}
s = - \liminf_{n\to\infty} \frac{1}{n} \log \nu(n).
\end{equation}
By a standard large deviations estimate,
\begin{equation}\label{eq:LDCr}
\lim_{n\to\infty} \frac{1}{n} \log \bP(S_n = 
\lfloor xn \rfloor) = - I(x),
\end{equation}
where
\begin{equation}
I(x) = \frac{1+x}{2} \log(1+x) + \frac{1-x}{2} \log(1-x), \qquad |x| \leq 1.
\end{equation}
From \eqref{eq:hatQestim} we get (henceforth we assume that $an := \lfloor an \rfloor$, etc.\ 
for notational convenience)
\begin{equation}
\nu(n) \geq c\ e^{G^*(n+an)} e^{-\mu(n+an)} \bP(S_{(1+a)n} = (1-a)n) \nu(an), \qquad a
 \ge 0,
\end{equation}
where $c$ is a constant. Therefore, using Proposition~\ref{pr:asymptGstar}, we get
\begin{equation}
-s \geq -\mu(1+a)  - (1+a)I\Big(\frac{1-a}{1+a}\Big) - as, \qquad a \ge  0.
\end{equation}
Setting $a=0$ we obtain, in particular, that $s < \infty$. We can rewrite the previous relation 
as
\begin{equation}
\label{eq:phira}
\begin{split}
\sup_{a \ge 0} \phi(s,a) \leq 0, \qquad \phi(s,a) 
& = -\mu (1+a) + (1-a)s - (1+a)I\Big( \frac{1-a}{1+a} \Big) \\
& = (1+a) \left( - \mu + 
\left\{ s \frac{1-a}{1+a} - I\Big( \frac{1-a}{1+a} \Big) \right\} \right).
\end{split}
\end{equation}
Let us recall that the Fenchel-Legendre transform of the rate function $I(x)$ is the
log-moment generating function of the the simple random walk: $\sup_{|x| \le 1}\{sx-I(x)\} 
= \log \cosh(s)$. By direct computation, the $\sup$ in \eqref{eq:phira} is uniquely attained at $x = (1-e^{-2s})
/(1+e^{-2s})$, so that
\begin{equation}
\label{eq:lookat}
\phi(s,e^{-2s}) = (1+e^{-2s}) \big( -\mu + \log\cosh(s) \big),
\end{equation}
and in order to have $\phi(s,e^{-2s}) \le 0$ we must have $\mu \ge \log\cosh(s)$. Since 
by definition $s \ge r$, and hence $\log\cosh(s) \ge \log\cosh(r)$, it follows from 
\eqref{eq:logchr} that $s=r = \log\cosh^{-1}(\mu)$.

Only a few minor modifications lead to the same result for $\tilde \nu$.
\end{proof}

Note that 
\begin{equation}
\label{eq:LDPhatQ}
\phi(a) = \phi(s,a) = \lim_{k\to \infty} \frac{1}{k} \log Q_{\gd,\gb}(k,ak), \qquad a>0,
\end{equation}
$\phi(a) = \phi(s,a)$ is defined in \eqref{eq:phira} with $s = r = \log\cosh^{-1}(\mu)$. 
Setting $a_0 = e^{-2s} < 1$, we know that $\phi(a_0) = 0$, while $\phi(a) < 0$ for 
$a \ne a_0$. Henceforth we fix $\epsilon > 0$ so that $a_0^+ = a_0 + \epsilon < 1$.
Again by \eqref{eq:phira}, we have $\lim_{a\to\infty}\phi(a)/a = - \mu - s - I(-1) < 0$,
hence $\sup_{x\ge a_0^+} \phi(x) / x < 0$.

We next choose $\eta > 0$ small, so that
\begin{equation} 
\label{eq:xi}
\xi = - \sup_{x \ge a_0^+} \Big\{\eta \Big(1+\frac{1}{x}\Big) + \frac{\phi(x)}{x} \Big\} > 0,
\end{equation}
Let us reinforce \eqref{eq:LDPhatQ}. We fix $K_0 < \infty$ such that $e^{-(r+\eta)i} 
\le \nu(i) \le e^{-(r-\eta)i}$ for $i \ge K_0$. Since $G^*$ is bounded, and we can 
replace $\lim_{n\to\infty}$ by $\sup_{n\in\N}$ in \eqref{eq:LDCr} by super-additivity, 
we get the following upper bound from \eqref{eq:hatQestim}, for $\ell, k \ge K_0$:
\begin{equation}
Q_{\gd,\gb}(\ell,k) \le C \, e^{\eta(\ell+k)} \, e^{-\mu(k+\ell)
-r(k-\ell) - (k+\ell)\, I(\frac{\ell-k}{\ell+k})} 
= C \, e^{\eta(\ell+k)} \, e^{\ell\phi(\frac{k}{\ell})},
\end{equation}
where $C$ is an absolute constant. In particular, recalling \eqref{eq:xi} and setting 
$K = K_0/a_0^+$, we get
\begin{equation} 
\label{eq:eheh}
Q_{\gd,\gb}(\ell,k) \le C \, e^{-\xi k} \,, \qquad \ell \ge K, \qquad k \ge a_0^+ \ell.
\end{equation}
In the next step we complete the proof of Lemma~\ref{lem:RFSS}, choosing $c < \xi$.

\medskip\noindent
{\bf Step II.} Define 
\begin{equation}
\cM_N = e^{R \sum_{i=1}^N M_i^+}, \qquad N\in\N,
\end{equation}
for some $R \in (0,\mu)$ to be determined later. Define the filtration $\cF_N 
= \sigma(M_i^+, i\leq N)$, $N\in\N_0$. Then
\begin{equation}
\bE(\cM_{N+1} | \cF_N) = \cM_N  \bE(e^{RM_{N+1}^+}| M_N^+),
\end{equation}
so the process defined by
\begin{equation}
\hat \cM_{N} = \exp\Big\{ R\sum_{i=1}^N M_i^+ - \sum_{i=1}^N \log 
\bE(e^{R M_i^+} | M_{i-1}^+) \Big\}, \qquad N\in\N,
\end{equation}
is a martingale with respect to $(\cF_N)_{N\in\N}$. (Even though we do not know yet
if $\cM_N$ is integrable, we know by \eqref{eq:hatQestim} that $\log \bE(e^{R M_i^+} 
| M_{i-1}^+)$ is a.s.\ finite for $R > 0$ small, and so $\hat \cM_N$ is well-defined with integral 
$1$.) Since $\sigma = \inf\{i\in\bbN\colon M_i^+\leq K\}$ is a stopping time, we get 
$\bE(\hat \cM_{N\wedge \gs}) = 1$.

We next provide an upper bound on $\bE(e^{RM_i^+} | M_{i-1}^+)$ for $1<i\leq \sigma$. 
Let $\gep >0$ be such that $a_0^+ = a_0 + \gep <1$. We know from Step I that typically 
$M_i^+$ is at most $a_0^+ M_{i-1}^+$. We split accordingly:
\begin{equation}
\label{eq:step2upperbound}
\begin{aligned}
&\bE(e^{R_\gep M_i^+} | M_{i-1}^+)\\ 
&= \bE\big(e^{R_\gep M_i^+} \ind_{\{M^+_i > a_0^+ M_{i-1}^+\}}\big| M_{i-1}^+\big)  
+ \bE\big(e^{R_\gep M_i^+} \ind_{\{M_i^+ \leq a_0^+ M_{i-1}^+\}} \big| M_{i-1}^+\big)\\
&\leq \bE\big(e^{R_\gep M_i^+} \ind_{\{M_i^+ > a_0^+ M_{i-1}^+\}}\big| M_{i-1}^+\big) + e^{R_\gep a_0^+ M_{i-1}^+}.
\end{aligned}
\end{equation}
Recalling \eqref{eq:eheh}, for $\ell \ge K$,
\begin{equation} 
\label{eq:step2upperbound2}
\begin{aligned}
\bE\big(e^{R_\gep M^+_i} \ind_{\{M^+_i > a_0^+ M^+_{i-1}\}}\big| M^+_{i-1} = \ell\big) 
&\leq \sum_{k \geq a_0^+ \ell} e^{R_\gep k} Q_{\gd,\gb}(\ell,k)
\le \sum_{k \geq a_0^+ \ell} C e^{-(\xi - R_\gep)k},
\end{aligned}
\end{equation}
which is finite for $R_\gep < \xi$. Combining \eqref{eq:step2upperbound2} with
\eqref{eq:step2upperbound}, we have (possibly enlarging $K$)
\begin{equation}
\bE(e^{R_{\gep} M_i^+} | M_{i-1}^+) \leq \exp(R_\gep a_0^+(1+\gep) M_{i-1}^+),
\qquad 1<i \leq \sigma.
\end{equation}
In the sequel, we redefine $a_0^+$ as $a_0^+(1+\gep)$, which we may safely assume 
to be $<1$. We get, since $M_0^+ = b \leq K$ a.s.,
\begin{equation}
1 = \bE[\hat \cM_{N \wedge \sigma}] \geq \bE(e^{R_\gep(1-a_0^+) 
\sum_{i=1}^{N \wedge \sigma} M^+_i - R_\gep a_0^+ K}).
\end{equation}
Note that $\sigma$ is a.s.\ finite, because the Markov chain $(M_N^+)_{N\in\N_0}$ is positive 
recurrent. Therefore, by Fatou's lemma,
\begin{equation}
\bE(e^{ R_\gep(1-a_0^+) \sum_{i=1}^\sigma M_i^+ }) \leq e^{R_\gep a_0^+ K} < \infty.
\end{equation}
which is the desired result.

Recall that $\bP$ is originally $\bP^{(n)} = \bP^{\gd + \gl/\sqrt{n}, \gb}$. To deal with 
uniformity in $n$, it is enough to note that $\gd\mapsto \mu(\gd,\gb)$ is continuous, so the 
limit $r$ in Lemma~\ref{lem:nurate} and the function $\phi$ in \eqref{eq:LDPhatQ} are continuous 
in $\gd$, and hence in $n$. The details are left to the reader.
\end{proof}


\section{Tail estimate for the eigenvector}
\label{AppC}

In this appendix we prove Lemma~\ref{lem:epiconvergence}(3). The proof strategy is that 
of \cite[Proposition 4]{vdHdH95}. For the sake of conciseness, we only write the proof in 
the regime
\begin{equation} 
\label{eq:regime}
\beta \downarrow 0, \qquad \tfrac12\delta^2 = \beta + C \beta^{4/3}, 
\qquad \mu = B \beta^{4/3}.
\end{equation}
The regime $\beta \downarrow 0$ with $\gd$ fixed follows the same line of argument and 
is left to the reader. In what follows we use $A_\gb$ and $\gl(\gb)$ as shorthand notation 
for $A_{\mu,\gd,\gb}$ and $\gl_{\gd,\gb}(\mu)$ with $\mu,\gd$ as in \eqref{eq:regime}.

We choose for $f_\gb$ the linear and renormalised interpolation of the solution of the discrete 
variational problem. Namely, if $\tau_\gb$ is the normed eigenvector of $A_\gb$ associated 
with $\gl(\gb)$, then we define
\begin{equation}
\hat\tau_\gb(u) = \gb^{-1/3} \Big\{ \tau_\gb(i) + (u\gb^{-2/3} - i)[\tau_\gb(i) - \tau_\gb(i-1)] \Big\},
\quad (i-1)\gb^{2/3} <u \leq i\gb^{2/3}, \quad i\in\bbN,
\end{equation}
and we pick $f_\gb = \hat\tau_\gb / \|\hat\tau_\gb\|_2$ as an approximate maximiser. Mimicking 
\cite[Lemmas 9--10]{vdHdH95}, we see that it is enough to prove the following adaptation of 
\cite[Lemma 11]{vdHdH95}: for $\gb$ small enough,
\begin{equation}
\label{eq:i-iv}
\begin{array}{lll}
&{\rm (i)}\quad \sum_{i\in\N} i^2 \tau_\gb^2(i) \leq C_1 \gb^{-4/3},
&{\rm(ii)}\quad \sum_{i\in\N} i \gD  \tau_\gb^2(i) \leq C_2 \gb^{2/3},\\
&{\rm(iii)}\quad  \tau_\gb^2(0) \leq C_3 \gb^{2/3} \log(1/\gb),
&{\rm(iv)}\quad \|\gD\tau_\gb\|_2^2 \leq C_4 \gb^{4/3} \log(1/\gb),
\end{array}
\end{equation}
where $\gD\tau_\gb(i) = \tau_\gb(i) - \tau_\gb(i-1)$, $i\in\N$, and the $C_i$'s are constants 
whose precise values are irrelevant. The estimates in \eqref{eq:i-iv} are proved in two steps: 
Gaussian disorder (Section~\ref{Gaussdisorder} below) and General disorder 
(Section~\ref{Generaldisorder} below). The first allows for explicit formulas, the second 
makes use of Gaussian approximations and Taylor expansions. To avoid a lengthy proof, 
we only indicate the necessary modifications to the proof in \cite{vdHdH95}.

Recall \eqref{eq:def.Qij} and \eqref{eq:def.Abeta.p.r}. In what follows we use the short-hand 
notation
\begin{equation}
A(i,j) = e^{h(i+j+1)}\,P(i,j), \qquad i,j\in\N_0,
\end{equation}
with
\begin{equation}
\label{eq:def_fh}
P(i,j)= Q(i+1,j) = \binom{i+j}{i}\,\frac{1}{2^{i+j+1}}, \qquad h(x) = G^*(x) - \mu x.
\end{equation}

\subsection{Gaussian disorder}
\label{Gaussdisorder}
 
\begin{proof}[Proof of (i) in \eqref{eq:i-iv}]

\medskip\noindent
{\bf 1.}
Recall \eqref{eq:Gstar.gaussian}. We have
\begin{equation}
\label{eq:f}
\begin{split}
h(x) & = \frac{\tfrac12\delta^2x}{1+2\beta x}
-\tfrac12 \log (1+2\beta x) - \mu x\\
& = \left(\tfrac12\delta^2-\beta-\mu\right) x
- \tfrac12\delta^2 x\, \frac{2\beta x}{1+2\beta x}
+ \tfrac12 \big[2\beta x - \log(1+2\beta x)\big].
\end{split}
\end{equation}
Recall \eqref{eq:regime}. A Taylor expansion as $x \downarrow 0$ 
(uniformly over $\beta$, as long as $\beta x \to 0$)
yields
\begin{equation} 
\label{eq:Taylor}
h(x) = \left(\tfrac12\delta^2-\beta-\mu\right) x
- \left(\gd^2 - \beta \right) \beta x^2 
+ O\left( \left(\max\{\tfrac12\delta^2, \beta\} x \right)^3 \right),
\end{equation}
and hence
\begin{equation} 
\label{eq:Taylor2}
h(x) = (C-B) \beta^{4/3} x - \beta^2 x^2 -2 C \beta^{2+1/3} x^2 + O((\beta x)^3).
\end{equation}
Henceforth we fix $\epsilon > 0$ such that $O((\gb x)^3) \leq \frac{1}{2} \gb^2x^2$ for all $x \leq \epsilon/\gb$.
In this way we get
\begin{equation} 
\label{eq:f<0}
h(x) \leq (C-B) \beta^{4/3} x - \tfrac12 \beta^2 x^2 \qquad 
\forall\,\,x \leq \frac{\epsilon}{\beta}.
\end{equation}
Also note that $h(\frac{\epsilon}{4\beta}) \sim -\frac{1}{32}\epsilon^2$ as $\beta \downarrow 0$,
and hence $h(\frac{\epsilon}{4\beta}) \leq -\frac{1}{40}\epsilon^2$ for $\beta > 0$ small enough.
We show in the proof of Lemma~\ref{th:sup} below that $h$ attains its global maximum at 
$\bar x = O(\beta^{-2/3})$, and that $x \mapsto h(x)$ is decreasing for $x \geq \bar x$. Since 
$\bar x \leq \frac{\epsilon}{8\beta}$ for $\beta > 0$ small, we have shown that
\begin{equation}
\label{eq:f<}
h(x) \leq -\frac{1}{40} \epsilon^2 \qquad \forall\,\,x \geq \frac{\epsilon}{ 8\beta}.
\end{equation}
Note that \eqref{eq:f<0} is the analogue of \cite[eq. (2.5) (i)]{vdHdH95} in our context. The 
analogue of \cite[eq. (2.5) (ii)]{vdHdH95} is given by the next lemma, where we estimate 
the global maximum of $f$.

\begin{lemma}
\label{th:sup}
\begin{equation}
\label{eq:sup}
\sup_{x \geq 0} h(x) = \begin{cases}
0, & \text{if } B > C, \\
\beta^{2/3} \left(\frac{C-B}{2}\right)^2 + O(\beta),
& \text{if } B \le C.
\end{cases}
\end{equation}
\end{lemma}

\begin{proof}
Note that
\begin{equation}
h'(x) = \frac{\tfrac12\delta^2}{(1+2\beta x)^2} - \frac{\beta}{1+2\beta x} - \mu.
\end{equation}
Setting $z = \frac{1}{1+2\beta \bar x}$, we have $h'(\bar x)=0$ if and only if
$\tfrac12\delta^2 z^2 - \beta z - \mu = 0$, whose positive solution (if any) is
\begin{equation}
z = \tfrac12 \left\{\frac{\beta}{\tfrac12\delta^2} 
+ \sqrt{\left( \frac{\beta}{\tfrac12\delta^2} \right)^2 + \frac{4 \mu}{\tfrac12\delta^2}}\right\}.
\end{equation}
Recalling \eqref{eq:regime}, we have $\frac{\beta}{\tfrac12\delta^2} - 1 = -C\beta^{1/3} + O(\gb^{2/3})$, 
and hence
\begin{equation}
z = \tfrac12 \left\{1-C\beta^{1/3} + \sqrt{1 + 2(2B-C) \beta^{1/3} + O(\beta^{2/3}) }\right\} 
= 1 - (C-B) \beta^{1/3} + O(\beta^{2/3}).
\end{equation}

Since $z = \frac{1}{1+2\beta \bar x} = 1 - 2\beta \bar x + O((\beta \bar x)^2)$, we get
\begin{equation}
\bar x = \tfrac{C-B}{2} \beta^{-2/3} + O(\beta^{-1/3}).
\end{equation}
Recalling \eqref{eq:Taylor2}, we get
\begin{equation}
\begin{split}
\sup_{x \geq 0} h(x) 
& = h(\bar x) = (C-B) \beta^{2/3} \tfrac{C-B}{2} - \beta^{2/3}
\left(\tfrac{C-B}{2}\right)^2 - 2C \beta \left(\tfrac{C-B}{2}\right)^2 + O\left(\beta \right) \\
& = \left(\tfrac{C-B}{2}\right)^2 \beta^{2/3} + O(\beta)
\end{split}
\end{equation}
which is the claim.
\end{proof}

Before continuing with the main line of the proof, we provide two estimates that are the analogue 
of \cite[Lemma 13]{vdHdH95} (see \eqref{eq:lambdaUB}-\eqref{eq:lambdaLB} below). Note that 
$A(i,j) = e^{h(i+j+1)}P(i,j) \leq e^{c \beta^{2/3}} P(i,j)$ by Lemma~\ref{th:sup}. Arguing as in 
\cite[eq. (4.16)]{vdHdH95}, we get
\begin{equation}
\lambda(\beta) = \sum_{i,j\in\N_0} \tau_\beta(i) A(i,j) \tau_\beta(j)
\le e^{c \beta^{2/3}} \sum_{i,j\in\N_0} \tau_\beta(i) P(i,j) \tau_\beta(j)
\le e^{c \beta^{2/3}},
\end{equation}
because $\|P\|_{\rm op} \leq 1$ and $\sum_{i\in\N_0} \tau_\beta(i)^2 =1$. Consequently,
\begin{equation} 
\label{eq:lambdaUB}
\limsup_{\beta \downarrow 0} \frac{\lambda(\beta)-1}{\beta^{2/3}} \le c < \infty.
\end{equation}
For an analogous lower bound, arguing as in \cite[eq. (4.17)-(4.20)]{vdHdH95}, we get
\begin{equation} 
\label{eq:lambdaLB}
\limsup_{\beta \downarrow 0} \frac{1-\lambda(\beta)}{\beta^{2/3}} < \infty.
\end{equation}
Recall that we want to prove (i) in \eqref{eq:i-iv}. We start by proving the analogue of 
\cite[Steps 1 and 2, pages 419-420]{vdHdH95}.

\begin{lemma}
\label{th:12}
The following relations hold (where $\epsilon > 0$ is fixed so that \eqref{eq:f<0} is in force):
\begin{equation}
\label{eq:12}
\sum_{i \leq \frac{\epsilon}{4\beta}} i^2 \tau_\beta(i)^2 \le C_5 \, \beta^{-4/3},
\qquad \sum_{i > \frac{\epsilon}{ 8\beta}} \tau_\beta(i)^2 \le C_7 \, \beta^{2/3}.
\end{equation}
\end{lemma}

\begin{proof}
In analogy with \cite[eq. (3.15)]{vdHdH95}, from the trivial inequality 
\begin{equation}
\sum_{i,j\in\N_0} [\tau_\beta(i) - \tau_\beta(j)]^2 A(i,j) \geq 0
\end{equation} 
we get
\begin{equation}
1-\lambda(\beta) + \sum_{i\in\N_0} \tau_\beta(i)^2 \sum_{j\in\N_0} 
[e^{h(i+j+1)}-1] P(i,j) \geq 0.
\end{equation}
Note that there exists a $t_0 > 0$ such that $e^t \le 1+t+t^2$ for all $t \in (-\infty, t_0]$. By 
Lemma~\ref{th:sup}, $\sup_{x \geq 0} h(x) \leq t_0$ provided $\beta$ is small enough.
Therefore
\begin{equation} 
\label{eq:plug}
\begin{split}
0 & \le 1-\lambda(\beta) + \sum_{i \in \N_0} \tau_\beta(i)^2 
\sum_{j\in\N_0} [h(i+j+1) + h(i+j+1)^2] P(i,j)  \\
& \le 1-\lambda(\beta) + \sum_{i\in\N_0} \tau_\beta(i)^2 
\sum_{j\in\N_0} h(i+j+1) P(i,j) + O(\beta^{4/3}),
\end{split}
\end{equation}
where the second inequality follows again by Lemma~\ref{th:sup} and the fact that 
$\sum_{j\in\N_0} P(i,j) = 1$. By \eqref{eq:f<}, we can write
\begin{equation}
\label{eq:if1}
\sum_{j\in\N_0} h(i+j+1) P(i,j) \leq -\frac{\epsilon^2}{40}
\sum_{j\in\N_0} P(i,j) = -\frac{\epsilon^2}{40}, \qquad i > \frac{\epsilon}{4\beta}.
\end{equation}
On the other hand, for $i \leq \frac{\epsilon}{4\beta}$ we can bound
\begin{equation}
\sum_{j\in\N_0} h(i+j+1) P(i,j) \leq \sum_{j \le \frac{\epsilon}{2\beta}} h(i+j+1) P(i,j),
\end{equation}
because for $j > \frac{\epsilon}{2\beta}$ we have $i+j+1 > \frac{\epsilon}{2\beta} \geq 
\frac{\epsilon}{8\beta}$, and consequently $h(i+j+1) < 0$ again by \eqref{eq:f<}. Having 
thus restricted the range of $j$, we have $i+j+1 \leq \frac{\epsilon}{\beta}$ and we can 
apply \eqref{eq:f<0}:
\begin{equation} 
\label{eq:bala}
i \leq \frac{\epsilon}{4\beta}\colon
\qquad \sum_{j\in\N_0} h(i+j+1) P(i,j) \le \sum_{j \le \frac{\epsilon}{2\beta}}
\left\{c\beta^{4/3}(i+j+1) - \frac{1}{2}\beta^2 (i+j+1)^2\right\} P(i,j)
\end{equation}
with $c = C-B$. The sums have been evaluated in \cite[eq. (1.17)]{vdHdH95} when $j$ runs 
over all of $\N_0$. Here we have the restriction $j \le \frac{\epsilon}{2\beta}$, which is 
harmless. In fact, since $i \leq \frac{\epsilon}{4\beta}$, the range of summation 
for $j$ includes $j \le 2i$, and we know that the mass of $P(i,j)$ is concentrated around the 
diagonal, and decays exponentially in $i,j$ when $|j-i| > \delta i$ for any $\delta > 0$. Consequently,
for some $c_1, c_2 \in (0,\infty)$ we have
\begin{equation}
\label{eq:if2}
i \leq \frac{\epsilon}{4\beta}\colon \qquad 
\sum_{j\in\N_0} h(i+j+1) P(i,j) \leq c_1 \beta^{4/3} i - c_2 \beta^2 i^2.
\end{equation}
Substituting \eqref{eq:if1} and \eqref{eq:if2} into \eqref{eq:plug}, we get
\begin{equation} 
\label{eq:bo}
0 \leq 1-\lambda(\beta) + O(\beta^{4/3})
+ c_1 \beta^{4/3} \sum_{i \le \frac{\epsilon}{4\beta}} i \tau_\beta(i)^2
- c_2 \beta^{2} \sum_{i \le \frac{\epsilon}{4\beta}} i^2 \tau_\beta(i)^2
-\frac{\epsilon^2}{10} \sum_{i > \frac{\epsilon}{4\beta}} \tau_\beta(i)^2.
\end{equation}
Let us abbreviate $I = \sum_{i \le \frac{\epsilon}{4\beta}} i^2 \tau_\beta(i)^2$ and 
$J = \sum_{i > \frac{\epsilon}{4\beta}} \tau_\beta(i)^2$. Note that $1-\lambda(\beta) 
\leq c_3 \beta^{2/3}$ by \eqref{eq:lambdaLB}. By Cauchy-Schwarz, $\sum_{i \le 
\frac{\epsilon}{4\beta}} i \tau_\beta(i)^2 \leq \sqrt{I}$, since $\sum_i \tau_\beta(i)^2 = 1$. 
Hence \eqref{eq:bo} yields
\begin{equation}
c_2 \beta^2 I + \frac{\epsilon^2}{10} J \le c_3 \beta^{2/3}
+ O(\beta^{4/3}) + c_1 \beta^{4/3} \sqrt{I}.
\end{equation}
Setting $c_4 = \frac{\epsilon^2}{10}$ and dividing by $\beta^{2/3}$, we obtain
\begin{equation} 
\label{eq:drop}
c_2(\beta^{4/3}I) + c_4 (\beta^{-2/3} J) \leq c_3 + o(1) + c_1 \sqrt{\beta^{4/3} I}.
\end{equation}
Since $J \geq 0$, we can drop it from the left-hand side. Setting $x = \beta^{4/3} I$, se see
that the inequality becomes $c_2 x \le c_3 + o(1) + c_1 \sqrt{x}$, which can only hold when 
$x$ is bounded from above, say $x \le C_5$. We have thus shown that $\beta^{4/3} I \le C_5$, 
i.e., the first relation in \eqref{eq:12}. Next, we can drop $I$ from the left-hand side of 
\eqref{eq:drop}, getting $c_4 (\beta^{-2/3} J) \le c_3 + o(1) + c_1 \sqrt{C_5}$, which shows 
that also $\beta^{-2/3} J$ is bounded from above, say $\beta^{-2/3} J \le C_7$. This 
\emph{almost} proves the second relation in \eqref{eq:12}, in the sense that we have proved it 
with $\frac{\epsilon}{4\beta}$ instead of $\frac{\epsilon}{8\beta}$. However, recalling \eqref{eq:f<},
we see that \eqref{eq:if1}, \eqref{eq:bala} and \eqref{eq:drop} still hold when $\frac{\epsilon}{4\beta}$ 
is replaced by $\frac{\epsilon}{8\beta}$. Consequently, writing \eqref{eq:bo} with $\frac{\epsilon}{4\beta}$ 
replaced by $\frac{\epsilon}{8\beta}$, we complete the proof of \eqref{eq:12}.
\end{proof}

We can now complete the proof of (i). Arguing as in \cite[eq. (3.24)]{vdHdH95}, by \eqref{eq:f<}, 
for every $\delta > 0$ (which will be fixed later) there is $C(\delta) > 0$ such that
\begin{equation}
\begin{split}
i > \frac{\epsilon}{4\beta}\colon \qquad 
\tau_\beta(i) & = \frac{1}{\lambda(\beta)} \sum_j A(i,j) \tau_\beta(j)
\le \frac{e^{-\frac{\epsilon^2}{10}}}{\lambda(\beta)} \sum_j P(i,j) \tau_\beta(j) \\
& \le (1+\delta) e^{-\frac{\epsilon^2}{10}} 
\sum_{j > (1-\delta) i} P(i,j) \tau_\beta(j) + O(e^{-C(\delta) i}) \\
& \le (1+\delta) e^{-\frac{\epsilon^2}{10}} 
\sqrt{\sum_{j > (1-\delta) i} P(i,j) \, \tau_\beta(j)^2} + O(e^{-C(\delta) i}),
\end{split}
\end{equation}
where the last inequality holds by Cauchy-Schwarz. Indeed, $P(i,j) = \sqrt{P(i,j)} \sqrt{P(i,j)}$ 
and hence
\begin{equation}
\sum_{j > (1-\delta) i} P(i,j) \tau_\beta(j)
\le \sqrt{\sum_{j > (1-\delta) i} P(i,j)} \sqrt{\sum_{j > (1-\delta) i} P(i,j) \, \tau_\beta(j)^2},
\end{equation}
and $\sum_{j\in\N_0} P(i,j) \le 1$. Since $\sum_{i > \frac{\epsilon}{4\beta}} i^2  e^{-C(\delta) i} = o(1)$
as $\beta \downarrow  0$, we get
\begin{equation}
\begin{split}
\sum_{i > \frac{\epsilon}{4\beta}} i^2 \, \tau_\beta(i)^2
& \le (1+\delta)^2 e^{-\frac{\epsilon^2}{5}} 
\left\{ \sum_{i > \frac{\epsilon}{4\beta}} i^2 \,\sum_{j > (1-\delta) i}
P(i,j) \, \tau_\beta(j)^2 \right\} [1+ o(1)] \\
& \le \frac{(1+\delta)^2}{(1-\delta)^2} e^{-\frac{\epsilon^2}{5}} 
\left\{ \sum_{j > (1-\delta) \frac{\epsilon}{4\beta}} j^2 \, \tau_\beta(j)^2
\sum_{i < \frac{j}{1-\delta}} 
P(i,j)  \right\} [1+ o(1)]\\
& \le \frac{(1+\delta)^2}{(1-\delta)^2} e^{-\frac{\epsilon^2}{5}} 
\left\{ \sum_{j > (1-\delta) \frac{\epsilon}{4\beta}} j^2 \, \tau_\beta(j)^2
\right\} [1+ o(1)],
\end{split}
\end{equation}
because $\sum_{i\in\N_0} P(i,j) \le 1$. Consequently,
\begin{equation}
\begin{split}
\sum_{i > \frac{\epsilon}{4\beta}} i^2 \, \tau_\beta(i)^2
& \le \frac{(1+\delta)^2}{(1-\delta)^2} e^{-\frac{\epsilon^2}{5}} 
\left\{ \sum_{i > \frac{\epsilon}{4\beta}} i^2 \, \tau_\beta(i)^2
+ \sum_{(1-\delta)\frac{\epsilon}{4\beta}
< i \le \frac{\epsilon}{4\beta}} i^2 \, \tau_\beta(i)^2 \right\}  [1+ o(1)]\\
& \le \frac{(1+\delta)^2}{(1-\delta)^2} e^{-\frac{\epsilon^2}{5}} 
\left\{ \sum_{i > \frac{\epsilon}{4\beta}} i^2 \, \tau_\beta(i)^2
+ \frac{\epsilon^2}{16\beta^2} \sum_{i > (1-\delta)\frac{\epsilon}{4\beta}} 
\tau_\beta(i)^2 \right\}  [1+ o(1)]\\
& \le \frac{(1+\delta)^2}{(1-\delta)^2} e^{-\frac{\epsilon^2}{5}} 
\left\{ \sum_{i > \frac{\epsilon}{4\beta}} i^2 \, \tau_\beta(i)^2
+ \frac{\epsilon^2}{16\beta^2} C_7 \, \beta^{2/3} \right\}  [1+ o(1)],
\end{split}
\end{equation}
where in the last inequality we use the second relation in \eqref{eq:12}, provided we 
choose $\delta < \tfrac12$ so that  $(1-\delta)\frac{\epsilon}{4\beta} > \frac{\epsilon}{8\beta}$.
Fixing $\delta > 0$ small enough so that $(1+o(1))\frac{(1+\delta)^2}{(1-\delta)^2} e^{-\frac{\epsilon^2}{5}} 
\le e^{-\frac{\epsilon^2}{10}}$, we get
\begin{equation}
(1-e^{-\frac{\epsilon^2}{10}}) 
\sum_{i > \frac{\epsilon}{4\beta}} i^2 \, \tau_\beta(i)^2
\le \frac{C_7 \, \epsilon^2 \, e^{-\frac{\epsilon^2}{10}}}{16}
\, \beta^{-4/3} + o(\beta^{-4/3}) = O(\beta^{-4/3}).
\end{equation}
Recalling the first relation in \eqref{eq:12}, we have completed the proof of (i) in \eqref{eq:i-iv}.
\end{proof}

\begin{proof}[Proof of (ii) in \eqref{eq:i-iv}]
The proof works similarly as in \cite{vdHdH95}. The only estimate we need is an upper bound on 
$h(\ell) - h(\ell+1)$ (see \cite[Eq. (3.32)]{vdHdH95}). Recall \eqref{eq:f}. Explicit calculation gives
\begin{equation}
\label{eq:deltaf}
h(\ell+1) - h(\ell) = \tfrac12\gd^2\frac{1}{(1+2\gb\ell)(1+2\gb(\ell+1))} 
- \tfrac12 \log\Big( 1 + \frac{2\gb}{1+2\gb\ell} \Big) - B \gb^{4/3}.
\end{equation}
Using that $\log(1+x) \leq x$ and \eqref{eq:regime}, we get, for some constant $c$,
\begin{equation}\label{eq:deltahell}
\begin{aligned}
h(\ell+1) - h(\ell) &\geq \frac{1}{(1+2\gb\ell)(1+2\gb(\ell+1))}\Big[ \tfrac12\gd^2 
- \gb - 2\gb^2(\ell+1) \Big]- B \gb^{4/3}\\
&\geq c\gb^{4/3} - 2\gb^2\ell.
\end{aligned}
\end{equation}
Inserting this estimate into the analogue of \cite[Eq. (3.32)]{vdHdH95}, we get
\begin{equation}\label{eq:mom.delta.tau}
\sum_{i\in\N_0} i \Delta\tau_\gb^2(i) \leq C \max\Big\{ \gb^{4/3} 
\sum_{i\in\N_0} i \tau_\gb^2(i), \, \gb^2 \sum_{i\in\N_0} i^2 \tau_\gb^2(i) \Big\}
\end{equation}
for some constant $C<\infty$. An application of Cauchy-Schwarz and (i) in \eqref{eq:i-iv} 
gives the result.
\end{proof}

\begin{proof}[Proof of (iii) in \eqref{eq:i-iv}]
The proof is the same as in \cite{vdHdH95}.
\end{proof}

\begin{proof}[Proof of (iv) in \eqref{eq:i-iv}]
Again, retracing the proof in \cite{vdHdH95}, we see that an upper bound on $h(\ell+1) - h(\ell)$ 
is needed. Recall \eqref{eq:deltaf}. Using that $\log(1+x) \geq x-x^2$ for $x$ small enough, we get
\begin{equation}
\begin{aligned}
h(\ell+1) - h(\ell) &\leq \tfrac12 \gd^2\frac{1}{(1+2\gb\ell)^2} 
-\tfrac12 \Big(\frac{2\gb}{1+2\gb\ell} - \frac{4\gb^2}{(1+2\gb\ell)^2} \Big) - B\gb^{4/3}\\
&\leq \frac{1}{(1+2\gb\ell)^2} \Big(\tfrac12\gd^2 -\gb + 2\gb^2 \Big), 
\end{aligned}
\end{equation}
which is less than a constant times $\gb^{4/3}$, by \eqref{eq:regime}.
\end{proof}

\subsection{General disorder}
\label{Generaldisorder} 

Here, the function $h$ defined in \eqref{eq:def_fh} is denoted by $h_g$ when $\go_1 \sim \cN(0,1)$. 
We use the proofs of (i)-(iv) for Gaussian disorder as a reference frame and focus on the necessary 
modifications only. Below $c$ denotes a positive and finite constant that may change from line to line.

\begin{proof}[Proof of (i) in \eqref{eq:i-iv}]
The first step is to prove \eqref{eq:f<0} for general disorder via a Taylor expansion. Since $\cH_\ell 
= -\gb \gO_\ell^2 + \gd \gO_\ell$ is bounded from above, we may write
\begin{equation}
e^{\cH_\ell} \leq 1+\cH_\ell +\frac12 \cH_\ell^2 + \frac16 \cH_\ell^3 
+ \frac{1}{24}\cH_\ell^4 + c|\cH_\ell|^5.
\end{equation}
We expand and take the expectation, keeping in mind that $\gb\ell$ will be chosen small and that 
$\gd\sqrt{\ell} \sim (2\gb\ell)^{1/2}$, by \eqref{eq:regime}. This gives
\begin{equation}
\begin{split}
&\bbE(e^{\cH_\ell}) \leq 1-\gb\bbE(\gO_\ell^2)\\
& + \tfrac12\gb^2 \bbE(\gO_\ell^4) - \gb\gd \bbE(\gO_\ell^3) + \tfrac12\gd^2 \bbE(\gO_\ell^2)\\
& -\tfrac16 \gb^3 \bbE(\gO_\ell^6) + \tfrac12\gb^2\gd \bbE(\gO_\ell^5) 
- \tfrac12 \gb\gd^2\bbE(\gO_\ell^4) + \tfrac16 \gd^3\bbE(\gO_\ell^3)\\
&+\tfrac{1}{24} \gb^4\bbE(\gO_\ell^8) - \tfrac16\gb^3\gd \bbE(\gO_\ell^7) 
+ \tfrac14\gb^2\gd^2 \bbE(\gO_\ell^6) - \tfrac16 \gb\gd^3\bbE(\gO_\ell^5) 
+ \tfrac{1}{24} \gd^4 \bbE(\gO_\ell^4)\\
& + c\Big\{ \gb^5 \bbE(\gO_\ell^{10}) + \gb^4 \gd \bbE(|\gO_\ell|^9) 
+ \gb^3 \gd^2 \bbE(\gO_\ell^8) +\gb^2 \gd^3 \bbE(|\gO_\ell|^7) + \gb \gd^4 \bbE(\gO_\ell^6) 
+ \gd^5 \bbE(|\gO_\ell|^5) \Big\}.
\end{split}
\end{equation}
By discarding all the terms that are $o(\{\gb\ell\}^2)$, we obtain
\begin{equation}
\bbE(e^{\cH_\ell}) \leq 1 + C \gb^{4/3}\ell - [1+o(1)](\gb\ell)^2 + o(\{\gb\ell\}^2), \qquad \ell\to\infty.
\end{equation}
Therefore there exists $\ell_0\in\N_0$ and $\gep\in(0,\infty)$ such that, for $\gb$ small enough,
\begin{equation}
h(\ell) = \log \bbE(e^{\cH_\ell}) - \mu\ell \leq (C-B)\gb^{4/3}\ell 
- \tfrac12 (\gb\ell)^2, \qquad \ell_0 \leq \ell\leq \gep/ \gb. 
\end{equation}

The second step is to extend \eqref{eq:f<} to general disorder. This will be a consequence of 
inequality \eqref{firstineq} in Lemma \ref{lem:Gapprox} below.

Finally, $\sup_{\ell\in\bbN} h(\ell) \leq c\gb^{2/3}$. Indeed, we have that $h(\ell)<0$ when 
$\ell >\gep/\gb$, whereas $h(\ell) \leq (C-B)\gb^{4/3}\ell - \tfrac12 (\gb\ell)^2$ when $\ell_0<\ell\leq 
\gep/\gb$ and $\sup_{z\in\bbR} \{(C-B)\gb^{1/3}z - \tfrac12 z^2\} = \tfrac12 (C-B)^2 \gb^{2/3}$. 
Moreover, the inequality is clearly satisfied for the finite set of values $\ell\in\{1,\ldots,\ell_0\}$.
\end{proof}

\begin{lemma}
\label{lem:Gapprox}
For every $\gep>0$ there exists a $C_\gep>0$ such that, for $\beta$ small enough and for 
$\beta \ell \geq \gep$,
\begin{equation}
\label{secineq}
\big|\{h(\ell+1)- h(\ell)\} - \{h_g(\ell+1) - h_g(\ell)\}\big|\leq C_\gep \beta^{3/2},
\end{equation}
and 
\begin{equation}
\label{firstineq}
|h(\ell) - h_g(\ell)| \leq \frac{C_\gep}{\sqrt{\ell}}.
\end{equation}
\end{lemma}

\begin{proof}
In the proof we use the additional assumption on the charge distribution stated in \eqref{intcond}. 
We give the proof only for case \eqref{intcond}(b), namely, $\omega_1$ is continuous with a 
density that is in $L^p$ for some $p>1$. The proof requires that the fifth moment of $\omega_1$ 
is finite (which is amply guaranteed by \eqref{Mdeltacond}) and that there exists a $\gamma 
\geq 1$ such that the function $t \mapsto |\bbE(e^{i t \omega_1})|^{\gamma}$ is integrable 
on $\R$ (which is guaranteed by \eqref{intcond}(b)). The proof for case \eqref{intcond}(a) is
analogous and is omitted. We refer to Petrov~\cite[Theorem 13, Chapter VII]{P75} for the necessary details on local 
limit theorems.

We recall \eqref{eq:def_fh} and note that in the proof we may remove the term $-\mu x$ from 
the definition of $h$ because we are only considering differences of $h$-functions. We will use 
$c_\gep$ to denote a strictly positive constant that depends on $\gep$ only and whose value 
may change from line to line. We begin with the proof of \eqref{secineq}.

To prove \eqref{secineq}, we use an Edgeworth expansion for the density $\mathfrak{m}_\ell$ 
of $\Omega_{\ell}/\sqrt{\ell}$ (see \cite[Theorem 2, Chapter XVI, Section 2]{Fe66}). Let  
$\mathfrak{n}(x)$ denote the standard normal density. Then there are three polynomials 
$Q_3,Q_4,Q_5$ such that 
\begin{equation}
\label{Edge}
{\textstyle \sup_{x\in \R} \Big| \mathfrak{m}_\ell(x)-\mathfrak{n}(x) 
\Big[1+\sum_{k=3}^5 \ell^{-(\frac{k}{2}-1)} Q_k(x)\Big]\Big|
= o\big(\ell^{-3/2}\big)}, \qquad \ell \to \infty,
\end{equation}
where, for $k\in\{3,4,5\}$,  $Q_k(X)=\sum_{s=0}^{k} \alpha_{k,s} X^s$ with $\alpha_{k,s}\in \R$ for 
each $0\leq s\leq k$. We recall \eqref{eq:regime}, and from \eqref{eq:Gstar.gaussian} deduce 
that  
\begin{equation}
\label{tryu}
e^{h_g(\ell)}=\frac{e^{\frac{\mathfrak{U_{\beta,\ell}}}{2}}}{(1+2\beta \ell)^{1/2}},
\qquad \mathfrak{U}_{\beta,\ell} = \frac{\delta^2 \ell}{1+2\beta \ell}
= \frac{1+C \beta^{1/3}}{1+1/(2\beta \ell)},\quad \ell\in \N.
\end{equation}
A similar computation allows us to write, for $s\in \N$ and $Z$ a standard normal random variable,   
\begin{equation}
\label{momgauss}
\bbE\big(Z^{s} e^{-\beta \ell Z^2+\delta \sqrt{\ell} Z}\big)
=\frac{e^{\frac{ \mathfrak{U}_{\beta,\ell}}{2}}}{(1+2\beta \ell)^{\frac{s+1}{2}}}
\sum_{j=0}^{s} \binom{s}{j}\,\bbE(Z^j)\,\big(\mathfrak{U}_{\beta,\ell}\big)^{\frac{s-j}{2}}.
\end{equation}
Combining \eqref{Edge} and \eqref{momgauss}, we get 
\begin{equation}
\label{rkpp}
e^{h(\ell)}=e^{h_g(\ell)}\, \bigg[1+\sum_{k=3}^5\sum_{s=0}^k\sum_{j=0}^{s} \tilde\alpha_{k,s,j} 
\frac{\big(\mathfrak{U}_{\beta,\ell}\big)^{\frac{s-j}{2}}}{(1+2\beta\ell)^{\frac{s}{2}}} 
\,\frac{1}{\ell^{\frac{k}{2}-1}} + o\Big(\tfrac{(1+2\beta \ell)^{1/2}}{(\gb\ell)^{1/2}\ell^{3/2}}\Big) \bigg],
\quad \ell \to \infty,\,\beta \ell \geq \gep,
\end{equation}
where 
\begin{equation}
\tilde\alpha_{k,s,j} = \binom{s}{j}\ga_{k,s,j}\,\bbE(Z^j) \in \R, 
\qquad k\in\{3,4,5\},\,0\leq s\leq k,\,0\leq j\leq s.
\end{equation}
In the remainder term in the right-hand side of \eqref{rkpp}, the factor $(\gb\ell)^{-1/2}$ comes 
from the integral $\int_\bbR e^{-\gb\ell x^2 + \gd\sqrt{\ell}x}\dd x$, while the factor 
$(1+2\beta \ell)^{1/2}$ comes from the factorisation of $e^{h_g(\ell)}$ (recall \eqref{tryu}) and 
the fact that $\mathfrak{U}_{\beta,\ell}$ is bounded from above and below by two strictly positive 
constants depending on $\gep$, uniformly in $\beta \ell \geq \gep$. The product of these 
factors is harmless when $\gb\ell\ge\gep$. Also note that the term between brackets in the 
right-hand side of \eqref{rkpp} converges to $1$ as $\ell\to\infty$, uniformly in $\gb\ell\ge \gep$. 

We next take the logarithm of \eqref{rkpp} for $\ell$ and $\ell+1$, and use the fact that $x\mapsto 
\log(1+x)$ is Lipschitz on $[-\tfrac12,\infty)$, so that there exists a $c>0$ such that 
\begin{equation}
\label{petusoli}
\begin{aligned}
&\big|h(\ell+1) - h_g(\ell+1)- h(\ell) + h_g(\ell)\big| \leq c  \sum_{k=3}^5\sum_{s=0}^k\sum_{j=0}^{s}  
\bigg |\Big[\mathfrak{T}_{k,s,j}(\beta, t)\Big]_{t=\ell}^{t=\ell+1}\bigg |
+ o\Big(\tfrac{(1+2\beta \ell)^{1/2}}{(\gb\ell)^{1/2}\ell^{3/2}}\Big),\\
&\ell \to \infty,\,\beta \ell \geq \gep,
\end{aligned}
\end{equation}
where 
\begin{equation}
\mathfrak{T}_{k,s,j}(\beta,t) 
= \frac{\big(\mathfrak{U}_{\beta,t}\big)^{\frac{s-j}{2}}}{(1+2\beta t)^{\frac{s}{2}}\,t^{\frac{k}{2}-1}}
= \frac{\gd^{s-j}t^{1+\frac{s-j-k}{2}}}{(1+2\gb t)^{s-\frac{j}{2}}}.
\end{equation}
Differentiating $t\mapsto \mathfrak{T}_{k,s,j}(\beta,t)$, we get

\begin{equation}
\label{eqinte}
\begin{aligned}
\Big|\frac{\partial}{\partial t} \mathfrak{T}_{k,s,j}(\beta,t)\Big| 
&= \frac{\gd^{s-j}t^{\frac{s-j-k}{2}}}{(1+2\gb t)^{1+s-\frac{j}{2}}} 
\Big| \Big(1+\frac{s-j-k}{2} \Big)(1+2\gb t) - (j-2s)\gb t \Big|\\
&\leq c_\gep \frac{\gb^{\frac{s-j}{2}}t^{\frac{s-j-k}{2}}}{(1+2\gb t)^{1+s-\frac{j}{2}}}
\leq c_\gep \gb^{-\frac{s}{2}}t^{-\frac{s+k}{2}}.
\end{aligned}
\end{equation}

Using \eqref{eqinte}, we have, for 
$k\in\{3,4,5\}$, $0\leq s\leq k$ and $0\leq j\leq s$,  
\begin{equation}
\label{petusol}
\bigg |\Big[\mathfrak{T}_{k,s,j}(\beta, t)\Big]_{\ell}^{\ell+1}\bigg |\leq \max_{t\in [\ell,\ell+1]}
\Big|\frac{\partial}{\partial t} \mathfrak{T}_{k,s,j}(\beta,t)\Big| 
\leq c_{\gep} \, \beta^{3/2}, \qquad \ell\geq \gep/\beta.
\end{equation}
Combining \eqref{petusoli} and \eqref{petusol}, we get \eqref{secineq}.

To prove \eqref{firstineq}, we again use the Edgeworth expansion introduced in \eqref{Edge} but 
until $k=3$ only, i.e., 
\begin{equation}
\label{Edge2}
{\textstyle \sup_{x\in \R} \Big| \mathfrak{m}_\ell(x)-\mathfrak{n}(x) 
\Big[1+ \tfrac{1}{\sqrt{\ell}} Q_3(x)\Big]\Big|
= o\Big(\tfrac{1}{\sqrt{\ell}}\Big)}, \qquad \ell \to \infty,
\end{equation}
where $Q_3(X)=\mu_3 \,\tfrac16(X^3-3X)$ with $\mu_3 = \bbE(\omega_1^3)$. We compute $e^{h(\ell)}$, 
and combine \eqref{momgauss} and \eqref{Edge2} with the observation that $\bbE(Z)=\bbE(Z^3)=0$, 
to obtain
\begin{equation}
\label{rkpp2}
e^{h(\ell)}=e^{h_g(\ell)}\, \bigg[1- \frac{\mu_3 \, (\mathfrak{U}_{\beta,\ell})^{1/2}\,}
{(1+2\beta \ell)^{3/2}}\, \beta \, \sqrt{\ell} \,
\Big( 1-\frac{\mathfrak{U}_{\beta,\ell}}{6 \beta \ell}\Big) 
+ o\Big(\tfrac{(1+2\beta \ell)^{1/2}}{\sqrt{\beta}\ell}\Big) \bigg],
\quad \ell \to \infty,\,\beta \ell \geq \gep.
\end{equation}
Since $\mathfrak{U}_{\beta,\ell}$ is bounded by strictly positive constants uniformly in $\beta \ell 
\geq \gep$, it follows from \eqref{tryu} that $\Big|1-\frac{\mathfrak{U}_{\beta,\ell}}{6 \beta \ell}\Big|$ 
is bounded for all $\beta \ell \geq \gep$. Thus, we can rewrite \eqref{rkpp2} in the form
\begin{equation}
\begin{aligned}
\label{rkpp3}
e^{h(\ell)}
&\leq e^{h_g(\ell)}\, \bigg[1+c_\gep \frac{ \beta \sqrt{\ell} }{(1+2\beta \ell)^{3/2}} 
+ o\Big(\tfrac{(1+2\beta \ell)^{1/2}}{\sqrt{\beta}\ell}\Big) \bigg]\\
&\leq e^{h_g(\ell)}\, \bigg[1+c_\gep \frac{1}{\sqrt{\beta} \ell} 
+ o\Big(\tfrac{1}{\sqrt{\ell}}\Big) \bigg], \qquad 
\ell \to \infty,\,\beta \ell \geq \gep.
\end{aligned}
\end{equation}
The reverse inequality holds with $c_\gep$ replaced by $-c_\gep$. Taking the logarithm, we conclude 
that 
\begin{equation}
\label{rkpp4}
h(\ell)\leq h_g(\ell) + \log  \bigg[1+c_\gep \frac{1}{\sqrt{\beta} \ell} 
+ o\Big(\tfrac{1}{\sqrt{\ell}}\Big) \bigg]
\leq h_g(\ell) + \frac{C_\gep}{\sqrt{\ell}}, \qquad \ell \to \infty,\,\beta \ell \geq \gep,
\end{equation}
and similarly $h(\ell) \ge h_g(\ell) - \frac{C_\gep}{\sqrt{\ell}}$, which completes the proof of 
\eqref{firstineq}.
\end{proof}

\begin{proof}[Proof of (ii) in \eqref{eq:i-iv}]
We need an upper bound on $1 - e^{h(\ell+1)-h(\ell)}$ (see \cite[(3.27)]{vdHdH95}). Write
\begin{equation}
1 - e^{h(\ell+1)-h(\ell)} = 1 - \frac{\bbE(e^{-\gb\gO_{\ell+1}^2 
+ \gd \gO_{\ell+1}})}{\bbE(e^{-\gb\gO_{\ell}^2 + \gd \gO_{\ell}})}.
\end{equation}

\medskip
\paragraph{Case $\ell \leq \gep/\gb$.} 
Define
\begin{equation}
\cG(t) = \frac{\bbE\Big(e^{-\gb(\gO_\ell + t\go_{\ell+1})^2+\gd(\gO_\ell + t\go_{\ell+1})}\Big)}
{\bbE\Big(e^{-\gb\gO_\ell^2+\gd\gO_\ell}\Big)} - 1, \qquad 0\leq t \leq 1,
\end{equation}
which is a function interpolating between $\cG(0) = 0$ and $\cG(1) = e^{h(\ell+1)-h(\ell)}-1$. 
Rewrite
\begin{equation}
\cG(t) = \tilde\bbE (e^{H_t}) - 1,\qquad H_t 
= -2\gb t \gO_\ell \go_{\ell+1} + \gd t \go_{\ell+1} - \gb t^2 \go_{\ell+1}^2,
\end{equation}
where
\begin{equation}
\tilde\bbP(\cdot) = \frac{\bbE\big[\, (\cdot) \, e^{-\gb\gO_\ell^2+\gd\gO_\ell}\big]}
{\bbE\big[e^{-\gb\gO_\ell^2+\gd\gO_\ell}\big]}.
\end{equation}
Note that (the derivatives are w.r.t.\ the parameter $t$)
\begin{equation}
H'_t = -2\gb\gO_\ell\go_{\ell+1} + \gd\go_{\ell+1} - 2\gb t \go_{\ell+1}^2, \quad H''_t 
= -2\gb\go_{\ell+1}^2,\quad H^{(k)}_t = 0, \qquad k \geq 3.
\end{equation}
A Taylor expansion to fifth order gives
\begin{equation}
\label{eq:cG1}
\cG(1) = \cG'(0) + \tfrac12 \cG''(0) + \tfrac16 \cG^{(3)}(0) + \tfrac{1}{24}\cG^{(4)}(0) 
+ \int_0^1 \frac{(1-t)^4}{24} \cG^{(5)}(t)\dd t.
\end{equation}
Moreover,
\begin{equation}
\begin{aligned}
&\cG'(t) = \tilde\bbE\Big( H'_t e^{H_t} \Big),\quad
\cG''(t) = \tilde\bbE\Big( [H''_t + H'^2_t] e^{H_t} \Big),\quad
\cG^{(3)}(t) = \tilde\bbE\Big( [3H'_tH''_t + H'^3_t] e^{H_t} \Big),\\
&\cG^{(4)}(t) = \tilde\bbE\Big( [3H''^2_t + 6 H'^2_t H''_t + H'^4_t] e^{H_t} \Big),\quad
\cG^{(5)}(t) = \tilde\bbE\Big( [15 H'_t H''^2_t +10 H'^3_t H''_t + H'^5_t] e^{H_t} \Big)
\end{aligned}
\end{equation}
and
\begin{equation}
\begin{aligned}
&\cG'(0) = 0,\qquad
\cG''(0) = \gd^2 - 2\gb + 4\gb^2 \tilde\bbE(\gO_\ell^2) - 4\gb\gd \tilde\bbE(\gO_\ell), \\
&\cG^{(3)}(0) = \bbE(\go_0^3) \Big\{ 12\gb^2\tilde\bbE(\gO_\ell) - 6 \gb\gd 
- 8\gb^3 \tilde\bbE(\gO_\ell^3) + 12 \gd\gb^2 \tilde\bbE(\gO_\ell^2) 
- 6 \gb\gd^2 \tilde\bbE(\gO_\ell) + \gd^3  \Big\},\\
&\cG^{(4)}(0) 
= \bbE(\go_0^4)\left\{ \begin{array}{l}
12\gb^2 -48 \gb^3 \tilde\bbE(\gO_\ell^2) + 48 \gb^2 \gd \tilde\bbE(\gO_\ell) - 12\gb\gd^2 
+ 16 \gb^4 \tilde\bbE(\gO_\ell^4)\\ 
-32 \gb^3\gd \tilde\bbE(\gO_\ell^3) + 24 \gb^2 \gd^2 \tilde\bbE(\gO_\ell^2) 
- 8 \gb \gd^3 \tilde\bbE(\gO_\ell) + \gd^4.
\end{array}
\right\}.
\end{aligned}
\end{equation}

The terms arising from $\cG^{(5)}$ in \eqref{eq:cG1} can be shown to be of negligible order. Therefore we obtain, for $\beta$ small enough,
\begin{equation}
\label{eq:G1geq}
\cG(1) \geq C\gb^{4/3} - c\big[\gb^{3/2}\tilde\bbE(\gO_\ell) + \gb^3\tilde\bbE(\gO_\ell^3) 
+ \gb^2\tilde\bbE(\gO_\ell^2)\big].
\end{equation}

It remains to provide upper bounds on $\tilde\bbE(\gO_\ell^k)$ for $k\in\{1,2,3\}$. To that end,
let us write
\begin{equation}
\tilde\bbE(\gO_\ell) = \frac{\bbE\Big( \gO_\ell e^{-\gb \gO_\ell^2 + \gd\gO_\ell} \Big)}
{\bbE\Big(e^{-\gb \gO_\ell^2 + \gd\gO_\ell} \Big)}.
\end{equation}
On the one hand, by a first order expansion of the exponential we get
\begin{equation}
\bbE\Big( \gO_\ell e^{-\gb \gO_\ell^2 + \gd\gO_\ell} \Big) \leq \gd\ell 
+ c[\gb^2\ell^{5/2} + \gb\ell^{3/2}] \leq \sqrt{2} \gb^{1/2}\ell (1+c\gep^{3/2} + c\gep^{1/2}).
\end{equation}
On the other hand, since $\gb\ell\leq \gep$ we get
\begin{equation}
\bbE\Big(e^{-\gb \gO_\ell^2 + \gd\gO_\ell} \Big) \geq 1-\gep.
\end{equation}
Therefore
\begin{equation}
\label{eq:tilde1}
\tilde\bbE(\gO_\ell) \leq c\gb^{1/2}\ell.
\end{equation}
With the same method of proof we obtain
\begin{equation}
\label{eq:tilde23}
\tilde\bbE(\gO_\ell^2) \leq c\ell, \qquad \tilde\bbE(\gO_\ell^3) \leq c (\ell + \gb^{1/2}\ell^2).
\end{equation}
Inserting \eqref{eq:tilde1} and \eqref{eq:tilde23} into \eqref{eq:G1geq}, we arrive at
\begin{equation}
\cG(1) \geq C\gb^{4/3} - c\gb^2\ell,
\end{equation}
which is the desired estimate, cf.\ \eqref{eq:deltahell}.

\medskip
\paragraph{Case $\ell > \gep/\beta$.} 
Decompose
\begin{equation}
\begin{aligned}
1 - e^{h(\ell+1)-h(\ell)} &\leq h(\ell) - h(\ell+1)\\
&=  \{h_g(\ell) - h_g(\ell+1)\} + \{h(\ell) - h_g(\ell)\} - \{h(\ell+1) - h_g(\ell+1)\}.
\end{aligned}
\end{equation}
We already know that $h_g(\ell) - h_g(\ell+1) \leq c\gb^2\ell$. Moreover, by Lemma \ref{lem:Gapprox},
\begin{equation}
\{h(\ell) - h_g(\ell)\} - \{h(\ell+1) - h_g(\ell+1)\} \leq c \beta^{3/2},
\end{equation}
from which we get
\begin{equation}
1 - e^{h(\ell+1)-h(\ell)} \leq c \gb^2\ell \Big(1 + \frac{1}{\sqrt{\beta} \ell} \Big) 
\leq c\gb^2 \ell (1+\gep^{-1}\gb^{1/2}),
\end{equation}
which is the desired estimate, cf.\ \eqref{eq:mom.delta.tau}.
\end{proof}

\begin{proof}[Proof of (iii) in \eqref{eq:i-iv}]
The proof is the same as in \cite{vdHdH95}.
\end{proof}

\begin{proof}[Proof of (iv) in \eqref{eq:i-iv}]
It is enough to prove that $1- e^{h(\ell)-h(\ell+1)} \leq c\gb^{4/3}$ (see \cite[(3.35)]{vdHdH95}). 
In the analogue of \cite[(3.35)]{vdHdH95}, we decompose the sum in two parts: (a) $\ell = i+j 
\geq \gep/\beta$; (b) $\ell = i+j \leq \gep/\beta$. 

For (a) we use a Gaussian approximation. Write
\begin{equation}
1- e^{h(\ell)-h(\ell+1)} \leq h(\ell+1) - h(\ell) \leq h_g(\ell+1) - h_g(\ell) 
+ \{h(\ell+1) - h_g(\ell+1)\} - \{h(\ell) - h_g(\ell)\}.
\end{equation}
We already know that $h_g(\ell+1) - h_g(\ell) \leq c\gb^{4/3}$. Moreover, by Lemma \ref{lem:Gapprox},
\begin{equation}
\{h_g(\ell+1) - h(\ell+1)\} - \{h(\ell) - h_g(\ell)\} \leq c \beta ^{3/2}
\ll \gb^{4/3},
\end{equation}
which is the desired estimate.

For (b) we write
\begin{equation}
1- e^{h(\ell)-h(\ell+1)} \leq h(\ell+1) - h(\ell) \leq e^{h(\ell+1) - h(\ell)} -1 
= \frac{\bbE(e^{\cH_{\ell+1}})}{\bbE(e^{\cH_\ell})} - 1,
\end{equation}
where $\cH_\ell = -\gb \gO_\ell^2 + \gd  \gO_\ell$. Using
\begin{equation}
1 + \cH_\ell + \tfrac12 \cH_\ell^2 + \tfrac16 \cH_\ell^3 \leq e^{\cH_\ell} 
\leq 1 + \cH_\ell + \tfrac12 \cH_\ell^2 + c|\cH_\ell|^3,
\end{equation}
we get
\begin{equation}
\frac{\bbE(e^{\cH_{\ell+1}})}{\bbE(e^{\cH_\ell})} -1 
\leq \frac{1+C\gb^{4/3}(\ell+1) + c\gb^2\ell^2}{1+C\gb^{4/3}\ell - c\gb^{3/2}\ell - c\gb^{7/3}\ell^2} -1 
\leq C\gb^{4/3} + c\gb^{8/3}\ell^2.
\end{equation}
Inserting this estimate into the analogue of \cite[(3.35)]{vdHdH95}, and noting that
\begin{equation}
\gb^{8/3}\sum_{i\in\N_0} i^2 \tau_\gb(i)^2 \leq c\gb^{4/3},
\end{equation}
which we know from (i) in \eqref{eq:i-iv}, we get the claim.
\end{proof}


\section{Quenched model}
\label{AppD}

As promised in Section~\ref{ss:motivation}, in this appendix we prove two modest results 
for the quenched version of the model, which has path measure $\bP_n^{\omega,\beta}$
defined by (recall (\ref{eq:def.polmeasure}--\ref{eq:def.hamiltonian}))
\begin{equation}
\frac{\dd\bP_n^{\omega,\beta}}{\dd\bP}(S) = \frac{1}{Z_n^{\omega,\beta}}\,
e^{-\beta H_n^\omega(S)}, \qquad S \in \Pi,
\end{equation}
where $\Pi$ is the set of nearest-neighbour paths starting at $0$ and $Z_n^{\omega,\beta}$ 
is the \emph{quenched partition function of length} $n$. 

Recall \eqref{eq:deflocaltimes}. The range of $S$ up to time $n$ is 
\begin{equation}
R_n(S) = |\{x\in \Z \colon\, L_n(S,x)>0\}|.
\end{equation}
We first show that $R_n(S)$ grows linearly in $n$ when the average charge is non-zero.

\begin{proposition}
\label{pr:1}
Suppose that $\delta,\beta \in (0,\infty)$. Then there exist $c_1, c_2>0$ 
(depending on $\delta,\beta$) such that, for $\bbP^\delta$-a.e.\ $\omega$,
\begin{equation}
\label{eq:1}
\bP_n^{\omega,\beta}(R_n(S) \leq c_1 n) \leq e^{-c_2 n+o(n)}.
\end{equation}
\end{proposition}

\begin{proof}
Let $\pi$ be the one-sided path that takes right-steps only, i.e., $\pi_i = i$ for $i\in\N_0$. 
Recall \eqref{eq:omegacond} and estimate
\begin{equation}
\label{eq:2}
Z_n^{\omega,\beta} \geq (\tfrac{1}{2})^n\, \bE\Big[e^{-\beta H_n^{\omega,\beta}(S)}\,
\ind_{\{S_i=\pi_i\,\forall\,1\leq i\leq n\}}\Big]
= (\tfrac{1}{2})^n\,e^{-\beta\sum_{i=1}^n \omega_i^2}
= (\tfrac{1}{2})^n\,e^{-\beta n + o(n)}.
\end{equation}
Moreover, by Jensen's inequality we have (recall \eqref{eq:def.hamiltonian})
\begin{equation}
\begin{aligned}
H_n^\omega(S) 
&= \sum_{ {x\in\Z\colon} \atop {L_n(S,x)>0} } 
\left(\sum_{i=1}^n \omega_i \ind_{\{S_i=x\}}\right)^2
= R_n(S) \left[\frac{1}{R_n(S)} \sum_{ {x\in\Z\colon} \atop {L_n(S,x)>0}} 
\left(\sum_{i=1}^n \omega_i \ind_{\{S_i=x\}}\right)^2\right]\\
\label{eq:3}
&\geq R_n(S) \left(\frac{1}{R_n(S)} \sum_{ {x\in\Z\colon} \atop {L_n(S,x)>0}}    
\sum_{i=1}^n \omega_i \ind_{\{S_i=x\}}\right)^2
= \frac{1}{R_n(S)}\,\Omega_n^2.
\end{aligned}
\end{equation}
Combining (\ref{eq:2}--\ref{eq:3}), we obtain 
\begin{equation}
\label{eq:4}
\begin{aligned}
\bP_n^{\omega,\beta}(R_n(S) \leq c_1n)
&\leq e^{\beta n}\,2^n\, \bE\left[\exp\left\{-\frac{\beta}{R_n(S)}\,\Omega_n^2\right\}\, 
\ind_{\{R_n(S) \leq c_1n\}}\right]\\
&\leq \exp\left\{-\beta n \left[\frac{1}{c_1n^2}\,\Omega_n^2 -1-\tfrac{\log 2}{\beta}\right]\right\}.
\end{aligned}
\end{equation}
By the strong law of large numbers for $\omega$, we have $\lim_{n\to \infty} n^{-1} \Omega_n
= (\partial/\partial\delta)\log M(\delta) = m(\delta)>0$ for $\bbP^\delta$-a.a.\ $\omega$, and so 
the term between square brackets equals $c_3[1+o(1)]$ with $c_3=\tfrac{1}{c_1} m(\delta)^2-1
-\tfrac{\log 2}{\beta}$. Therefore, by choosing $c_1>0$ small enough so that $c_3>0$, we get 
(\ref{eq:1}) with $c_2=\beta c_3$.
\end{proof}

We next show that the polymer chain is ballistic when the charges are sufficiently biased.

\begin{proposition}
\label{pr:2}
For every $\beta \in (0,\infty)$ there exists a $\delta_0=\delta_0(\beta) \in (0,\infty)$ 
such that 
\begin{equation}
\forall\,\delta>\delta_0\,\,\exists\,\epsilon=\epsilon(\delta)>0\colon \qquad 
\lim_{n\to\infty} \bP_n^{\omega,\beta}\big(n^{-1}S_n>\epsilon \mid S_n>0\big) = 1.
\end{equation}
\end{proposition}

\begin{proof}
Fix $\beta \in (0,\infty)$. Pick $\delta_0$ such that
\begin{equation}
m(\delta_0) = \sqrt{\tfrac{1}{2}\left(1 + \tfrac{\log 2}{\beta} \right)}.
\end{equation}
If $\delta>\delta_0$, then we can choose $c_1 > \tfrac12$ in Proposition~{\rm \ref{pr:1}} 
and use the inequality
\begin{equation}
\frac{\left|\{x\in \bbZ\colon\, L_n(S,x) = 1\}\right|}{n} \geq \frac{2R_n(S)}{n} -1
\end{equation}
to conclude that a positive fraction of the sites is visited precisely once. Consequently, if the 
polymer chain chooses to go to the right, then $S_n/n$ has a strictly positive $\liminf$.
\end{proof}



\end{document}